\documentclass[12pt]{amsart}
\usepackage{amssymb}
\usepackage{amsfonts,graphicx}
\usepackage{hyperref}  

\usepackage{mathtools}
\mathtoolsset{showonlyrefs,showmanualtags} 
 
\usepackage{pgf,tikz}
\usetikzlibrary{shapes.geometric}
\usetikzlibrary{trees}
\usepgflibrary{arrows}

\theoremstyle{plain}
\newtheorem{theorem}[equation]{Theorem}

\newtheorem{definition}[equation]{Definition}

\newtheorem{lemma}[equation]{Lemma}
\newtheorem{notation}[equation]{Notation}

\newtheorem{proposition}[equation]{Proposition}
\newtheorem{remark}[equation]{Remark}

\numberwithin{equation}{section}

\allowdisplaybreaks
\swapnumbers

\hypersetup{
    colorlinks=true,           linkcolor=blue,              citecolor=magenta,            filecolor=magenta,          urlcolor=cyan           }

\begin{document}
\title[Two weight norm inequality]{A Two Weight Inequality for the Hilbert transform
 Assuming  an Energy Hypothesis}
\author[M.T. Lacey]{Michael T. Lacey}
\address{School of Mathematics \\
Georgia Institute of Technology \\
Atlanta GA 30332 }
\email{lacey@math.gatech.edu}
\thanks{Research supported in part by the NSF grant 0456611}
\author[E.T. Sawyer]{Eric T. Sawyer}
\address{ Department of Mathematics \& Statistics, McMaster University, 1280
Main Street West, Hamilton, Ontario, Canada L8S 4K1 }
\email{sawyer@mcmaster.ca}
\thanks{Research supported in part by NSERC}
\author[I. Uriarte-Tuero]{Ignacio Uriarte-Tuero}
\address{ Department of Mathematics \\
Michigan State University \\
East Lansing MI }
\email{ignacio@math.msu.edu}
\thanks{Research supported in part by the NSF, through grant DMS-0901524.}
\date{}

\begin{abstract}
Let $\sigma $ and $\omega $ be locally finite positive Borel measures on $%
\mathbb{R}$. Subject to the pair of weights satisfying a side condition, we
characterize boundedness of the Hilbert transform $H$ from $L^{2}\left(
\sigma \right) $ to $L^{2}\left( \omega \right) $ in terms of the $A_{2}$
condition%
\begin{equation*}
\left[ \int_{I}\left( \frac{\left\vert I\right\vert }{\left\vert
I\right\vert +\left\vert x-x_{I}\right\vert }\right) ^{2}d\omega \left(
x\right) \int_{I}\left( \frac{\left\vert I\right\vert }{\left\vert
I\right\vert +\left\vert x-x_{I}\right\vert }\right) ^{2}d\sigma \left(
x\right) \right] ^{\frac{1}{2}}\leq C\left\vert I\right\vert ,
\end{equation*}%
and the two testing conditions: For all intervals $I$ in $\mathbb{R}$
\begin{eqnarray*}
\int_{I}H\left( \mathbf{1} _{I}\sigma \right) (x)^{2}d\omega (x) &\leq
&C\int_{I}d\sigma (x), \\
\int_{I}H\left( \mathbf{1} _{I}\omega \right) (x)^{2}d\sigma (x) &\leq
&C\int_{I}d\omega (x),
\end{eqnarray*}%
 The proof uses the beautiful corona
argument of Nazarov, Treil and Volberg. There is a range of side conditions,
termed Energy conditions; at one endpoint, the Energy conditions are also a
consequence of the testing conditions above, and at the other endpoint they are 
the Pivotal Conditions of Nazarov, Treil and Volberg.  We detail an example 
which shows that  the Pivotal Conditions are not 
necessary for boundedness of the Hilbert
transform.
\end{abstract}

\maketitle
\setcounter{tocdepth}{1}
\tableofcontents

\def\baselinestretch{1.25}
\section{Introduction}

We provide sufficient conditions for the two weight inequality for the
Hilbert transform. Indeed, subject to a side condition, a characterization
of the two weight $L ^2 $ inequality is given.

For a signed measure $\omega $ on $\mathbb{R}$ define 
\begin{equation}
H\omega \left( x\right) \equiv \textup{p.v.}\int \frac{1}{x-y}\;\omega
(dy)\,.  \label{e.H}
\end{equation}%
A \emph{weight} $\omega $ is a non-negative locally finite measure. For two
weights $\omega ,\sigma $, we are interested in the inequality 
\begin{equation}
\left\Vert H(\sigma f)\right\Vert _{L^{2}(\omega )}\lesssim \left\Vert
f\right\Vert _{L^{2}(\sigma )}.  \label{e.H<}
\end{equation}
See Definition \ref{def inequality} below for a precise definition of $p.v.$
and the meaning of \eqref{e.H<}. The two weight problem for the Hilbert
transform is to provide a real variable characterization of the pair of
weights $\omega ,\sigma $ for which inequality \eqref{e.H<} holds.

What should such a characterization look like? Motivated by the very
successful $A_{p}$ theory for the Hilbert transform in Hunt, Muckenhoupt and
Wheeden \cite{HuMuWh}, one suspects that the weights should satisfy the two
weight analog of the $A_{2}$ condition:%
\begin{equation*}
\sup_{I} \frac{1}{\left\vert I\right\vert }\int_{I}\omega (dx) \cdot \frac{1%
}{\left\vert I\right\vert }\int_{I}\sigma (dx) <\infty \,.
\end{equation*}

As it turns out this two weight $A_{2}$ condition is not sufficient. The
suggestion for additional necessary conditions comes from the $T1$ theorem
of David and Journ\'{e} `{MR763911} and the two weight theorems of the
second author for fractional integral operators, \cite{MR930072}. These
conditions require the following, holding uniformly over intervals $I$:%
\begin{align}
\int_{I}\lvert H(\mathbf{1}_{I}\sigma )\rvert ^{2}\;\omega (dx) & \leq 
\mathcal{H}^{2}\sigma (I)\,,  \label{e.H1} \\
\int_{I}\lvert H(\mathbf{1}_{I}\omega )\rvert ^{2}\;\sigma (dx) & \leq (%
\mathcal{H}^{\ast })^{2}\omega (I)\,.  \label{e.H2}
\end{align}%
Here, we are letting $\mathcal{H}$ and $\mathcal{H}^{\ast }$ denote the
smallest constants for which these inequalities are true uniformly over all
intervals $I$, and we write $\sigma (I)\equiv \int_{I}\sigma (dx)$.

Clearly, \eqref{e.H1} is derived from applying the inequality \eqref{e.H<}
to indicators of intervals. One advantage of formulating the inequality %
\eqref{e.H<} with the measure $\sigma $ on both sides of the inequality is
that duality is then easy to derive: Interchange the roles of $\omega $ and $%
\sigma $. Thus, the condition \eqref{e.H2} is also derived from \eqref{e.H<}%
. We call these `testing conditions' as they are derived from simple
instances of the claimed inequality.  Also, we emphasize that duality in this sense 
is basic to the subject, and we will appeal to it repeatedly. 

In a beautiful series of papers, Nazarov, Treil and Volberg have developed a
sophisticated approach toward proving the sufficiency of these testing
conditions combined with an improvement of the two weight $A_{2}$ condition.
To describe this improvement, we define this variant of the Poisson integral
for use throughout this paper. For an interval $I$ and measure $\omega $,%
\begin{gather}
\mathsf{P}(I,\omega )\equiv \int_{\mathbb{R}}\frac{\lvert I\rvert }{(\lvert
I\rvert +\textup{dist}(x,I))^{2}}\;\omega (dx)  \label{e.P} \\
\sup_{I}\mathsf{P}(I,\omega )\cdot \mathsf{P}(I,\sigma )=\mathcal{A}%
_{2}^{2}<\infty \,.  \label{A2}
\end{gather}%
The last line is the improved condition of Nazarov, Treil and Volberg. We
will refer to \eqref{A2} as simply the $A_{2}$ condition. F. Nazarov has
shown that even this strengthened $A_{2}$ condition is not sufficient for
the two weight inequality \eqref{e.H<} - see e.g. Theorem 2.1 in \cite{NiTr}.

The approach of Nazarov, Treil and Volberg involves a delicate combination
of ideas: random grids (see \cite{MR1998349}), weighted Haar functions and
Carleson embeddings (see \cite{MR2407233}), stopping intervals (see \cite%
{Vol}) and culminates in the use of these techniques with a corona
decomposition in the brilliant 2004 preprint \cite{NTV3}. Theorem 2.2 of that 
paper proves 
the sufficiency of conditions \eqref{A2}, \eqref{e.H1} and \eqref{e.H2} for
the two weight inequality \eqref{e.H<} in the presence of two additional
side conditions, the \emph{Pivotal Conditions} given by 
\begin{equation}
\sum_{r=1}^{\infty }\omega (I_{r})\mathsf{P}(I_{r},\mathbf{1} _{I_{0}}\sigma
)^{2}\leq \mathcal{P}^{2}\sigma (I_{0}),  \label{pivotalcondition}
\end{equation}%
(and its dual) where the inequality is required to hold for all intervals $%
I_{0}$, and decompositions $\{I_{r}\;:\;r\geq 1\}$ of $I_{0}$ into disjoint
intervals $I_{r}\subsetneq I_{0}$. As a result they obtain the
equivalence of \eqref{e.H<} with the three conditions \eqref{A2}, %
\eqref{e.H1} and \eqref{e.H2} when both weights are doubling,\ and also when
two maximal inequalities hold.   Our Theorem below contains this result as a special case.  

In our approach, we replace the Pivotal Condition \eqref{pivotalcondition} by certain
weaker side conditions of \emph{energy type}. We begin with a necessary form
of energy.

\begin{definition}
\label{d.energy} For a weight $\omega $, and interval $I$, we set 
\begin{equation*}
\mathsf{E}\left( I,\omega \right) \equiv \left[ \mathbb{E}_{I}^{\omega (dx)}%
\left[ \mathbb{E}_{I}^{\omega (dx^{\prime })}\frac{x-x^{\prime }}{\lvert
I\rvert }\right] ^{2}\right] ^{1/2}.
\end{equation*}
\end{definition}

It is important to note that $\mathsf{E}(I,\omega )\leq 1$, and can be quite
small, if $\omega $ is highly concentrated inside the interval $I$; in
particular if $\omega \mathbf{1}_{I}$ is a point mass, then $\mathsf{E}%
(I,\omega )=0$. Note also that $\omega (I)\left\vert I\right\vert ^{2}%
\mathsf{E}(I,\omega )^{2}$ is the variance of the variable $x$, and that we
have the identity%
\begin{eqnarray*}
\mathsf{E}\left( I,\omega \right) ^{2} &=&
\frac{1}{2}\mathbb{E}_{I}^{\omega (dx)}\mathbb{E}_{I}^{\omega (dx^{\prime
})} \frac{\left( x-x^{\prime }\right) ^{2}}{\left\vert I\right\vert
^{2}}\,. 
\end{eqnarray*}

The following \emph{Energy Condition} is necessary for the two weight
inequality:%
\begin{equation}
\sum_{r\geq 1}\omega (I_{r})\mathsf{E}(I_{r},\omega )^{2}\mathsf{P}%
(I_{r},\sigma \mathbf{1}_{I_{0}})^{2}\leq \mathcal{E}^{2}\sigma (I_{0}),
\label{energy condition}
\end{equation}%
where the sum is taken over all decompositions $I_{0}=\bigcup_{r=1}^{\infty
}I_{r}$ of the interval $I_{0}$ into pairwise disjoint intervals $\left\{
I_{r}\right\} _{r\geq 1}$. As $\mathsf{E}(I,\omega )\leq 1$, the Energy
Condition is weaker than the Pivotal Condition.

As a preliminary sufficient side condition, we consider the \emph{geometric
mean} of the pivotal and energy conditions: for $0\leq \epsilon \leq 2$ we
say that the weight pair $\left( \omega ,\sigma \right) $ satisfies the 
\emph{Hybrid Energy Condition} or simply \emph{Hybrid Condition} provided%
\begin{equation}
\sum_{r\geq 1}\omega (I_{r})\mathsf{E}(I_{r},\omega )^{\epsilon }\mathsf{P}%
(I_{r},\sigma \mathbf{1}_{I_{0}})^{2}\leq \mathcal{E}_{\epsilon }^{2}\sigma
(I_{0}),  \label{Hy Con}
\end{equation}%
where the sum is taken over all decompositions $I_{0}=\bigcup_{r=1}^{\infty
}I_{r}$. When $\epsilon =2$ this is the necessary Energy Condition and when $%
\epsilon =0$ this is the Pivotal Condition. A corollary of our main theorem
is that if the weight pair $\left( \omega ,\sigma \right) $ satisfies the
Hybrid Condition \eqref{Hy Con} and its dual for some $\epsilon <2$, then
the two weight inequality \eqref{e.H<} is equivalent to the $A_{2}$
condition \eqref{A2} and the testing conditions \eqref{e.H1} and (\ref{e.H2}%
).

Later in this paper we exhibit a weight pair $\left( \omega ,\sigma \right) $
satisfying \eqref{e.H1}, \eqref{e.H2}, \eqref{A2} and the Hybrid Conditions
for some $\epsilon <2$, but for which the dual Pivotal Condition \emph{fails}%
. In particular this shows that the Pivotal Conditions are \emph{not}
necessary for the two weight inequality \eqref{e.H<}.

\subsection{An optimal condition}

Now we describe an optimal---for the method of proof---sufficient side
condition. First it is convenient to introduce two functionals of pairs of
sets that arise.

\begin{definition}
\label{funct}Fix $0\leq \epsilon <2$. We define the functionals%
\begin{eqnarray}
\Phi \left( J,E\right) &\equiv &\omega \left( J\right) \mathsf{E}\left(
J,\omega \right) ^{2}\mathsf{P}\left( J,\mathbf{1}_{E}\sigma \right) ^{2},
\label{def functionals} \\
\Psi \left( J,E\right) &\equiv &\omega \left( J\right) \mathsf{E}\left(
J,\omega \right) ^{\epsilon }\mathsf{P}\left( J,\mathbf{1}_{E}\sigma \right)
^{2}.  \notag
\end{eqnarray}
\end{definition}

Note that $\Phi \left( I_{r},I_{0}\right) $ appears in the sum on the left
side of the Energy Condition \eqref{energy condition}, and $\Phi \left( J,%
\widehat{I}\setminus I^{\prime }\right) $ appears again on the right side of
the dual Energy Estimate \eqref{e.Edual} below. The larger functional $\Psi
\left( I_{r},I_{0}\right) $ appears in the sum on the left side of the
Hybrid Condition \eqref{Hy Con}.

It turns out that one can replace $\Psi $ in the proof below with any
functional, subject to three properties holding. We now describe the three
properties required of the functional $\Psi $.

Set $e(I)\equiv \left\{ a,b,\frac{a+b}{2}\right\} $ to be the set consisting
of the endpoints and midpoint of an interval $I=\left[ a,b\right] $. We say
that a subpartition $\left\{ J_{r}\right\} $ of $I$ is $\varepsilon $\emph{%
-good} if%
\begin{equation}
\operatorname{dist}(J_{r},e(I))>\tfrac{1}{2}\lvert J_{r}\rvert ^{\varepsilon }\lvert
I\rvert ^{1-\varepsilon }.  \label{good property}
\end{equation}%
For $\gamma >0$ and $\varepsilon >0$, and for all pairs of intervals $%
I_{0}\subset \widehat{I}$ in $\mathcal{D}^{\sigma }$ we require 
\begin{equation} \label{Psi properties}
\phantom{.}\ 
\begin{cases}
\Psi \left( I_{0},I_{0}\right) \leq \mathcal{F}_{\gamma ,\varepsilon
}^{2}\sigma (I_{0}),   \\ \\
\sum_{r\geq 1}\Psi \left( I_{r},I_{0}\right) \leq \mathcal{F}_{\gamma
,\varepsilon }^{2}\sigma (I_{0}), \qquad   \text{for all subpartitions }%
\left\{ I_{r}\right\} \text{ of }I_{0}\ , \\ \\
\sum_{r\geq 1}\Phi \left( J_{r},\widehat{I}\setminus I_{0}\right) \leq
\sup_{r\geq 1}\left( \frac{\left\vert J_{r}\right\vert }{\left\vert
I_{0}\right\vert }\right) ^{\gamma }\Psi \left( I_{0},\widehat{I}\right)  \\ 
\phantom{\sum_{r\geq 1}\Psi \left( I_{r},I_{0}\right) \leq \mathcal{F}_{\gamma
,\varepsilon }^{2}\sigma (I_{0}),} \qquad 
\text{for all }\varepsilon \text{-good subpartitions }%
\left\{ J_{r}\right\} \text{ of }I_{0}\ .%
\end{cases}
\end{equation}

\begin{description}
\item[Note] It is important to note that the second line requires us to test
over \emph{all} subpartitions $\left\{ I_{r}\right\} $ of $I_{0}$. In the
third line we need only test over the $\varepsilon $-good subpartitions, but
must include \emph{differences} $\widehat{I}\setminus I_{0}$ of intervals in
the argument of $ \Phi$ on the left side.
\end{description}

When $\Psi $ is given by \eqref{def functionals}, the first line in \eqref{Psi properties} 
is the usual $A_{2}$ condition, the second line is the
Hybrid Condition \eqref{Hy Con} with $\epsilon =\gamma $, and the third line
is proved in Lemma \ref{hybrid implies hypothesis} below.

For fixed $\gamma ,\varepsilon >0$, there is a \emph{smallest} functional $%
\Psi _{\gamma ,\varepsilon }$ satisfying the third line in \eqref{Psi
properties}, namely%
\begin{equation}
\Psi _{\gamma ,\varepsilon }\left( I,E\right) \equiv \sup_{I\supset
\bigcup_{s\geq 1}J_{s}}\left[ \inf_{s\geq 1}\left( \frac{\left\vert
I\right\vert }{\left\vert J_{s}\right\vert }\right) ^{\gamma }\right]
\sum_{s\geq 1}\Phi \left( J_{s},E\right) ,  \label{def Psi gamma}
\end{equation}%
where the supremum is taken over all $\varepsilon $-good subpartitions $%
\left\{ J_{s}\right\} $ of the interval $I$. Note that $\Psi _{\gamma
,\varepsilon }\left( I,E\right) $ becomes smaller as either $\gamma $ or $%
\varepsilon $ becomes smaller, and also as $E$ becomes smaller. The
functional $\Psi _{\gamma ,\varepsilon }$ also satisfies the first line as
we see by taking $E=I_{0}$ and the trivial decomposition $I_{1}=I_{0}$. Then
the second line in \eqref{Psi properties} becomes%
\begin{equation}
\sum_{r\geq 1}\Psi _{\gamma ,\varepsilon }\left( I_{r},I_{0}\right) \leq 
\mathcal{F}_{\gamma ,\varepsilon }^{2}\sigma (I_{0}),\ \ \ \ \ \text{for 
\emph{all} subpartitions }\left\{ I_{r}\right\} \text{ of }I_{0}\ .
\label{smallest condition}
\end{equation}%
This condition \eqref{smallest condition}, which we call the \emph{Energy
Hypothesis}, thus represents the optimal side condition that can be used,
along with its dual version, with the methods of this paper (all three lines
in \eqref{Psi properties} hold and the third line is optimal). From Lemma %
\ref{hybrid implies hypothesis} and the optimal property of $\Psi _{\gamma
,\varepsilon }$, we see that the Hybrid Condition \eqref{Hy Con} implies the
Energy Hypothesis \eqref{smallest condition} with $\gamma =2-\epsilon
-2\varepsilon >0$.

\begin{theorem}
\label{main}Suppose that $\omega $ and $\sigma $ are locally finite positive
Borel measures on the real line having no point masses in common, namely $%
\omega \left( \left\{ x\right\} \right) \sigma \left( \left\{ x\right\}
\right) =0$ for all $x\in \mathbb{R}$. Suppose in addition that for some $%
\gamma  >0$, and $ 0 < \varepsilon<1$ we have both Energy Hypothesis constants $\mathcal{F}_{\gamma
,\varepsilon }$ and $\mathcal{F}_{\gamma ,\varepsilon }^{\ast }$ finite.
Then the two weight inequality \eqref{e.H<} holds \emph{if and only if}

\begin{itemize}
\item the pair of weights satisfies the $A_{2}$ condition \eqref{A2};

\item the testing conditions \eqref{e.H1} and \eqref{e.H2} both hold.
\end{itemize}
\end{theorem}

\begin{remark}
The reader can easily check that Theorem \ref{main} holds if the infimum $%
\inf_{s\geq 1}\left( \frac{\left\vert I\right\vert }{\left\vert
J_{s}\right\vert }\right) ^{\gamma }$ in \eqref{def Psi gamma} is replaced
by $\inf_{s\geq 1}\eta \left( \frac{\left\vert I_{r}\right\vert }{\left\vert
J_{r,s^{\prime }}\right\vert }\right) $ for a suitable Dini function $\eta $
on $\left[ 1,\infty \right) $.
\end{remark}

The quantitative estimate we give for the norm of the Hilbert transform is
given in \eqref{e.prove1}. Consider now the conjecture of Volberg \cite{Vol}
that the two weight inequality holds if and only if the $A_{2}$ and testing
conditions hold. Since the Energy Condition \eqref{energy condition} is
actually a consequence of the $A_{2}$ and testing condition \eqref{e.H1},
Volberg's conjecture would be proved if we could take $\gamma =0$ in Theorem %
\ref{main}. (That we can take $ \varepsilon >0$ follows from the general techniques 
of \S\ref{s.good}.)
 There are subtle obstacles to overcome in order to achieve such
a characterization.

\bigskip

We will follow the beautiful approach of Nazarov, Treil and Volberg using
random grids, stopping intervals and corona decompositions. Energy enters
into the argument at those parts based upon the smoothness of the kernel,
see the Energy Lemma, especially \eqref{e.Eimplies} below. Much of the
argument we use appears in Chapters 17-22 of the CBMS book by Volberg \cite%
{Vol}, with the final touches in the preprint of Nazarov, Treil and Volberg 
\cite{NTV3}. In order to make this complicated proof self-contained, we
reproduce these critical ideas in our sufficiency proof below.

In \S \ref{s.example}, we exhibit a pair of weights which satisfy the
two-weight inequality, as they fall within the scope of our Main Theorem,
yet they do \emph{not} satisfy the Pivotal Condition of
Nazarov-Treil-Volberg.

The main novelty of this paper is that (1) the energy condition
is \emph{necessary} for the two weight testing conditions, (2) the Energy
Hypothesis can be inserted into the approach of \cite{NTV3}, and (3) that
the Pivotal Conditions are not necessary for the two-weight inequality.

\medskip

The integral defining $H(\sigma f)$ in \eqref{e.H<} is not in general
absolutely convergent, and we must introduce appropriate truncations. The
following canonical construction from \cite{Vol} serves our purposes here.

\begin{definition}
\label{def inequality}Let $\zeta $ be a fixed smooth nondecreasing function
on the real line satisfying%
\begin{equation*}
\zeta (t)=0\text{ for }t\leq \frac{1}{2}\text{ and }\zeta (t)=1\text{ for }%
t\geq 1.
\end{equation*}%
Given $\varepsilon >0$, set $\zeta _{\varepsilon }(t)=\zeta \left( \frac{t}{%
\varepsilon }\right) $ and define the smoothly truncated operator $%
T_{\varepsilon }$ by the absolutely convergent integral 
\begin{equation*}
T_{\varepsilon }f(x)=\int \frac{1}{y-x}\zeta _{\varepsilon }(\left\vert
x-y\right\vert )f(y)d\sigma \left( y\right) ,\qquad f\in L^{2}\left( \mathbb{%
\sigma }\right) \text{ with compact support}.
\end{equation*}%
We say \eqref{e.H<} holds if the inequality there holds for all compactly
supported $f$ with $T_{\varepsilon }$ in place of $T$, uniformly in $%
\varepsilon >0$.
\end{definition}

One easily verifies that all of the necessary conditions derived below can
be achieved using this definition provided $\omega $ and $\sigma $ have no
point masses in common (note that if $\omega =\sigma =\delta _{x}$, then (%
\ref{e.H<}) holds trivially with this definition while \eqref{A2} fails).
Moreover, the kernels $\frac{1}{y-x}\zeta _{\varepsilon }(\left\vert
x-y\right\vert )$ of $T_{\varepsilon }$ are uniformly standard Calder\'{o}%
n-Zygmund kernels, and thus all of the sufficiency arguments below hold as
well using this definition. In the sequel we will suppress the use of $%
T_{\varepsilon }$ and simply write $T$.


\section{Necessary Conditions}

In this section, we collect some conditions which follow either from the
assumed norm inequality or the testing conditions. These are the $A_{2}$
condition, a weak-boundedness condition, and the Energy Condition. The
principal novelty is the Energy Condition.

\subsection{The Necessity of the $A_2$ Condition}

\label{s.A2}

In this section we will give a new proof of this known fact due to
F.~Nazarov:


\begin{proposition}
\label{p.A2necceary} Assuming the norm inequality \eqref{e.H<}, we have the $%
A_2$ condition \eqref{A2}. Qualitatively, 
\begin{equation}
\mathcal{N}\equiv \left\Vert H(\cdot \sigma )\right\Vert _{L^{2}(\sigma
)\rightarrow L^{2}(\omega )} \gtrsim \mathcal{A}_{2}  \label{e.lower}
\end{equation}
\end{proposition}


The analogue of this inequality in the unit disk was proved for the
conjugate operator in \cite{NTV3} and \cite{Vol} Chapter 16. We provide a
real-variable proof here.


\begin{proof}
Fix an interval $I$ and for $a\in \mathbb{R}$ and $r>0$ let 
\begin{eqnarray*}
s_{I}\left( x\right) &=&\frac{\left\vert I\right\vert }{\left\vert
I\right\vert +\left\vert x-x_{I}\right\vert }, \\
f_{a,r}\left( y\right) &=&\mathbf{1}_{\left( a-r,a\right) }\left( y\right)
s_{I}\left( y\right) ,
\end{eqnarray*}%
where $x_{I}$ is the center of the interval $I$. For $y<x$ we have%
\begin{eqnarray*}
\left\vert I\right\vert \left( x-y\right) &=&\left\vert I\right\vert \left(
x-x_{I}\right) +\left\vert I\right\vert \left( x_{I}-y\right) \\
&\lesssim &\left( \left\vert I\right\vert +\left\vert x-x_{I}\right\vert
\right) \left( \left\vert I\right\vert +\left\vert x_{I}-y\right\vert
\right) ,
\end{eqnarray*}%
and so%
\begin{equation*}
\frac{1}{x-y}\geq \left\vert I\right\vert ^{-1}s_{I}\left( x\right)
s_{I}\left( y\right) ,\ \ \ \ \ y<x.
\end{equation*}%
Thus for $x>a$ we obtain that%
\begin{eqnarray*}
H\left( f_{a,r}\sigma \right) \left( x\right) &=&\int_{a-r}^{a}\frac{1}{x-y}%
s_{I}\left( y\right) ^{{}}d\sigma \left( y\right) \\
&\geq &\left\vert I\right\vert ^{-1}s_{I}\left( x\right)
\int_{a-r}^{a}s_{I}\left( y\right) ^{2}d\sigma \left( y\right) .
\end{eqnarray*}%
Applying our assumed two weight inequality \eqref{e.H<} in the sense of
Definition \ref{def inequality}, and then letting $\varepsilon >0$ there go
to $0$, we see that 
\begin{align*}
& \left\vert I\right\vert ^{-2}\int_{a}^{\infty }s_{I}\left( x\right)
^{2}\left( \int_{a-r}^{a}s_{I}\left( y\right) ^{2}d\sigma \left( y\right)
\right) ^{2}d\omega \left( x\right) \\
& \leq \left\Vert H(\sigma f_{a,r})\right\Vert _{L^{2}(\omega )}^{2}\lesssim 
\mathcal{N}^{2}\left\Vert f_{a,r}\right\Vert _{L^{2}(\sigma )}^{2}=\mathcal{N%
}^{2}\int_{a-r}^{a}s_{I}\left( y\right) ^{2}d\sigma \left( y\right) .
\end{align*}

Rearranging the last inequality, we obtain%
\begin{equation*}
\left\vert I\right\vert ^{-2}\int_{a}^{\infty }s_{I}\left( x\right)
^{2}d\omega \left( x\right) \int_{a-r}^{a}s_{I}\left( y\right) ^{2}d\sigma
\left( y\right) \lesssim \mathcal{N}^{2},
\end{equation*}%
and upon letting $r\rightarrow \infty $, and taking a square root, 
\begin{equation} \label{e.s=}
\left( \int_{a}^{\infty }s_{I}\left( x\right) ^{2}d\omega \left( x\right)
\int_{-\infty }^{a}s_{I}\left( y\right) ^{2}d\sigma \left( y\right) \right)
^{\frac{1}{2}}\lesssim \mathcal{N}\left\vert I\right\vert .
\end{equation}%
The ranges of integration are complementary half-lines, and clearly we can
reverse the role of the two weights above.

Choose $a\in \mathbb{R}$ which evenly divides the $L^{2}(\sigma )$-norm of $%
s_{I}$ in this sense:  
\begin{equation} \label{e.p=}
\int_{-\infty }^{a}s_{I}\left( y\right) ^{2}d\sigma \left( y\right)
=\int_{a}^{\infty }s_{I}\left( y\right) ^{2}d\sigma \left( y\right) =\frac{1%
}{2}\int_{-\infty }^{\infty }s_{I}\left( y\right) ^{2}d\sigma \left(
y\right) ,
\end{equation}%
and conclude that%
\begin{align}
\int_{-\infty }^{\infty }s_{I}\left( x\right) ^{2}d\omega \left( x\right)
\int_{-\infty }^{\infty }s_{I}\left( y\right) ^{2}d\sigma \left( y\right) &
=\int_{-\infty }^{a}s_{I}\left( x\right) ^{2}d\omega \left( x\right)
\int_{-\infty }^{\infty }s_{I}\left( y\right) ^{2}d\sigma \left( y\right) \\
& \qquad +\int_{a}^{\infty }s_{I}\left( x\right) ^{2}d\omega \left( x\right)
\int_{-\infty }^{\infty }s_{I}\left( y\right) ^{2}d\sigma \left( y\right) \\ 
\label{e.rep}
& \leq 2\int_{-\infty }^{a}s_{I}\left( x\right) ^{2}d\omega \left( x\right)
\int_{a}^{\infty }s_{I}\left( y\right) ^{2}d\sigma \left( y\right) \\
& \qquad +2\int_{a}^{\infty }s_{I}\left( x\right) ^{2}d\omega \left(
x\right) \int_{-\infty }^{a}s_{I}\left( y\right) ^{2}d\sigma \left( y\right)
\\
& \lesssim \mathcal{N}^{2}\left\vert I\right\vert ^{2}.
\end{align}%
Dividing through by $\lvert I\rvert ^{2}$, and forming the supremum over $I$
concludes the proof in the case where we can choose $ a$ as in \eqref{e.p=}.

We now consider the case where a point masses in $ \sigma $ prevents \eqref{e.p=} from holding.  
If we replace $a$ by $a+\varepsilon $ in \eqref{e.s=},  and then
let $\varepsilon \rightarrow 0$ this gives%
\begin{equation*}
\int_{\left( a,\infty \right) }s_{I}\left( x\right) ^{2}d\omega \left(
x\right) \int_{\left( -\infty ,a\right] }s_{I}\left( y\right) ^{2}d\sigma
\left( y\right) \lesssim \mathcal{N}^{2}\left\vert I\right\vert ^{2}.
\end{equation*}%
The ranges of integration are complementary half-lines, and clearly we can
reverse the role of the open and closed half-lines, as well as the role of
the two weights, resulting in four such inequalities altogether.

Now choose $a\in \mathbb{R}$ to be the largest number satisfying 
\begin{equation}
\int_{\left( -\infty ,a\right) }s_{I}\left( y\right) ^{2}d\sigma \left(
y\right) \leq \frac{1}{2}\int_{-\infty }^{\infty }s_{I}\left( y\right)
^{2}d\sigma \left( y\right) .  \label{max a}
\end{equation}%
Of course it may happen that \emph{strict} inequality occurs in (\ref{max a}%
) due to a point mass in $\sigma $ at the point $a$. In the event that this
point mass is missing or relatively small, i.e. 
\begin{equation*}
\sigma \left( \left\{ a\right\} \right) s_{I}\left( a\right) ^{2}\leq \frac{1%
}{2}A,
\end{equation*}%
where $A=\int_{-\infty }^{\infty }s_{I}\left( y\right) ^{2}d\sigma \left(
y\right) $, we can conclude that at least one of the integrals $\int_{\left(
-\infty ,a\right) }s_{I}\left( y\right) ^{2}d\sigma \left( y\right) $ or $%
\int_{\left( a,\infty \right) }s_{I}\left( y\right) ^{2}d\sigma \left(
y\right) $ is at least $\frac{1}{4}A$. Suppose that the first integral $%
\int_{\left( -\infty ,a\right) }s_{I}\left( y\right) ^{2}d\sigma \left(
y\right) $ is at least $\frac{1}{4}A$, and moreover is the smaller of the
two if both are at least $\frac{1}{4}A$. Then we also have $\int_{\left[
a,\infty \right) }s_{I}\left( y\right) ^{2}d\sigma \left( y\right) \geq 
\frac{1}{4}A$, where we have included the point mass at $a$ in the integral
on the left. We can now repeat the argument of \eqref{e.rep} to conclude this case.  

It remains to consider the case that the point mass at $a$ is a relatively
large proportion of the Poisson integral, i.e.%
\begin{equation*}
\sigma \left( \left\{ a\right\} \right) s_{I}\left( a\right) ^{2}>\frac{1}{2}%
\int_{-\infty }^{\infty }s_{I}\left( y\right) ^{2}d\sigma \left( y\right) .
\end{equation*}%
But then, consider the two universal inequalities 
\begin{eqnarray*}
\int_{\left( a,\infty \right) }s_{I}\left( x\right) ^{2}d\omega \left(
x\right) \int_{\left( -\infty ,a\right] }s_{I}\left( y\right) ^{2}d\sigma
\left( y\right) &\lesssim &\mathcal{N}^{2}\left\vert I\right\vert ^{2}, \\
\int_{\left( -\infty ,a\right) }s_{I}\left( x\right) ^{2}d\omega \left(
x\right) \int_{\left[ a,\infty \right) }s_{I}\left( y\right) ^{2}d\sigma
\left( y\right) &\lesssim &\mathcal{N}^{2}\left\vert I\right\vert ^{2}\,. 
\end{eqnarray*}%
Both integrals against $ \sigma $ include the point mass at $a $, hence 
they exceed 
 $\frac{1}{2}%
\int_{-\infty }^{\infty }s_{I}\left( y\right) ^{2}d\sigma \left( y\right) $.
It is our hypothesis that  $ \omega $ and $ \sigma $ do not have common 
point masses, so we conclude the $ A_2$ condition in this case.

\end{proof}

\begin{remark}
\label{r.fat} Preliminary results in this direction were obtained by
Muckenhoupt and Wheeden, and in the setting of fractional integrals by
Gabidzashvili and Kokilashvili, and here we follow the argument proving
(1.9) in Sawyer and Wheeden \cite{MR1175693}, where `two-tailed'
inequalities, like those in the $A_{2}$ condition \eqref{A2}, originated in
the fractional integral setting. A somewhat different approach to this for
the conjugate operator in the disk uses conformal invariance and appears in 
\cite{NTV3}, and provides the first instance of a strengthened $A_{2}$
condition being proved necessary for a two weight inequality for a singular
integral.
\end{remark}

\begin{remark}\label{r.a2}
 In the proof of the sufficient direction of the Main Theorem \ref{main}, we
only need `half' of the $A_{2}$ condition. Namely, we only need 
\begin{equation*}
\sup_{I}\frac{\omega (I)}{\lvert I\rvert }\mathsf{P}(I,\sigma )<\infty \,,
\end{equation*}%
along with the  dual condition. This point could be of use in
seeking to verify that a particular pair of weights satisfies the testing
conditions.
\end{remark}

\subsection{The Weak Boundedness Condition}

We show that a condition analogous to the weak-boundedness criteria of the
David and Journ\'e $T1$ theorem is a consequence of the $A_2$ condition and
the two testing conditions \eqref{e.H1} and \eqref{e.H2}.

For a constant $C>1$, let $\mathcal{W}_C$ be the best constant in the
inequality 
\begin{equation}
\left\vert \int_{J}H(\mathbf{1}_{I}\sigma )\;\omega (dx)\right\vert \leq 
\mathcal{W}_C\sigma (I)^{1/2}\omega (J)^{1/2}\,,  \label{e.W}
\end{equation}%
where the inequality is uniform over all intervals $I,J$ with $\textup{dist}%
(I,J)\leq \lvert I\rvert +\lvert J\rvert $ and $C^{-1}\leq \lvert I\rvert
/\lvert J\rvert \lesssim C $.  
The exact value of $ C$ that we will need in the sufficient direction of our 
Theorem depends upon the choice of 
$ \varepsilon >0$ in the Energy Hypothesis.  It is therefore a constant, and we will 
simply write $ \mathcal W$ below.

\begin{proposition}
For $C>1$, we have the inequality \label{wbc}%
\begin{equation*}
\mathcal{W}\leq \min \left\{ \mathcal{H},\mathcal{H}^{\ast }\right\}
+C^{\prime }\mathcal{A}_{2}.
\end{equation*}
\end{proposition}


\begin{proof}
To see this we write%
\begin{equation*}
\int_{J}H(\mathbf{1}_{I}\sigma )\;d\omega =\int_{J_{L}}H(\mathbf{1}%
_{I}\sigma )\;d\omega +\int_{J_{C}}H(\mathbf{1}_{I}\sigma )\;d\omega
+\int_{J_{R}}H(\mathbf{1}_{I}\sigma )\;d\omega ,
\end{equation*}%
where%
\begin{eqnarray*}
J_{L} &=&\left\{ x\in J\setminus I:x\text{ lies to the \emph{left} of }%
I\right\} , \\
J_{C} &=&J\cap I, \\
J_{R} &=&\left\{ x\in J\setminus I:x\text{ lies to the \emph{right} of }%
I\right\} .
\end{eqnarray*}%
Now we easily have%
\begin{eqnarray*}
\left\vert \int_{J_{C}}H(\mathbf{1}_{I}\sigma )\;d\omega \right\vert
&\lesssim &\sqrt{\omega (J_{C})}\left( \int_{J_{C}}\left\vert H(\mathbf{1}%
_{I}\sigma )\right\vert ^{2}\;d\omega \right) ^{\frac{1}{2}} \\
&\lesssim &\sqrt{\omega (J)}\left( \int_{I}\left\vert H(\mathbf{1}_{I}\sigma
)\right\vert ^{2}\;d\omega \right) ^{\frac{1}{2}} \\
&\lesssim &\sqrt{\omega (J)}\mathcal{H}\sqrt{\sigma (I)}.
\end{eqnarray*}%
The two remaining terms are each handled in the same way, so we treat only
the first one $\int_{J_{L}}H(\mathbf{1}_{I})\sigma \;d\omega $. We will use
Muckenhoupt's characterization of Hardy's inequality \cite{Muc} for weights $%
\widehat{\omega }$ and $\sigma $: if $B$ is the best constant in%
\begin{equation}
\int_{0}^{a}\left( \int_{0}^{x}f\sigma \right) ^{2}d\widehat{\omega }\left(
x\right) \leq B^{2}\int_{0}^{a}\left\vert f\right\vert ^{2}d\sigma ,\ \ \ \
\ f\geq 0,  \label{Muck}
\end{equation}%
then, 
\begin{equation}
B^{2}\approx \sup_{0<r<a}\left( \int_{r}^{a}d\widehat{\omega }\right) \left(
\int_{0}^{r}d\sigma \right) .  \label{B is}
\end{equation}%
We will give the proof here assuming that $\omega $ and $\sigma $ have no
point masses,  as the general case is hard.  

Without loss of generality we consider the\ extreme case $J_{L}=\left(
-a,0\right) $ and $I=\left( 0,b\right) $ with $0<a<b$. We decompose $%
I=I_{1}\cup I_{2}$ with $I_{1}=\left( 0,a\right) $ and $I_{2}=\left(
a,b\right) $. First we note the easy estimate%
\begin{eqnarray*}
\left\vert \int_{J_{L}}H(\mathbf{1}_{I_{2}}\sigma )\;d\omega \right\vert
&\lesssim &\int_{-a}^{0}\left( \int_{a}^{b}\frac{1}{y}d\sigma \left(
y\right) \right) d\omega \left( x\right) =\omega( J_{L})\int_{a}^{b}\frac{1}{%
y}d\sigma \left( y\right) \\
&\lesssim &\omega (J_{L})\sqrt{\sigma( I_{2})}\left( \int_{a}^{b}\frac{1}{%
y^{2}}d\sigma \left( y\right) \right) ^{\frac{1}{2}} \\
&\lesssim &\sqrt{\omega( J_{L})\sigma( I_{2})}\sqrt{\frac{ \omega (J_{L})}{a}%
\mathsf{P}\left( I_{1},\sigma \right) }\leq 2\mathcal{A}_{2}\sqrt{\omega(
J_{L})\sigma( I_{2})},
\end{eqnarray*}%
since $J_{L}$ and $I_{1}$ are touching intervals of equal length $a$. Then
we use \eqref{B is} for the other term:%
\begin{eqnarray*}
\left\vert \int_{J_{L}}H(\mathbf{1}_{I_{1}}\sigma )\;d\omega \right\vert
&=&\int\!\!\int_{\left( -a,0\right) \times \left( 0,a\right) }\mathbf{1}
_{\left\{ -x>y\right\} }\frac{1}{y-x}d\sigma \left( y\right) \;d\omega
\left( x\right) \\
&&+\int\!\!\int_{\left( -a,0\right) \times \left( 0,a\right) }\mathbf{1}
_{\left\{ -x<y\right\} }\frac{1}{y-x}d\sigma \left( y\right) \;d\omega
\left( x\right) \\
&=&I+I\!I.
\end{eqnarray*}%
These two terms are symmetric in $\omega $ and $\sigma $ so we consider only
the first one $I$. We have letting $z=-x$ and $d\widetilde{\omega }\left(
z\right) =d\omega \left( -z\right) $ and $d\widehat{\omega }\left( z\right) =%
\frac{1}{z^{2}}d\widetilde{\omega }\left( z\right) $,%
\begin{eqnarray*}
I &=&\int_{0}^{a} \int_{0}^{z}\frac{1}{y+z}d\sigma \left( y\right) d\omega
\left( -z\right) \leq \int_{0}^{a} \frac{1}{z}\int_{0}^{z}d\sigma d%
\widetilde{\omega }\left( z\right) \\
&\lesssim &\left[ \int_{0}^{a}d\widetilde{\omega } \int_{0}^{a}\left(
\int_{0}^{z}d\sigma \right) ^{2}d\widehat{\omega }\left( z\right) \right] ^{%
\frac{1}{2}} \\
&\lesssim &B\left[ \int_{-a}^{0}d\omega \times \int_{0}^{a}d\sigma \right]^{%
\frac{1}{2}}=B\sqrt{\omega( J_{L})\sigma( I_{1})},
\end{eqnarray*}%
upon using \eqref{Muck} with $f\equiv 1$. Finally we obtain $B\lesssim 
\mathcal{A}_{2}$ from \eqref{B is} and 
\begin{eqnarray*}
\int_{r}^{a}d\widehat{\omega } \int_{0}^{r}d\sigma &=& \int_{r}^{a}rd%
\widehat{\omega } \times \frac{1}{r} \int_{0}^{r}d\sigma \\
&\leq& \mathsf{P}(\left( 0,r\right) ,\widetilde{\omega }) \times \frac{1}{r}%
\int_{\left( 0,r\right) }d\sigma \\
&=&\mathsf{P}(\left( -r,0\right) ,\omega ) \times \frac{1}{r}\int_{\left(
0,r\right) }d\sigma \lesssim \mathcal{A}_{2}^{2}.
\end{eqnarray*}%
Thus we have proved $\mathcal{W}\leq \mathcal{H}+C\mathcal{A}_{2}$. We
obtain $\mathcal{W}\leq \mathcal{H}^{\ast }+C\mathcal{A}_{2}$ by applying
the above reasoning to%
\begin{equation*}
\int_{J}H(\mathbf{1}_{I}\sigma )\;d\omega =\int_{I}H(\mathbf{1}_{J}\omega
)\;d\sigma .
\end{equation*}
\end{proof}


\subsection{The Energy Condition\label{section energy}}

We show here that the Energy Conditions are implied by the $A_{2}$ and
testing conditions.

\begin{proposition}
\label{p.energy}We have the inequality $\mathcal{E}\lesssim \mathcal{A}_{2}+%
\mathcal{H}\,,$ and similarly for $\mathcal{E}^{\ast }$.
\end{proposition}

The Energy Hypotheses with $\gamma >0$ are the essential tools in organizing
the sufficient proof.  The proof begins with this Lemma.

\begin{lemma}
\label{l.neccInequality} For any interval $I$ and any positive measure $\nu $
supported in $\mathbb{R}\setminus I$, we have%
\begin{equation}
\mathsf{P}\left( I;\nu \right) \leq 2\left\vert I\right\vert \inf_{x,y\in I}%
\frac{H \nu \left( x\right) -H \nu \left( y\right) }{x-y},  \label{neccinequ}
\end{equation}
\end{lemma}

For specificity, in this section, we are re-defining the Poisson integral to
be 
\begin{equation}  \label{redef}
\mathsf{P}\left( I;\nu \right) \equiv \frac{\nu (I)}{\left\vert I\right\vert 
} +\frac{\left\vert I\right\vert }{2}\int_{\mathbb{R}\setminus I}\frac{1}{%
\left\vert z-z_{I}\right\vert ^{2}}\nu \left(dz\right) ,
\end{equation}%
with $z_{I}$ the center of $I$. Note that this definition of $\mathsf{P}%
\left( I;\nu \right) $ is comparable to that in \eqref{e.P}. Note that $%
H\left( \mathbf{1 }_{I^{c}}\nu \right) $ is increasing on $I$ when $\nu $ is
positive, so that the infimum in \eqref{neccinequ} is nonnegative.

\begin{proof}
To see \eqref{neccinequ}, we suppose without loss of generality that $%
I=\left( -a,a\right) $, and a calculation then shows that for $-a\leq
x<y\leq a$,%
\begin{align*}
H\nu \left( y\right) -H\nu \left( x\right) & =\int_{\mathbb{R}\setminus
I}\left\{ \frac{1}{z-y}-\frac{1}{z-x}\right\} \nu (dz) \\
& =\left( y-x\right) \int_{\mathbb{R}\setminus I}\frac{1}{\left( z-y\right)
\left( z-x\right) }\nu (dz) \\
& \geq \frac{1}{4}\left( y-x\right) \int_{\mathbb{R}\setminus I}\frac{1}{%
z^{2}}\nu (dz)
\end{align*}%
since ${\left( z-y\right) \left( z-x\right) }$ is positive and satisfies%
\begin{equation*}
\frac{1}{\left( z-y\right) \left( z-x\right) }\geq \frac{1}{4z^{2}}
\end{equation*}%
on each interval $\left( -\infty ,-a\right) $ and $\left( a,\infty \right) $
in $\mathbb{R}\setminus I$ when $-a\leq x<y\leq a$. Thus we have from %
\eqref{redef}, and the assumption about the support of $\nu $, 
\begin{eqnarray*}
\mathsf{P}\left( I;\nu \right) &= & \frac{\left\vert I\right\vert }{2}\int_{%
\mathbb{R}\setminus I}\frac{1}{z^{2}}\nu (dz) \\
&\lesssim &2\left\vert I\right\vert \inf_{x,y\in I}\frac{H\nu \left(
y\right) -H\nu \left( x\right) }{y-x}.
\end{eqnarray*}
\end{proof}

\begin{proof}[Proof of Proposition~\protect\ref{p.energy}]
We recall the energy condition in \eqref{energy condition}. Fix an interval $%
I_{0}$, and pairwise disjoint strict subintervals $\{I_{r}\;:\;r\geq 1\}$.
Let $\{I _{s} \;:\; s\geq 1 \}$ be pairwise disjoint subintervals of $I _{r}$%
.

We apply \eqref{neccinequ}, so that for $x,y\in I_{r}$, we have 
\begin{eqnarray}
\frac{\left\vert y-x\right\vert }{\lvert I_{r}\rvert }\mathsf{P}\left( I_{r};%
\mathbf{1}_{I_{0}}\sigma \right) &\lesssim &\frac{\left\vert y-x\right\vert 
}{\left\vert I_{r}\right\vert } \frac{\sigma (I_{r}) } {\lvert I_r\rvert }
\label{e.2tgt} 
+\left\vert H\left( \mathbf{1}_{I_{0}\cap I_{r}^{c}}\sigma \right) \left(
y\right) -H\left( \mathbf{1}_{I_{0}\cap I_{r}^{c}}\sigma \right) \left(
x\right) \right\vert .  \notag
\end{eqnarray}%
Let us for the moment assume that the second term on the right is dominant.
Squaring the inequality above, averaging with respect to the measure $\omega 
$ in both $x$ and $y$, we obtain 
\begin{align}
\mathsf{E}(I_{r},\omega )^{2}\mathsf{P}\left( I_{r};\mathbf{1}_{I_{0}}\sigma
\right) ^{2}& \leq \mathbb{E}_{I_{r}}^{\omega (dx)}\mathbb{E}%
_{I_{r}}^{\omega (dy)}\left( \frac{\left\vert y-x\right\vert }{\left\vert
I_{r}\right\vert }\right) ^{2}\mathsf{P}\left( I_{r};\mathbf{1}%
_{I_{0}}\sigma \right) ^{2}  \label{e.strict!} \\
& \lesssim \mathbb{E}_{I_{r}}^{\omega (dx)}\mathbb{E}_{I_{r}}^{\omega
(dy)}\left\vert H\left( \mathbf{1}_{I_{0}\cap I_{r}^{c}}\sigma \right)
\left( y\right) -H\left( \mathbf{1}_{I_{0}\cap I_{r}^{c}}\sigma \right)
\left( x\right) \right\vert ^{2} \\
& \leq 4\mathbb{E}_{I_{r}}^{\omega (dx)}\left\vert H\left( \mathbf{1}%
_{I_{0}\cap I_{r}^{c}}\sigma \right) \right\vert ^{2}\,.
\end{align}

Multiply the last inequality by $\omega (I_{r})$ and sum in $r$ to get 
\begin{align*}
\sum_{r\geq 1}\omega (I_{r})\mathsf{E}(I_{r},\omega )^{2}\mathsf{P}\left(
I_{r};\mathbf{1}_{I_{0}}\sigma \right) ^{2}& \lesssim
\sum_{r}\int_{I_{r}}\left\vert H\left( \mathbf{1}_{I_{0}\cap
I_{r}^{c}}\sigma \right) \right\vert ^{2}d\omega \\
& \lesssim \int_{I_{0}}\left\vert H\left( \mathbf{1}_{I_{0}}\sigma \right)
\right\vert ^{2}d\omega +C\sum_{r}\int_{I_{r}}\left\vert H\left( \mathbf{1}%
_{I_{r}}\sigma \right) \right\vert ^{2}d\omega \\
& \leq2 \mathcal{H}^{2}\sigma (I_{0})
\end{align*}%
by the testing condition \eqref{e.H1} applied to both $I_{0}$ and $I_{r}$.

Returning to \eqref{e.2tgt}, it remains to consider the case where the first
term on the right is dominant. By the same reasoning, we arrive at 
\begin{align*}
\mathsf{E}(I_{r},\omega )^{2}\mathsf{P}\left( I_{r};\mathbf{1}_{I_{0}}\sigma
\right) ^{2}& \leq \mathbb{E}_{I_{r}}^{\omega (dx)}\mathbb{E}%
_{I_{r}}^{\omega (dy)}\frac{\left\vert y-x\right\vert ^{2}}{\left\vert
I_{r}\right\vert ^{2}}\frac{\sigma (I_{r})^{2}}{\lvert I_{r}\rvert ^{2}} \\
& \leq \frac{\sigma (I_{r})^{2}}{\lvert I_{r}\rvert ^{2}}\,.
\end{align*}%
Multiply the last inequality by $\omega (I_{r})$ and sum in $r$ to get 
\begin{equation*}
\sum_{r=1}^{\infty }\frac{\sigma (I_{r})^{2}}{\lvert I_{r}\rvert ^{2}}\omega
(I_{r})\leq \mathcal{A}_{2}^{2}\sum_{r=1}^{\infty }\sigma (I_{r})\leq 
\mathcal{A}_{2}^{2}\sigma (I_{0})\,.
\end{equation*}
\end{proof}

\begin{remark}
We refer to $\mathsf{E}(I,\omega )$ as the energy functional because in
dimension $n\geq 3$ the integral%
\begin{equation*}
\int_{I}\int_{I}\left\vert x-x^{\prime }\right\vert ^{2-n}d\omega \left(
x\right) d\omega \left( x^{\prime }\right)
\end{equation*}%
represents the energy required to \emph{compress} charge from infinity to a
distribution $\omega $ on $I$, assuming a repulsive inverse square law
force. When $n=1$, the force is attractive and the integral%
\begin{equation*}
\int_{I}\int_{I}\left\vert x-x^{\prime }\right\vert d\omega \left( x\right)
d\omega \left( x^{\prime }\right)
\end{equation*}%
represents the energy required to \emph{disperse} charge from a point to a
distribution $\omega $ on $I$.
\end{remark}

\subsection{The Hybrid Condition}

We begin with a monotonicity property of energy, and then apply it to show
that the Hybrid Condition implies the Energy Hypothesis. This Lemma helps
clarify the role of the Hybrid Conditions.


\begin{lemma}
\label{l.monotone} Fix $0\leq \epsilon \leq 2$. Let $I_{0}$ be an interval,
and $\{I_{r}\;:\;r\geq 1\}$ a partition of $I_{0}$. We have the inequalities
for $0<\varepsilon <1-\frac{\epsilon }{2}$: 
\begin{align}
\sum_{r\geq 1}\omega (I_{r})\lvert I_{r}\rvert ^{\epsilon }\mathsf{E}%
(I_{r};\omega )^{\epsilon }& \leq \omega (I_{0})\lvert I_{0}\rvert
^{\epsilon }\mathsf{E}(I_{0};\omega )^{\epsilon }\,.  \label{e.mono<} \\
\sum_{r\geq 1}\omega (I_{r})\lvert I_{r}\rvert ^{2-2\varepsilon }\mathsf{E}%
(I_{r};\omega )^{\epsilon }& \leq \sup_{r\geq 1}\left( \frac{\lvert
I_{r}\rvert }{\lvert I_{0}\rvert }\right) ^{2-2\varepsilon -\epsilon }\omega
(I_{0})\lvert I_{0}\rvert ^{2-2\varepsilon }\mathsf{E}(I_{0};\omega
)^{\epsilon }\,.
\end{align}
\end{lemma}


The second inequality is obvious given the first; as it turns out this is
the basic fact used to exploit the Hybrid Conditions for $0\leq \epsilon <2$%
, so we have stated it explicitly.


\begin{proof}
The inequality is obvious for $\epsilon =0$. We prove it for $\epsilon =2$.
This is rather clear if we make the definition 
\begin{equation*}
\mathsf{Var}_{I}^{\omega }\equiv \omega \left( I\right) \ \mathbb{E}%
_{I}^{\omega }(x-\mathbb{E}_{I}^{\omega }x)^{2}\,.
\end{equation*}%
Then, we have $\omega \left( I\right) \lvert I\rvert ^{2}\mathsf{E}(I;\omega
)^{2}=\mathsf{Var}_{I}^{\omega }$.

Second, variation of a random variable $Z$ is the squared $L^{2}$-distance
of $Z$ from the linear space of constants. And $\omega \left( I_{0}\right)
\lvert I_{0}\rvert ^{2}\mathsf{E}(I_{0};\omega )^{2}$ admits a transparent
reformulation in this language: The random variable is $x$ and the
probability measure is normalized $\omega $ measure. Moreover, 
\begin{equation*}
\sum_{r\geq 1}\omega (I_{r})\lvert I_{r}\rvert ^{2}\mathsf{E}(I_{r};\omega
)^{2}
\end{equation*}%
is the squared $L^{2}$-distance of $x$ to the space of functions piecewise
constant on the intervals of the partition $\{I_{r}\;:\;r\geq 1\}$. Hence,
the inequality above is immediate.

\smallskip For the case of $0<\epsilon <2$, we apply H\"{o}lder's inequality
and appeal to the case of $\epsilon =2$. 
\begin{align*}
\left( \sum_{r\geq 1}\omega (I_{r})\lvert I_{r}\rvert ^{\epsilon }\mathsf{E}%
(I_{r};\omega )^{\epsilon }\right) ^{1/\epsilon }& \leq \omega
(I_{0})^{(2-\epsilon )/2\epsilon }\left( \sum_{r\geq 1}\omega (I_{r})\lvert
I_{r}\rvert ^{2}\mathsf{E}(I_{r};\omega )^{2}\right) ^{1/2} \\
& \leq \omega (I_{0})^{(2-\epsilon )/2\epsilon }\omega (I_{0})^{1/2}\lvert
I_{0}\rvert \mathsf{E}(I_{0},\omega )
\end{align*}%
which is the claimed inequality.
\end{proof}


Here is a Poisson inequality for good intervals that will see service both
here and later in the paper.

\begin{lemma}
\label{Poisson inequality}Suppose that $J\subset I\subset \widehat{I}$ and
that $\operatorname{dist}(J,e(I))>\tfrac{1}{2}\lvert J\rvert ^{\varepsilon }\lvert
I\rvert ^{1-\varepsilon }$. Then%
\begin{equation}
\lvert J\rvert ^{2\varepsilon -2}\mathsf{P}(J,\sigma \mathbf{1}_{\widehat{I}%
\setminus I})^{2}\lesssim \lvert I\rvert ^{2\varepsilon -2}\mathsf{P}%
(I,\sigma \mathbf{1}_{\widehat{I}\setminus I})^{2}.  \label{e.Jsimeq}
\end{equation}
\end{lemma}

\begin{proof}
We have%
\begin{equation*}
\mathsf{P}\left( J,\sigma \chi _{\widehat{I}\setminus I}\right) \approx
\sum_{k=0}^{\infty }2^{-k}\frac{1}{\left\vert 2^{k}J\right\vert }%
\int_{\left( 2^{k}J\right) \cap \left( \widehat{I}\setminus I\right)
}d\sigma ,
\end{equation*}%
and $\left( 2^{k}J\right) \cap \left( \widehat{I}\setminus I\right) \neq
\emptyset $ requires%
\begin{equation*}
dist\left( J,e\left( I\right) \right) \lesssim \left\vert 2^{k}J\right\vert .
\end{equation*}%
By our distance assumption we must then have%
\begin{equation*}
\left\vert J\right\vert ^{\varepsilon }\left\vert I\right\vert
^{1-\varepsilon }\leq dist\left( J,e\left( I\right) \right) \lesssim
2^{k}\left\vert J\right\vert ,
\end{equation*}%
or%
\begin{equation*}
2^{-k}\lesssim \left( \frac{\left\vert J\right\vert }{\left\vert
I\right\vert }\right) ^{1-\varepsilon }.
\end{equation*}%
Thus we have%
\begin{equation*}
\mathsf{P}\left( J,\sigma \chi _{\widehat{I}\setminus I}\right) \lesssim
2^{-k}\mathsf{P}\left( I,\sigma \chi _{\widehat{I}\setminus I}\right)
\lesssim \left( \frac{\left\vert J\right\vert }{\left\vert I\right\vert }%
\right) ^{1-\varepsilon }\mathsf{P}\left( I,\sigma \chi _{\widehat{I}%
\setminus I}\right) ,
\end{equation*}%
which is the inequality \eqref{e.Jsimeq}.
\end{proof}

We can now obtain that the Hybrid Condition implies the Energy Hypothesis.

\begin{lemma}
\label{hybrid implies hypothesis}Let $0\leq \epsilon <2$. Then the functional%
\begin{equation*}
\Psi \left( J,E\right) \equiv \omega \left( J\right) \mathsf{E}\left(
J,\omega \right) ^{\epsilon }\mathsf{P}\left( J,\mathbf{1}_{E}\sigma \right)
^{2}
\end{equation*}%
satisfies the three properties in \eqref{Psi properties} with $0<\varepsilon
<1-\frac{\epsilon }{2}$. As a consequence, the Hybrid Condition (\ref{Hy Con}%
) implies the Energy Hypothesis \eqref{smallest condition} with $\gamma
=2-2\varepsilon -\epsilon $.
\end{lemma}

\begin{proof}
The first line in \eqref{Psi properties} is the usual $A_{2}$ condition, and
the second line is the Hybrid Condition \eqref{Hy Con} with $\epsilon
=\gamma $. Thus we must show the third line:%
\begin{equation*}
\sum_{r\geq 1}\Phi \left( J_{r},\widehat{I}\setminus I_{0}\right) \leq
\sup_{r\geq 1}\left( \frac{\left\vert J_{r}\right\vert }{\left\vert
I_{0}\right\vert }\right) ^{\gamma }\Psi \left( I_{0},\widehat{I}\setminus
I_{0}\right) ,
\end{equation*}%
for all $\varepsilon $-good subparitions $\left\{ J_{r}\right\} $ of $I_{0}$%
, i.e. those satisfying \eqref{good property}. From Lemma \ref{Poisson
inequality} we have%
\begin{equation*}
\lvert J_{r}\rvert ^{2\varepsilon -2}\mathsf{P}(J_{r},\sigma \mathbf{1}_{%
\widehat{I}\setminus I_{0}})^{2}\lesssim \lvert I_{0}\rvert ^{2\varepsilon
-2}\mathsf{P}(I_{0},\sigma \mathbf{1}_{\widehat{I}\setminus I_{0}})^{2}.
\end{equation*}%
Now use Lemma \ref{l.monotone} and $\mathsf{E}(J_{r},\omega )^{2}\leq 
\mathsf{E}(J_{r},\omega )^{\epsilon }$ to obtain%
\begin{eqnarray*}
\sum_{r\geq 1}\Phi \left( J_{r},\widehat{I}\setminus I_{0}\right)
&=&\sum_{r\geq 1}\omega (J_{r})\lvert J_{r}\rvert ^{2-2\varepsilon }\mathsf{E%
}(J_{r},\omega )^{2}\lvert J_{r}\rvert ^{2\varepsilon -2}\mathsf{P}%
(J_{r},\sigma \mathbf{1}_{\widehat{I}\setminus I_{0}})^{2} \\
&\lesssim &\sum_{r\geq 1}\omega (J_{r})\lvert J_{r}\rvert ^{2-2\varepsilon }%
\mathsf{E}(J_{r},\omega )^{\epsilon }\lvert I_{0}\rvert ^{2\varepsilon -2}%
\mathsf{P}(I_{0},\sigma \mathbf{1}_{\widehat{I}\setminus I_{0}})^{2} \\
&\leq &\sup_{r\geq 1}\left( \frac{\lvert J_{r}\rvert }{\lvert I_{0}\rvert }%
\right) ^{2-2\varepsilon -\epsilon }\times \omega (I_{0})\mathsf{E}%
(I_{0};\omega )^{\epsilon }\mathsf{P}(I_{0},\sigma \mathbf{1}_{\widehat{I}%
\setminus I_{0}})^{2} \\
&=&\sup_{r\geq 1}\left( \frac{\left\vert J_{r}\right\vert }{\left\vert
I_{0}\right\vert }\right) ^{\gamma }\Psi \left( I_{0},\widehat{I}\setminus
I_{0}\right) .
\end{eqnarray*}
\end{proof}

\section{Grids, Haar Function, Carleson Embedding}

This section collects some standard facts which can be found e.g. in \cite%
{Vol}. We call a collection of intervals $\mathcal{G}$ a \emph{grid} iff for
all $I,J\in \mathcal{G}$ we have $I\cap J\in \{\emptyset ,I,J\}$. An
interval $I\in \mathcal{G}$ may have a \emph{parent} $I^{(1)}$: The unique
minimal interval $J\in \mathcal{G}$ that strictly contains $I$. Recursively
define $I^{(j+1)}=(I^{(j)})^{(1)}$. In the analysis of the paper, it will be
necessary to distinguish the grid in question when passing to a parent. We
accordingly set 
\begin{equation}
\pi _{\mathcal{G}}^{1}(I)\equiv \text{The unique minimal interval }J\in 
\mathcal{G}\text{ that strictly contains }I.  \label{e.parent}
\end{equation}%
Recursively set $\pi _{\mathcal{G}}^{j+1}(I)=\pi _{\mathcal{G}}^{1}(\pi _{%
\mathcal{G}}^{j}(I))$. Note that the definition of $\pi _{\mathcal{G}%
}^{1}(I) $ makes sense even if $I\not\in \mathcal{G}$.

A grid $\mathcal{D}$ is \emph{dyadic} if each interval $I\in \mathcal{D}$ is
union of $I _{-}, I _{+} \in \mathcal{D}$, with $I _{-}$ being the left-half
of $I$, and likewise for $I _{+}$. We will refer to $I _{\pm}$ as the \emph{%
children} of $I$.

\medskip

A dyadic grid $\mathcal{D}$, with weight $\sigma $ admits the \emph{Haar
basis adapted to }$\sigma $\emph{\ and $\mathcal{D}$}. This basis is
especially nice if the weight $\sigma $ does not assign positive mass to any
endpoint of an interval in $\mathcal{D}$. This can be achieved by
e.\thinspace g.\thinspace a joint translation of the intervals in $\mathcal{D%
}$, and so it will be a standing assumption.

The Haar basis $\{h_{I}^{\sigma }\;:\;I\in \mathcal{D}\}$ is explicitly
defined to be 
\begin{equation}
h_{I}^{\sigma }\equiv \frac{-\sigma (I_{+})\mathbf{1}_{I_{-}}+\sigma (I_{-})%
\mathbf{1}_{I_{+}}}{[\sigma (I_{+})^{2}\sigma (I_{-})+\sigma
(I_{-})^{2}\sigma (I_{+})]^{1/2}}=\sqrt{\frac{\sigma (I_{-})\sigma (I_{+})}{%
\sigma (I)}}\left( -\frac{\mathbf{1}_{I_{-}}}{\sigma (I_{-})}+\frac{\mathbf{1%
}_{I_{+}}}{\sigma (I_{+})}\right) ,  \label{e.hmu}
\end{equation}%
with the convention that $h_{I}^{\sigma }\equiv 0$ if the restriction of $%
\sigma $ to either child $I_{-}$ or $I_{+}$ vanishes. The martingale
properties of the Haar function are decisive, still at a couple of points,
we have recourse to the formula 
\begin{equation}
\left\vert \mathbb{E}_{I_{\theta }}^{\sigma }h_{I}^{\sigma }\right\vert =%
\sqrt{\frac{\sigma (I_{-\theta })}{\sigma (I)\sigma (I_{\theta })}}\leq 
\sqrt{\frac{1}{\sigma (I_{\theta })}}.  \label{e.Eh}
\end{equation}

These functions are (1) pairwise orthogonal, (2) have $\sigma $-integral
zero, (3) have $L^{2}(\sigma )$-norm either $0$ or $1$, and (4) form a basis
for $L^{2}(\sigma )$. We also define 
\begin{equation}
\Delta _{I}^{\sigma }f\equiv \left\langle f,h_{I}^{\sigma }\right\rangle
_{\sigma }\cdot h_{I}^{\sigma }  \label{e.Delta}
\end{equation}%
where by $\left\langle \cdot ,\cdot \right\rangle _{\sigma }$ we mean the
natural inner product on $L^{2}(\sigma )$. We then have the $L^{2}(\sigma )$
identity 
\begin{equation}
f=\sum_{I\in \mathcal{D}}\Delta _{I}^{\sigma }f\,,  \label{e.DeltaF}
\end{equation}%
for all $f\in L^{2}\left( \sigma \right) $ that are supported in a dyadic
interval $I^{0}$ and satisfy $\int_{I^{0}}fd\sigma =0$. We remark that by a
simple reduction in (17.3) of \cite{Vol}, we only need \eqref{e.DeltaF} for
such $f$ in the proof of our theorem.

Note that \eqref{e.DeltaF} yields the Plancherel formula%
\begin{equation}
\left\Vert f\right\Vert _{L^{2}\left( \sigma \right) }^{2}=\sum_{I\in 
\mathcal{D}}\left\vert \left\langle f,h_{I}^{\sigma }\right\rangle _{\sigma
}\right\vert ^{2},\ \ \ \ \ f\in L^{2}\left( \sigma \right) ,supp\ f\subset
I^{0},\int_{I^{0}}fd\sigma =0.  \label{Plancherel}
\end{equation}%
The following simple identities are basic as well. We have 
\begin{equation}
\Delta _{I}^{\sigma }f=\left\{ \mathbf{1}_{I_{-}}\mathbb{E}_{I_{-}}^{\sigma
}f+\mathbf{1}_{I_{+}}\mathbb{E}_{I_{+}}^{\sigma }f\right\} -\mathbf{1}_{I}%
\mathbb{E}_{I}^{\sigma }f\,.  \label{e.mart1}
\end{equation}%
Consequently, for two intervals $I_{1}\subset I_{2}$, $I_{1},I_{2}\in 
\mathcal{D}$, the sum below is telescoping, so easily summable: 
\begin{equation}
\sum_{I\;:\;I_{1}\subsetneq J\subset I_{2}}\Delta _{J}^{\sigma }f\left(
x\right) =\mathbb{E}_{I_{1}}^{\sigma }f-\mathbb{E}_{I_{2}}^{\sigma }f\,,\ \
\ \ \ x\in I_{1}.  \label{e.mart2}
\end{equation}%
In these displays, we are using the notation 
\begin{equation}
\mathbb{E}_{I}^{\sigma }\phi \equiv \sigma (I)^{-1}\int_{I}\phi \;\sigma
\left( {dx}\right) \,,  \label{e.expect}
\end{equation}%
thus, $\mathbb{E}_{I}^{\sigma }f$ is the average value of $f$ with respect
to the weight $\sigma $ on interval $I$. \medskip

We turn to a brief description of paraproducts. The familiar Carleson
Embedding Theorem is fundamental for the proof. The proof follows classical
lines, using that the map $f\rightarrow \mathbb{E}_{I}^{\sigma }f$ is type $%
\left( \infty ,\infty \right) $ and also weak type $\left( 1,1\right) $ with
respect to the measure $\sum_{I\in \mathcal{D}}a_{I}\delta _{I}$ on $%
\mathcal{D}$ by the Carleson condition \eqref{e.embed2}.

\begin{theorem}
\label{carleson embedding}Fix a weight $\sigma $ and consider nonnegative
constants $\left\{ a_{I}:I\in \mathcal{D}\right\} $. The following two
inequalities are equivalent:%
\begin{equation}
\sum_{I\in \mathcal{D}}a_{I}\left\vert \mathbb{E}_{I}^{\sigma }f\right\vert
^{2}\leq C _{1}\left\Vert f\right\Vert _{L^{2}\left( \sigma \right) }^{2},
\label{e.embed1}
\end{equation}%
\begin{equation}
\sum_{I\in \mathcal{D}:I\subset J}a_{I}\leq C_{2}\sigma \left( J\right) ,\ \
\ \ \ J\in \mathcal{D}.  \label{e.embed2}
\end{equation}%
Taking $C_{1}$ and $C_{2}$ to be the best constants in these inequalities,
we have $C_{1}\approx C_{2}$ with the implied constant independent of $%
\sigma $.
\end{theorem}

There is another language commonly associated with the Carleson Embedding
Theorem. For the purposes of this discussion, let $\widehat{I}=I\times
\lbrack 0,\lvert I\rvert ]$ be the square in the upper half-plane $\mathbb{R}%
_{+}^{2}$ with face $I$ on the real line, viewed as the boundary of $\mathbb{%
R}_{+}^{2}$. This is called the \emph{box over $I$}. And consider the linear
map, a $\sigma $-weighted analog of the Poisson integral, 
\begin{equation*}
Af(x,t)\equiv \mathbb{E}_{(x-t/2,x+t/2)}^{\sigma }f\,.
\end{equation*}%
Given a measure $\mu $ on $\mathbb{R}_{+}^{2}$, this operator maps $L^{2}(%
\mathbb{R},\sigma )$ into $L^{2}(\mathbb{R}_{+}^{2},\mu )$ if and only if
the measure $\mu $ satisfies the \emph{Carleson measure condition} 
\begin{equation}
\mu (\widehat{I})\lesssim \sigma (I)\,,\qquad  \label{e.box}
\end{equation}%
The condition \eqref{e.embed2} above is a discrete analog of this condition,
with the measure $\mu $ being defined by a sum of Dirac point masses at the
center of the tops of the boxes over $I$: 
\begin{equation}
\mu \equiv \sum_{I\in \mathcal{D}}a_{I}\delta _{(c_{I},\lvert I\rvert )}\,.
\end{equation}

In seeking to verify the Carleson measure condition \eqref{e.embed2}, there
is a store of common reductions. A very simple one is that it suffices to
test \eqref{e.embed2} for dyadic intervals $J$ for which $a_J \neq 0$.

A slightly more complicated one is this: Let $\mathcal{S}\subset \mathcal{D}$
be such that we have the estimate 
\begin{equation}
\sum_{S\in \mathcal{S}:S\subset S_{0}}\mu (S)\lesssim C_{1}\sigma
(S_{0})\,,\qquad S_{0}\in \mathcal{S}\,.  \label{e.wS}
\end{equation}%
That is, the measure $\sigma $ has the Carleson measure property, provided
one only sums intervals in $\mathcal{S}$. Now, suppose that $\mu $ is a
measure on $\mathbb{R}_{+}^{2}$ such that for any $S_{0}\in \mathcal{S}$, we
have 
\begin{equation}
\mu \left( \widehat{S_{0}}\backslash \bigcup_{S\in \mathcal{S}:S\subsetneq
S_{0}}\widehat{S}\right) \lesssim C_{2}\sigma (S_{0})\,.  \label{e.wS1}
\end{equation}%
Here, we have the box over $S_{0}$, and we remove the smaller boxes. Then,
we have $\mu (\widehat{S})\leq \left( C_{1}+C_{2}\right) \sigma (S)$ for all 
$S\in \mathcal{S}$. We shall implicitly use this reduction.

\medskip In the two weight setting, a paraproduct would be, for example, an
operator of the form 
\begin{equation}
Tf=\sum_{I\in \mathcal{D}}\alpha _{I}\mathbb{E}_{I}^{\sigma }f\cdot
h_{I}^{\omega }\,.  \label{e.para}
\end{equation}%
By the orthogonality of the Haar system, it follows that $T$ maps $%
L^{2}(\sigma )$ into $L^{2}(\omega )$ if and only if the sequence of square
coefficients $\{\alpha _{I}^{2}\;:\;I\in \mathcal{D}\}$ satisfies the
condition \eqref{e.embed2}. This is the type of argument we will be
appealing to below.

\section{The Good-Bad Decomposition} \label{s.good}

Here we follow the random grid idea of Nazarov, Treil and Volberg as set out
for example in Chapter 17 of \cite{Vol}. The first step in the proof is to
obtain two grids, one for each weight, that work well with each other. There
are in fact many dyadic grids in $\mathbb{R}$. For any $\beta =\{\beta
_{l}\}\in \{0,1\}^{\mathbb{Z}}$, define the dyadic grid ${\mathbb{D}}_{\beta
}$ to be the collection of intervals 
\begin{equation*}
{\mathbb{D}}_{\beta }=\left\{ 2^{n}\left( [0,1)+k+\sum_{i<n}2^{i-n}\beta
_{i}\right) \right\} _{n\in {\mathbb{Z}},\,k\in {\mathbb{Z}}}
\end{equation*}%
This parametrization of dyadic grids appears explicitly in \cite{MR2464252},
and implicitly in \cite{MR1998349} {section 9.1}. Place the usual uniform
probability measure $\mathbb{P}$ on the space $\{0,1\}^{\mathbb{Z}}$,
explicitly 
\begin{equation*}
\mathbb{P}(\beta :\beta _{l}=0)=\mathbb{P}(\beta :\beta _{l}=1)=\frac{1}{2}%
,\qquad \text{for all }l\in \mathbb{Z},
\end{equation*}%
and then extend by independence of the $\beta _{l}$. Note that the \emph{%
endpoints} and \emph{centers} of the intervals in the grid ${\mathbb{D}}%
_{\beta }$ are contained in $\mathbb{Q}^{dy}+x_{\beta }$, the dyadic
rationals $\mathbb{Q}^{dy}\equiv \left\{ \frac{m}{2^{n}}\right\} _{m,n\in 
\mathbb{Z}}$ translated by $x_{\beta }\equiv \sum_{i<0}2^{i}\beta _{i}\in %
\left[ 0,1\right] $. Moreover the pushforward of the probability measure $%
\mathbb{P}$ under the map $\beta \rightarrow x_{\beta }$ is Lebesgue measure
on $\left[ 0,1\right] $. The locally finite weights $\omega ,\sigma $ have
at most countably many point masses, and it follows with probability one
that $\omega ,\sigma $ do \emph{not} charge an endpoint or center of any
interval in ${\mathbb{D}}_{\beta }$.

For a weight $\omega $, we consider a random choice of dyadic grid $\mathcal{%
D}^{\omega }$ on the probability space $\Sigma ^{\omega }$, and likewise for
second weight $\sigma $, with a random choice of dyadic grid $\mathcal{D}%
^{\sigma }$ on the probability space $\Sigma ^{\sigma }$.

\begin{notation}
We fix $\varepsilon >0$ for use throughout the remainder of the paper.
\end{notation}

\begin{definition}
\label{d.good} For a positive integer $r$, an interval $J\in \mathcal{D}%
^{\sigma }$ is said to be \emph{\ $r$-bad} if there is an interval $I\in 
\mathcal{D}^{\omega }$ with $\lvert I\rvert \geq 2^{r}\lvert I\rvert $, and 
\begin{equation*}
\operatorname{dist}(e(I),J)\leq \tfrac{1}{2}\lvert J\rvert ^{\varepsilon }\lvert
I\rvert ^{1-\varepsilon }\,.
\end{equation*}%
Here, $e(J)$ is the set of three points consisting of the two endpoints of $%
J $ and its center. (This is the set of discontinuities of $h_{J}^{\sigma }$%
.) Otherwise, $J$ is said to be \emph{$r$-good}. We symmetrically define $%
J\in \mathcal{D}^{\omega }$ to be \emph{$r$-good}.
\end{definition}

The basic proposition here is:

\begin{proposition}
Fix a grid $\mathcal{D}^{\omega }$ and $J\in \mathcal{D}^{\omega }$. Then $%
\mathbb{P}\left( J\text{ is }r\text{-bad}\right) \leq C2^{-\varepsilon r}$.
\end{proposition}

\begin{proof}
Let $I\in \mathcal{D}^{\sigma }$ with the same length as $J$ and $\lvert
I\cap J\rvert \geq \tfrac{1}{2}\lvert J\rvert $. Let $s=\lfloor \left(
1-\varepsilon \right) r\rfloor $ and consider the $s$-fold ancestor $\pi _{%
\mathcal{D}^{\sigma }}^{s}I$ of $I$ in the grid $\mathcal{D}^{\sigma }$. We
have%
\begin{equation*}
\operatorname{dist}(e\left( \pi _{\mathcal{D}^{\sigma }}^{s}I\right) ,J)\leq
2^{s}\lvert J\rvert \leq \lvert J\rvert ^{\varepsilon }\lvert \pi _{\mathcal{%
D}^{\sigma }}^{s}I\rvert ^{1-\varepsilon }\,.
\end{equation*}%
In order that%
\begin{equation}
\operatorname{dist}(e(\pi _{\mathcal{D}^{\sigma }}^{r}I),J)\leq \tfrac{1}{2}\lvert
J\rvert ^{\varepsilon }\lvert \pi _{\mathcal{D}^{\sigma }}^{r}I\rvert
^{1-\varepsilon },  \label{exact r bad}
\end{equation}%
it would then be required that all of the further ancestors of $I$ up to $%
\pi _{\mathcal{D}^{\sigma }}^{r}I$, namely $\pi _{\mathcal{D}^{\sigma
}}^{s+1}I,\dotsc ,\pi _{\mathcal{D}^{\sigma }}^{r}I$, \emph{must share a
common endpoint.} Indeed, if not there is $1\leq \ell \leq r-s$ such that%
\begin{eqnarray*}
\operatorname{dist}(e(\pi _{\mathcal{D}^{\sigma }}^{r}I),J) &\geq &\operatorname{dist}(e(\pi
_{\mathcal{D}^{\sigma }}^{s+\ell }I),J)\geq \frac{1}{2}\left\vert \pi _{%
\mathcal{D}^{\sigma }}^{s+\ell }I\right\vert =2^{s+\ell -1}\left\vert
I\right\vert \\
&\geq &2^{s+1}\left\vert I\right\vert >\left\vert J\right\vert ^{\varepsilon
}2^{\left( 1-\varepsilon \right) r}\left\vert I\right\vert ^{1-\varepsilon
}\geq \left\vert J\right\vert ^{\varepsilon }\left\vert \pi _{\mathcal{D}%
^{\sigma }}^{r}I\right\vert ^{1-\varepsilon }.
\end{eqnarray*}%
The essential point about the random construction of the grids used here is
that for any interval $K$, $K$ is equally likely to be the left or right
half of its parent, and the selection of parents is done independently. But
sharing a common endpoint means that $\pi _{\mathcal{D}^{\sigma }}^{t}I$ has
to be the left-half, say, of $\pi _{\mathcal{D}^{\sigma }}^{t+1}I$, for all $%
t=s,\dotsc ,r-1$. So the probability that \eqref{exact r bad} holds is at
most $2^{-r+s+2}\leq 2^{-\varepsilon r+2}$. Now by definition, $J$ is $r$%
-bad if at least \emph{one} of the ancestors $\left\{ \pi _{\mathcal{D}%
^{\sigma }}^{r+k}I\right\} _{k=0}^{\infty }$ at or beyond $\pi _{\mathcal{D}%
^{\sigma }}^{r}I$ satisfies 
\begin{equation}
\operatorname{dist}(e(\pi _{\mathcal{D}^{\sigma }}^{r+k}I),J)\leq \tfrac{1}{2}\lvert
J\rvert ^{\varepsilon }\lvert \pi _{\mathcal{D}^{\sigma }}^{r+k}I\rvert
^{1-\varepsilon }.  \label{exact r+k bad}
\end{equation}%
The argument above shows that the probability that \eqref{exact r+k bad}
holds is at most $2^{-\varepsilon \left( r+k\right) +2}$. It follows that
the probability that $J$ is $r$-bad is at most%
\begin{equation*}
\sum_{k=0}^{\infty }2^{-\varepsilon \left( r+k\right) +2}=2^{-\varepsilon
r+2}\sum_{k=0}^{\infty }2^{-\varepsilon k}=\frac{2^{-\varepsilon r+2}}{%
1-2^{-\varepsilon }}\leq C_{\varepsilon }2^{-\varepsilon r}.
\end{equation*}%
which proves the Proposition.
\end{proof}

We restate the previous Proposition in a new setting. Let $\mathcal{D}%
^{\sigma }$ be randomly selected, with parameter $\beta $, and $\mathcal{D}%
^{\omega }$ with parameter $\beta ^{\prime }$. Define a projection 
\begin{equation}
\mathsf{P}_{\textup{good}}^{\sigma }f\equiv \sum_{I\text{ is }r\text{-good }%
\in \mathcal{D}^{\sigma }}\Delta _{I}^{\sigma }f\,,  \label{e.Pgood}
\end{equation}%
and likewise for $\mathsf{P}_{\textup{good}}^{\omega }\phi $. We define $%
\mathsf{P}_{\textup{bad}}^{\sigma }f\equiv f-\mathsf{P}_{\textup{good}%
}^{\sigma }f$. The basic Proposition is:

\begin{proposition}
\label{p.Pgood}(Theorem 17.1 in \cite{Vol}) We have the estimates 
\begin{equation*}
\mathbb{E}_{\beta ^{\prime }}\left\Vert \mathsf{P}_{\textup{bad}}^{\sigma
}f\right\Vert _{L^{2}(\sigma )}\leq C2^{-\frac{\varepsilon r}{2}}\left\Vert
f\right\Vert _{L^{2}(\sigma )}.
\end{equation*}%
and likewise for $\mathsf{P}_{\textup{bad}}^{\omega }\phi $.
\end{proposition}

\begin{proof}
We have 
\begin{align*}
\mathbb{E}_{\beta ^{\prime }}\left\Vert \mathsf{P}_{\textup{bad}}^{\sigma
}f\right\Vert _{L^{2}(\sigma )}^{2}& =\mathbb{E}_{\beta ^{\prime }}\sum_{I%
\text{ is }r\text{-bad}}\left\langle f,h_{I}^{\sigma }\right\rangle ^{2} \\
& \leq C2^{-\varepsilon r}\sum_{I}\left\langle f,h_{I}^{\sigma
}\right\rangle ^{2}=C2^{-\varepsilon r}\left\Vert f\right\Vert
_{L^{2}(\sigma )}^{2}\,.
\end{align*}
\end{proof}

From this we conclude the following: There is an absolute choice of $r$ so
that the following holds. Let $T\;:\;L^{2}(\sigma )\rightarrow L^{2}(\omega
) $ be a bounded linear operator. We then have 
\begin{equation}
\left\Vert T\right\Vert _{L^{2}(\sigma )\rightarrow L^{2}(\omega )}\leq
2\sup_{\left\Vert f\right\Vert _{L^{2}(\sigma )}=1}\sup_{\left\Vert \phi
\right\Vert _{L^{2}(\omega )}=1}\mathbb{E}_{\beta }\mathbb{E}_{\beta
^{\prime }}\lvert \left\langle T\mathsf{P}_{\textup{good}}^{\sigma }f,%
\mathsf{P}_{\textup{good}}^{\omega }\phi \right\rangle _{\omega }\rvert \,.
\label{e.Tgood}
\end{equation}%
Indeed, we can choose $f\in L^{2}(\sigma )$ of norm one, and $\phi \in
L^{2}(\omega )$ of norm one, and we can write 
\begin{equation*}
f=\mathsf{P}_{\textup{good}}^{\sigma }f+\mathsf{P}_{\textup{bad}}^{\sigma
}f\,
\end{equation*}%
and similarly for $\phi $, so that 
\begin{align*}
\left\Vert T\right\Vert _{L^{2}(\sigma )\rightarrow L^{2}(\omega )}&
=\left\langle Tf,\phi \right\rangle _{\omega } \\
& \leq \mathbb{E}_{\beta }\mathbb{E}_{\beta ^{\prime }}\lvert \left\langle T%
\mathsf{P}_{\textup{good}}^{\sigma }f,\mathsf{P}_{\textup{good}}^{\omega
}\phi \right\rangle _{\omega }\rvert +\mathbb{E}_{\beta }\mathbb{E}_{\beta
^{\prime }}\lvert \left\langle T\mathsf{P}_{\textup{bad}}^{\sigma }f,%
\mathsf{P}_{\textup{good}}^{\omega }\phi \right\rangle _{\omega }\rvert \\
& \quad +\mathbb{E}_{\beta }\mathbb{E}_{\beta ^{\prime }}\lvert \left\langle
T\mathsf{P}_{\textup{good}}^{\sigma }f,\mathsf{P}_{\textup{bad}}^{\omega
}\phi \right\rangle _{\omega }\rvert +\mathbb{E}_{\beta }\mathbb{E}_{\beta
^{\prime }}\lvert \left\langle T\mathsf{P}_{\textup{bad}}^{\sigma }f,%
\mathsf{P}_{\textup{bad}}^{\omega }\phi \right\rangle _{\omega }\rvert \\
& \leq \mathbb{E}_{\beta }\mathbb{E}_{\beta ^{\prime }}\lvert \left\langle T%
\mathsf{P}_{\textup{good}}^{\sigma }f,\mathsf{P}_{\textup{good}}^{\omega
}\phi \right\rangle _{\omega }\rvert +3C\cdot 2^{-r/16}\left\Vert
T\right\Vert _{L^{2}(\sigma )\rightarrow L^{2}(\omega )}\ .
\end{align*}%
And this proves \eqref{e.Tgood} for $r$ sufficiently large depending on $%
\varepsilon >0$.

This has the following implication for us: \emph{It suffices to consider
only $r$-good intervals, and prove an estimate for $\left\Vert H(\sigma
\cdot )\right\Vert _{L^{2}(\sigma )\rightarrow L^{2}(\omega )}$ that is
independent of this assumption.} Accordingly, we will call $r$-good
intervals just \emph{good intervals} from now on.

\section{Main Decomposition}

\label{s.MD}

Fix (large) intervals $I^{0}\in \mathcal{D}^{\sigma }$ and $J^{0}\in 
\mathcal{D}^{\omega }$, and consider the following modification of the
`good' projections 
\begin{equation*}
\mathsf{P}_{\textup{good},I^{0}}^{\sigma }f\equiv \sum_{I\in \mathcal{D}%
^{\sigma }\,,I\subset I^{0}\text{ and }\lvert I\rvert \leq 2^{-r}\lvert
I^{0}\rvert \,,I\textup{\ good}}\Delta _{I}^{\sigma }f\,.
\end{equation*}%
Likewise, define $\mathsf{P}_{\textup{good},J^{0}}^{\omega }\phi $ as
above. We will prove that 
\begin{equation}
\left\vert \left\langle H(\sigma \mathsf{P}_{\textup{good},I^{0}}^{\sigma
}f),\mathsf{P}_{\textup{good},J^{0}}^{\omega }\phi \right\rangle _{\omega
}\right\vert \lesssim \max \{\mathcal{A}_{2},\,\mathcal{H},\,\mathcal{H}%
^{\ast },\,\mathcal{F}_{\gamma,\varepsilon},\,\mathcal{F}_{\gamma,\varepsilon}^{\ast }\}\left\Vert
f\right\Vert _{L^{2}(\sigma )}\left\Vert \phi \right\Vert _{L^{2}(\omega
)}\,.  \label{e.prove1}
\end{equation}%
As this estimate will hold for all $I^{0},J^{0}$, and all joint shifts of $%
\mathcal{D}^{\sigma }$ and $\mathcal{D}^{\omega }$ that avoid point masses
at the boundary of intervals, this is sufficient to derive the Main Theorem %
\ref{main}. We also use here that the constant terms associated with the
initial intervals $I^{0}$ and $J^{0}$ in the expansion of $f$ and $\phi $
respectively can be handled by the weak boundedness condition \eqref{wbc}.
This means we can assume the expectations $\mathbb{E}_{I_{0}}^{\sigma }f$
and $\mathbb{E}_{J_{0}}^{\omega }\phi $ both vanish.

\begin{figure}[tbp]
\begin{center}
\begin{tikzpicture}
[level distance=20mm][>=open triangle 90] 
\node[shape aspect=2,diamond,draw]  { $ \bigl\langle \operatorname H (\sigma  f), \phi  \bigr\rangle _{\omega  }$ }
  child {node[shape aspect=2,diamond,draw] (1) {$ A^1_2$} }
  child[missing] {node {}}	 
   child[missing] {node {}}	 
  child[missing] {node {}}
  child {node[shape aspect=2,diamond,draw] (2) {$ A^1_1$}
  		 child {node[shape aspect=2,diamond,draw] {$ A^2_3$} 
	   			child {node[shape aspect=2,diamond,draw] {$ A^3_2$} 
	   				child  {node[rectangle,draw]  {$ A ^{4} _{1}$} 	edge from parent node[left,above] {$ \mathcal A_2$}  }
	   			 child[missing] {node {}}
	   			child {node[shape aspect=2,diamond,draw] {$ A^4_2$} %
child  {node[rectangle,draw]  {$ A ^{5} _{1}$} edge from parent node[left] {$ \mathcal H+ \mathcal F_{\gamma,\varepsilon}$}  }
 child[missing] {node {}}	   			
  child[missing] {node {}}	   			
child {node[shape aspect=2,diamond,draw] {$ A^5_2$}
	child  {node[shape aspect=2,diamond,draw]  {$ A ^{6} _{1}$} 
	 	child{node[shape aspect=2,rectangle,draw]    {$ A ^{7} _{3}$} 	edge from parent node[left] {$ \mathcal F_{\gamma,\varepsilon}$}  node[left, sloped, below] {$_{ 0 <\gamma}$ }}
	 	child[missing] {node {}}
	 	child  {node[shape aspect=2,rectangle,draw]   {$ A ^{7} _{4}$} 	edge from parent node[right] {$ \mathcal F_{\gamma,\varepsilon}$} node[left, sloped, below] {${}_{\gamma >0} $ } } 
	 	}   
	child[missing] {node {}}
	child[missing] {node {}}
	child  {node[shape aspect=2,diamond,draw] (62) {$ A ^{6} _{2}$}
		child  {node[rectangle,draw]  {$ A ^{7} _{1}$} edge from parent node[left] {$ \mathcal H + \mathcal F_{\gamma,\varepsilon}  $}  }	
		  child[missing] {node {}}
		child {node[rectangle,draw]  {$ A ^{7}_2$}  edge from parent node[right] {$ \mathcal F_{\gamma,\varepsilon}  $}  }
					} 
	 child[missing] {node {}}					
	child  {node[rectangle,draw]  {$ A ^{6} _{3}$}   edge from parent node[right,above] {$ \mathcal F_{\gamma,\varepsilon}$   }  node[left, sloped, below] {${}_{\gamma >0} $ }  }
				edge from parent node[right, midway] {Paraproducts} 	}  
	   			edge from parent node[right,near end] {Corona} 	} 
	   				 child[missing] {node {}}
	   				 child[missing] {node {}}	   			
	   				child  {node[rectangle,draw]  {$ A ^{4} _{3}$} edge from parent node[right,above] {$ \mathcal F_{\gamma,\varepsilon}$}   node[left, sloped, below] {${}_{\gamma >0} $ }} } 
  		 		  child[missing] {node {}}
	   			child  {node[rectangle,draw]  {$ A ^{3} _{1}$} edge from parent node[right] {$  \mathcal A_2$}  }
	       	}   
	   	  child[missing] {node {}}	
  		child  {node[rectangle,draw]  {$ A ^{2} _{1}$}
  		 edge from parent node[right] {$ \mathcal W$}
	       }
	         child[missing] {node {}}
	       child  {node[rectangle,draw]  {$ A ^{2} _{2}$}
  		 edge from parent node[right] {$ \mathcal A_2$}
	       }
  	  }; 
  	   \draw[thick,dotted,thick,->>] (1) -- (2) node[above,midway] {duality};
\end{tikzpicture}  
\end{center}
\caption{The flow chart of the decomposition of the inner product $%
\left\langle H\left( \protect\sigma f\right) ,\protect\phi \right\rangle _{%
\protect\omega }$. }
\label{f.scheme}
\end{figure}
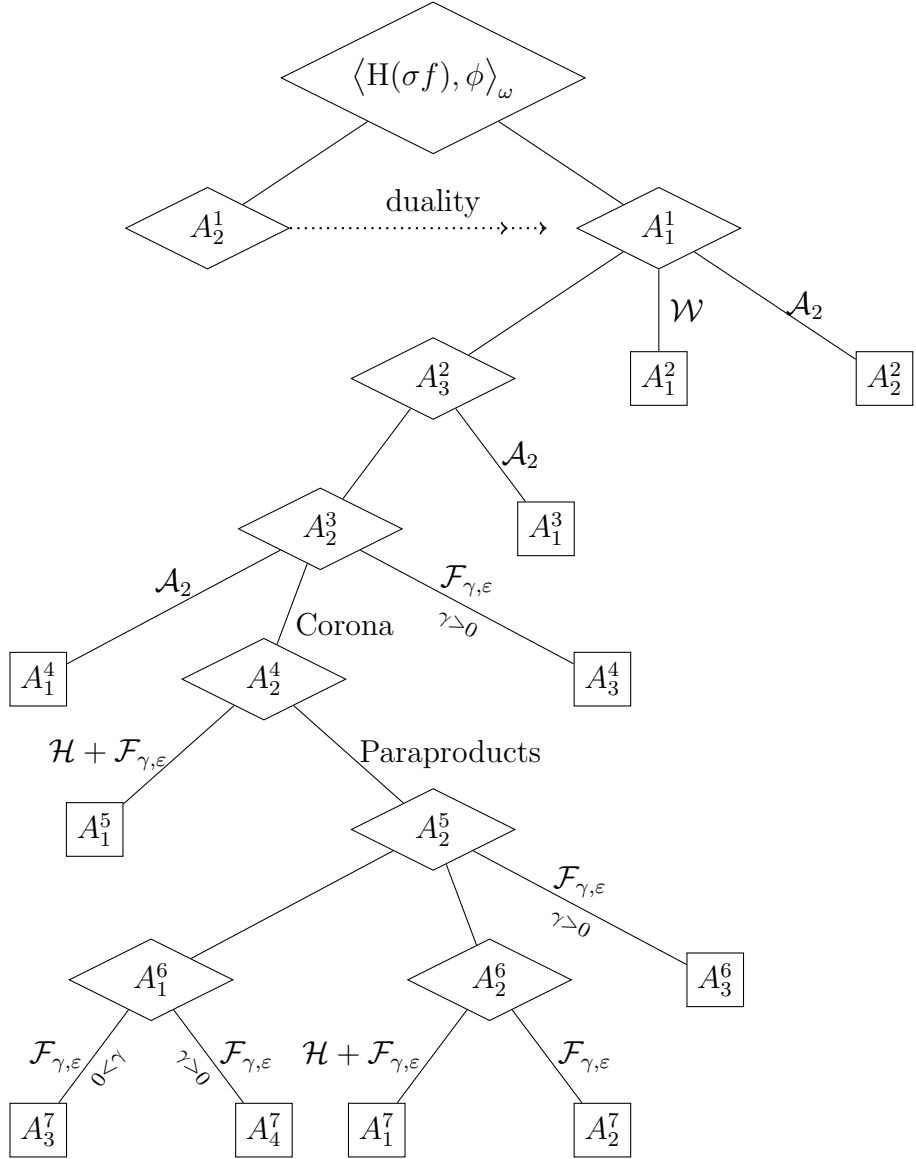


We may assume that $\mathsf{P}_{\textup{good},I^{0}}^{\sigma }f=f$, and
likewise for $\phi $. \textbf{From this point forward, we will only consider
good intervals $I,J$. } We suppress this dependence in the notation and we
will clearly note the use of this hypothesis when it arises. Similarly we
will only consider intervals $I,J$ that contribute to the definition of the
projections $\mathsf{P}_{\textup{good},I^{0}}^{\sigma }$ and $\mathsf{P}_{%
\textup{good},J^{0}}^{\omega }$, and suppress this fact in the notation.
The role of $I^{0}$ and $J^{0}$ permit the recursive constructions of the 
\emph{stopping intervals} in Definition~\ref{d.stopping} below.

Now, the inner product in \eqref{e.prove1} is 
\begin{align*}
\sum_{I\in \mathcal{D}^{\sigma }}\sum_{J\in \mathcal{D}^{\omega
}}\left\langle H(\sigma \Delta _{I}^{\sigma }f),\Delta _{J}^{\omega }\phi
\right\rangle _{\omega }& =\sum_{I\in \mathcal{D}^{\sigma }}\sum_{J\in 
\mathcal{D}^{\omega }}\left\langle f,h_{I}^{\sigma }\right\rangle _{\sigma
}\left\langle H(\sigma h_{I}^{\sigma }),h_{J}^{\omega }\right\rangle
_{\omega }\left\langle \phi ,h_{J}^{\omega }\right\rangle _{\omega } \\
& =A_{1}^{1}+A_{2}^{1}\,,
\end{align*}%
where $\mathcal{A}_{1}^{1}\equiv \{(I,J)\in \mathcal{D}^{\sigma }\times 
\mathcal{D}^{\omega }\;:\;\lvert J\rvert \leq \lvert I\rvert \}$, and we use
the notation 
\begin{equation*}
A_{j}^{i}\equiv \sum_{(I,J)\in \mathcal{A}_{j}^{i}}\left\langle
f,h_{I}^{\sigma }\right\rangle _{\sigma }\left\langle H(\sigma h_{I}^{\sigma
}),h_{J}^{\omega }\right\rangle _{\omega }\left\langle \phi ,h_{J}^{\omega
}\right\rangle _{\omega }\,.
\end{equation*}%
The term $A_{2}^{1}$ is the complementary sum. The sums are estimated
symmetrically. Thus it suffices to prove \eqref{e.prove1} for the sum $%
A_{1}^{1}$. Indeed, the starred constants do not enter into this estimate,
but by duality will enter into those for $A_{2}^{1}$, in which the roles of $%
\omega $ and $\sigma $ are reversed.

We shall follow the argument outlined in Chapters 18-22 of \cite{Vol} by
making several more decompositions, generating a number of terms $A_{j}^{i}$%
. These will be bilinear forms, but we will suppress the dependence of these
forms on the functions $f$ and $\phi $. In this notation, the superscript $i$
denotes the generation of the decomposition, and we will go to seven
generations. The subscript $j$ counts the number of decompositions in a
generation. To aid the reader's understanding of the argument, a flow chart
of the decompositions is given in Figure~\ref{f.scheme}. It contains
information about the proof, which we describe here.

\begin{itemize}
\item The chart is read from top to bottom, with the root of the chart
containing the inner product $\left\langle H(\sigma f),\phi \right\rangle
_{\omega }$.

\item Terms in diamonds are further decomposed, while terms in rectangles
are final estimates. The edges leading into rectangles are labeled by the
hypotheses used to control them, $\mathcal{A}_{2}$, $\mathcal{H}$, $\mathcal{%
W}$, or $\mathcal{F}_{\gamma,\varepsilon}$ in the figure.

\item There are three terms, $A_{3}^{4}$, $A_{3}^{6}$ and the two from $%
A_{1}^{6}$ for which the Energy Hypothesis with $\gamma >0$ is essential.
The edges leading into these terms are labeled to indicate this.

\item The horizontal dotted arrow from $A_{2}^{1}$ to $A_{1}^{1}$, labeled
`duality', indicates that $A_{2}^{1}$ is controlled by the argument for $%
A_{1}^{1}$, after exchanging the roles of $f$ and $\phi $. Accordingly, the
final estimates in the dual tree will be in terms of the dual hypotheses,
namely $\mathcal{F}_{\gamma,\varepsilon}^{\ast }$ and $\mathcal{H}^{{\ast }}$.

\item The edge leading into $A_{2}^{4}$ is labeled `Corona' to indicate that
the Corona decomposition of \S \ref{s.stopping} is used at this point. This
is an important stage in the decomposition and one of the key ideas in \cite%
{NTV3}. We modify it with the use of our Energy Hypothesis.

\item The edge leading into $A_{2}^{5}$ is labeled `Paraproducts' as all of
the estimates in the fifth and subsequent generations use paraproduct
arguments to control them. See \S \ref{s.paraproducts}. We organize the
written proof to pass to the paraproducts first, with the other estimates
taken up second in \S \ref{s.remain}.
\end{itemize}

We return to the main line of the proof of sufficiency. The collection of
pairs of intervals $\mathcal{A}_{1}^{1}$ is decomposed into the collections%
\begin{equation}
\mathcal{A}_{1}^{2}\equiv \{(I,J)\in \mathcal{A}_{1}^{1}\;:\;2^{-r}\lvert
I\rvert \leq \lvert J\rvert \leq \lvert I\rvert \,,\textup{dist}(I,J)\leq
\lvert I\rvert \}\,,  \label{A21}
\end{equation}%
\begin{equation}
\mathcal{A}_{2}^{2}\equiv \{(I,J)\in \mathcal{A}_{1}^{1}\;:\;\lvert J\rvert
\leq \lvert I\rvert \,,\textup{dist}(I,J)>\lvert I\rvert \}\,,  \label{A22}
\end{equation}%
\begin{equation}
\mathcal{A}_{3}^{2}\equiv \{(I,J)\in \mathcal{A}_{1}^{1}\;:\;\lvert J\rvert
<2^{-r}\lvert I\rvert \,,\textup{dist}(I,J)\leq \lvert I\rvert \}\,.
\label{A23}
\end{equation}%
Using the notation of \cite{Vol} and \cite{NTV3}, we refer to $\mathcal{A}%
_{1}^{2}$ as the `diagonal short-range' terms; $\mathcal{A}_{2}^{2}$ are the
`long-range' terms; and $\mathcal{A}_{3}^{2}$ are the `short-range' terms.

We will show in \S \ref{s.21} and \S \ref{s.22}, respectively,%
\begin{equation}
\lvert A_{1}^{2}\rvert \lesssim \mathcal{H}\left\Vert f\right\Vert
_{L^{2}(\sigma )}\left\Vert \phi \right\Vert _{L^{2}(\omega )}\,,
\label{A21<}
\end{equation}%
\begin{equation}
\lvert A_{2}^{2}\rvert \lesssim \mathcal{A}_{2}\left\Vert f\right\Vert
_{L^{2}(\sigma )}\left\Vert \phi \right\Vert _{L^{2}(\omega )}\,.
\label{A22<}
\end{equation}%
These inequalities are obtained in Chapters 18 and 19 of \cite{Vol}.

The term $\mathcal{A}_{3}^{2}$ is the important one, and will be further
decomposed into%
\begin{gather}
\mathcal{A}_{1}^{3}\equiv \{(I,J)\in \mathcal{D}^{\sigma }\times \mathcal{D}%
^{\omega }\;:\;\lvert J\rvert <2^{-r}\lvert I\rvert \,,I\cap J=\emptyset \,,%
\textup{dist}(I,J)\leq \lvert I\rvert \}  \label{e.a31} \\
\mathcal{A}_{2}^{3}\equiv \{(I,J)\in \mathcal{D}^{\sigma }\times \mathcal{D}%
^{\omega }\;:\;\lvert J\rvert <2^{-r}\lvert I\rvert \,,I\cap J\neq \emptyset
\}\,.  \label{e.a32}
\end{gather}%
The `mid-range' term $A_{1}^{3}$ will be handled by a variant of the method
used on the `long-range' term, along with the $A_{2}$ condition. In
particular, in \S \ref{s.31} we prove 
\begin{equation}
\lvert A_{1}^{3}\rvert \lesssim \mathcal{F}_{\gamma ,\varepsilon }\left\Vert
f\right\Vert _{L^{2}(\sigma )}\left\Vert \phi \right\Vert _{L^{2}(\omega
)}\,.  \label{e.a31<}
\end{equation}%
Thus, $A_{2}^{3}$ is the true `short-range' term. It is imperative to
observe that for $(I,J)\in \mathcal{A}_{2}^{3}$, we must have $J\subset I$,
for otherwise we violate the fact that $J$ is good.

\section{Energy, Stopping Intervals, Corona Decomposition}

\label{s.stopping}

Our focus is on the short-range term, as given by \eqref{e.a32}, and it is
here that our Energy Condition \eqref{energy condition} will arise in place
of the Pivotal Condition in \cite{NTV3}. This is a critical section in this
proof, and it has three purposes. First, to derive the Energy Lemma, and
combine it with the Energy Hypothesis. Second, to make the definition of the
Corona. Third, use the Corona to obtain the next stage in the decomposition
of the short-range term.


\subsection{The Energy Lemma}

\label{s.EL}

As is typical in proofs that involve identification of a paraproduct, one
should add and subtract cancellative terms, in order that the paraproducts
become more apparent. Take a pair $(I,J)\in \mathcal{A}_{2}^{3}$. Thus, $%
J\cap I\neq \emptyset $ and $\lvert J\rvert \leq 2^{-r}\lvert I\rvert $. But 
$J$ is good, so that we have $J\subset I$, but not only that, we have 
\begin{equation}
\textup{dist}(e(I),J)\geq \lvert J\rvert ^{\varepsilon }\lvert I\rvert
^{1-\varepsilon }\,.  \label{e.goodFar}
\end{equation}%
Let $I_{J}$ be the child of $I$ that contains $J$, and let $\widehat{I}$
denote some ancestor of $I_{J}$. (The specific sequence of ancestors is to
be selected below, in Definition~\ref{d.stopping}.) We write 
\begin{align}
\left\langle H(\sigma \Delta _{I}^{\sigma }f),\Delta _{J}^{\omega }\phi
\right\rangle _{\omega }& =\left\langle H(\mathbf{1}_{I\backslash
I_{J}}\sigma \Delta _{I}^{\sigma }f),\Delta _{J}^{\omega }\phi \right\rangle
_{\omega }+\left\langle H(\mathbf{1}_{I_{J}}\sigma \Delta _{I}^{\sigma
}f),\Delta _{J}^{\omega }\phi \right\rangle _{\omega } \\
& =\left\langle H(\mathbf{1}_{I\backslash I_{J}}\sigma \Delta _{I}^{\sigma
}f),\Delta _{J}^{\omega }\phi \right\rangle _{\omega }  \label{e.neighbor} \\
& \qquad +\mathbb{E}_{I_{J}}^{\sigma }\Delta _{I}^{\sigma }f\cdot
\left\langle H(\sigma \mathbf{1}_{\widehat{I}}),\Delta _{J}^{\omega }\phi
\right\rangle _{\omega }  \label{e.energy} \\
& \qquad -\mathbb{E}_{I_{J}}^{\sigma }\Delta _{I}^{\sigma }f\cdot
\left\langle H(\sigma \mathbf{1}_{\widehat{I}\backslash I_{J}}),\Delta
_{J}^{\omega }\phi \right\rangle _{\omega }  \label{e.stopping}
\end{align}%
As in \cite{Vol} and \cite{NTV3}, we refer to the three terms in the last
line as, respectively, the `neighbor' term, the `paraproduct' term (called
`middle' term in \cite{Vol} and \cite{NTV3}), and the `stopping' term. Note
that $\Delta _{I}^{\sigma }f$ takes a single value on $I_{J}$, which is
exactly $\mathbb{E}_{I_{J}}^{\omega }\Delta _{I}^{\sigma }f$.

Our analysis of the stopping term in \eqref{e.stopping} will bring forward
the energy condition \eqref{energy condition}, and yields a more general
inequality which we formulate in this Lemma. Recall that%
\begin{equation*}
\Phi \left( I,E\right) \equiv \omega \left( I\right) \mathsf{E}\left(
I,\omega \right) ^{2}\mathsf{P}\left( I,\mathbf{1}_{E}\sigma \right) ^{2}.
\end{equation*}

\begin{lemma}[\textbf{Energy Lemma}]
\label{energy lemma}Let $J\subset I^{\prime }\subset \widehat{I}$ be three
intervals with%
\begin{equation}
dist\left( \partial I^{\prime },J\right) \geq \lvert J\rvert  \label{e.Jgood}
\end{equation}%
(This follows from the  good property for dyadic intervals, but we do not assume that
any of these three intervals are dyadic.) Let $\Phi _{J}$ be a function
supported in $J$ and with $\omega $-integral zero. Then we have%
\begin{equation}
\left\vert \left\langle H\left( \mathbf{1}_{\widehat{I}\backslash I^{\prime
}}\sigma \right) ,\Phi _{J}\right\rangle _{\omega }\right\vert \leq
C\left\Vert \Phi _{J}\right\Vert _{L^{2}\left( \omega \right) }\Phi \left( J,%
\widehat{I}\setminus I^{\prime }\right) ^{\frac{1}{2}}.  \label{e.Eimplies}
\end{equation}
\end{lemma}

The $L^{2}$ formulation in \eqref{e.Eimplies} proves useful in many
estimates below, in particular in the proof of the Carleson measure
estimate, Theorem \ref{t.miraculous}. Indeed, we will apply %
\eqref{e.Eimplies} in the dual formulation. Namely, we have 
\begin{equation}
\left\Vert H\left( \mathbf{1}_{\widehat{I}\backslash I^{\prime }}\sigma
\right) -\mathbb{E}_{J}^{\omega }H\left( \mathbf{1}_{\widehat{I}\backslash
I^{\prime }}\sigma \right) \right\Vert _{L^{2}(J,\omega )}\lesssim \Phi
\left( J,\widehat{I}\setminus I^{\prime }\right) ^{\frac{1}{2}}.
\label{e.Edual}
\end{equation}%
Note that on the left, we are subtracting off the mean value, and only
testing the $L^{2}(\omega )$ norm on $J$.

\begin{proof}
For $x,x^{\prime }\in J$, and $y\in \widehat{I}\backslash I^{\prime }$, we
have the equality 
\begin{equation}
\frac{1}{x-y}-\frac{1}{x^{\prime }-y}=\frac{x-x^{\prime }}{\lvert J\rvert }%
\cdot \frac{\lvert J\rvert }{(x-y)(x^{\prime }-y)}  \label{e.g=}
\end{equation}%
We use \eqref{e.Jgood} to estimate the second term by 
\begin{equation}
\frac{\left\vert J\right\vert }{(x-y)(x^{\prime }-y)}\lesssim \frac{%
\left\vert J\right\vert }{\vert y-c_J \vert ^{2}},  \label{sec term}
\end{equation}%
where $c_J$ is the center of $J$.

Turning to the inner product, the fact that $\Phi _{J}$ is supported on $J$
and has $\omega $-mean zero permits us the usual cancellative estimate on
the kernel. This familiar argument requires the selection of an auxiliary
point in $J$, and we use the measure $\omega $ to select it. We have 
\begin{align}
\left\vert \left\langle H(\sigma \mathbf{1}_{\widehat{I}\backslash I^{\prime
}}),\Phi _{J}\right\rangle _{\omega }\right\vert & =\left\vert \left\langle
H(\sigma \mathbf{1}_{\widehat{I}\backslash I^{\prime }})-\mathbb{E}%
_{J}^{\omega }H(\sigma \mathbf{1}_{\widehat{I}\backslash I^{\prime }}),\Phi
_{J}\right\rangle _{\omega }\right\vert  \label{e.g<<} \\
& =\left\vert \int_{J}\int_{I\backslash I^{\prime }}\mathbb{E}_{J}^{\omega
(dx^{\prime })}\left( \frac{1}{x-y}-\frac{1}{x^{\prime }-y}\right) \Phi
_{J}(x)\;\sigma (dy)\;\omega (dx)\right\vert \\
& \lesssim \int_{J}\mathbb{E}_{J}^{\omega (dx^{\prime })}\frac{\lvert
x-x^{\prime }\rvert }{\lvert J\rvert }\lvert \Phi _{J}(x)\rvert \;\omega
(dx)\cdot \mathsf{P}(J,\mathbf{1}_{\widehat{I}\backslash I^{\prime }}\sigma )
\\
& \leq \left\Vert \Phi _{J}\right\Vert _{L^{2}(\omega )}\left(
\int_{J}\left( \mathbb{E}_{J}^{\omega (dx^{\prime })}\frac{\lvert
x-x^{\prime }\rvert }{\lvert J\rvert }\right) ^{2}\;\omega (dx)\right) ^{1/2}%
\mathsf{P}(J,\mathbf{1}_{\widehat{I}\backslash I^{\prime }}\sigma ) \\
& =\left\Vert \Phi _{J}\right\Vert _{L^{2}(\omega )}\Phi \left( J,\widehat{I}%
\setminus I^{\prime }\right) ^{\frac{1}{2}}\,.
\end{align}%
This completes the proof.
\end{proof}

\subsection{The Corona Decomposition}

We now make two important definitions from \cite{NTV3}: `stopping intervals'
and the `Corona Decomposition'. This is the main point of departure for our
proof. But first we recall the notation introduced in (\ref{def functionals}%
),  \eqref{def Psi gamma},
\begin{eqnarray*}
\Phi \left( I,E\right) &\equiv &\omega \left( I\right) \mathsf{E}\left(
I,\omega \right) ^{2}\mathsf{P}\left( I,\mathbf{1}_{E}\sigma \right) ^{2}, \\
\Psi _{\gamma ,\varepsilon }\left( I,E\right) &\equiv
&\sup_{I=\bigcup_{s\geq 1}J_{s}}\inf_{s\geq 1}\left[ \frac{\left\vert
I\right\vert }{\left\vert J_{s}\right\vert }\right] ^{\gamma }\times
\sum_{s\geq 1}\Phi \left( J_{s},E\right) ,
\end{eqnarray*}%
where the supremum is over all $\varepsilon $-good subpartitions $\left\{
J_{s}\right\} _{s\geq 1}$ of $I$, and the Energy Hypothesis \eqref{smallest condition}%
\begin{equation*}
\sum_{r\geq 1}\Psi _{\gamma ,\varepsilon }\left( I_{r},I_{0}\right) \leq 
\mathcal{F}_{\gamma,\varepsilon}^{2}\sigma (I_{0}),
\end{equation*}%
where $\gamma >0,\varepsilon >0$ are fixed. Recall that $\Phi $ appears in
the Energy Condition \eqref{energy condition} and in the dual Energy
Estimate \eqref{e.Edual}, while the larger functional $\Psi _{\gamma
,\varepsilon }$ appears in the Energy Hypothesis \eqref{smallest condition}.
The key properties required of $\Psi $ are given in \eqref{Psi properties},
and result in the crucial off-diagonal decay of $\Psi $ relative to $\Phi $
in Theorem \ref{t.miraculous} used to estimate term $A_{3}^{6}$, as well as
the estimates for the term $A_{1}^{6}$ and the stopping term $A_{3}^{4}$ in
Subsection \ref{s.43}.

\begin{definition}
\label{d.stopping} Given any interval $I_{0}$, set $\mathcal{S}(I_{0})$ to
be the maximal $\mathcal{D}^{\sigma }$ strict subintervals $S\subsetneqq
I_{0}$ such that%
\begin{equation}
\Psi _{\gamma ,\varepsilon }\left( S,I_{0}\right) \geq 4\mathcal{F}_{\gamma,  \varepsilon  }^{2}\sigma (S),  \label{e.stoppingDef}
\end{equation}%
The collection $\mathcal{S}(I_{0})$ can be empty.

We now recursively define $\mathcal{S}_{1}\equiv \{I^{0}\}$, and $\mathcal{S}%
_{j+1}\equiv \bigcup_{S\in \mathcal{S}_{j}}\mathcal{S}(S)$. The collection $%
\mathcal{S}\equiv \bigcup_{j=1}^{\infty }\mathcal{S}_{j}$ is the collection
of \emph{stopping intervals.} Define $\rho \;:\;\mathcal{S}\rightarrow 
\mathbb{N}$ by $\rho (S)=j$ for all $S\in \mathcal{S}_{j}$, so that $\rho
(S) $ denotes the `generation' in which $S$ occurs in the construction of $%
\mathcal{S}$.
\end{definition}

\begin{remark}
\label{r.stopping} It is worth emphasizing that we will \emph{not have} a
uniform inequality of the following nature available to us: 
\begin{equation*}
\Psi _{\gamma ,\varepsilon }\left( S,I_{0}\right)
\lesssim \sigma (S).
\end{equation*}%
In a similar, but different direction, one might be tempted to make the
simpler definition of a stopping interval that it is a maximal subinterval $%
S\subsetneqq I_{0}$ for which one has%
\begin{equation*}
\Phi (S, I_0) = 
\mathsf{E}(S,\omega )^{2}\mathsf{P}(S,\mathbf{1}_{I_{0}}\sigma )^{2}\omega
(S)\geq 4\mathcal{E}^{2}\sigma (S).
\end{equation*}%
This simpler condition does not permit one to fully exploit the Energy
Hypothesis.
\end{remark}

We now define the associated Corona Decomposition.

\begin{definition}
\label{d.corona} For $S\in \mathcal{S}$, we set $\mathcal{P }(S)$ to be all
the pairs of intervals $(I, J)$ such that 

\begin{enumerate}
\item $I\in \mathcal{D }^{\sigma }$, $J\in \mathcal{D }^{\omega }$, $%
J\subset I$, and $\lvert J\rvert< 2 ^{-r} \lvert I\rvert $.

\item $S$ is the $\mathcal{S}$-parent of $I _{J}$, the child of $I$ that
contains $J$.
\end{enumerate}

Note that $\mathcal{A }^{3} _{2} = \bigcup _{S\in \mathcal{S}} \mathcal{P }%
(S)$, where $\mathcal{A }^{3} _{2}$ is defined in \eqref{e.a32}. Let $%
\mathcal{C}^{\sigma }(S)$ to be all those $I\in \mathcal{D}^{\sigma }$ such
that $S$ is a minimal member of $\mathcal{S}$ that contains a $\mathcal{D }%
^{\sigma }$-child of $I$. (A fixed interval $I$ can be in two collections $%
\mathcal{C }^{\sigma } (S)$.) The definition of $\mathcal{C}^{\omega }(S)$
is similar but not symmetric: all those $J\in \mathcal{D} ^{\omega }$ such
that $S$ is the smallest member of $\mathcal{S}$ that contains $J$ and
satisfies $2^{r}\lvert J\rvert <\lvert S\rvert $. The collections $\{%
\mathcal{C}^{\sigma }(S)\;:\;S\in \mathcal{S}\}$ and $\{\mathcal{C}^{\omega
}(S)\;:\;S\in \mathcal{S}\}$ are referred to as the \emph{Corona
Decompositions}. Note that $\mathcal{S}\subset \mathcal{D}^{\sigma }$ and $%
\mathcal{C}^{\sigma }(S)\subset \mathcal{D}^{\sigma }$ for $S\in \mathcal{S}$
while $\mathcal{C}^{\omega }(S)\subset \mathcal{D}^{\omega }$ for $S\in 
\mathcal{S}$.

We denote the associated projections by 
\begin{equation}
\mathsf{P}_{S}^{\sigma }f\equiv \sum_{I\in \mathcal{C}^{\sigma
}(S)}\left\langle f,h_{I}^{\sigma }\right\rangle _{\sigma }h_{I}^{\sigma }\,,
\label{e.PS}
\end{equation}%
and similarly for $\mathsf{P}_{S}^{\omega }\phi $. Note that $\mathsf{P}%
_{S}^{\omega }$ projects only on intervals $J$ with $\lvert J\rvert
<2^{-r}\lvert S\rvert $.
\end{definition}

We have the estimate below that we will appeal to a few times. 
\begin{equation}
\sum_{S\in \mathcal{S}}\lVert \mathsf{P}_{S}^{\sigma }f\rVert _{L^{2}(\sigma
)}^{2}\leq 2\lVert f\rVert _{L^{2}(\sigma )}^{2}  \label{e.2}
\end{equation}%
There is a similar inequality for $\mathsf{P}_{S}^{\omega }$ which we will
also use.

\begin{remark}
\label{r.epsilon} In the definition of the stopping intervals, we are using
the functional $\Psi _{\gamma ,\varepsilon } $ associated with the Energy Hypothesis \eqref{smallest
condition}. Thus the stopping intervals can be viewed as the \emph{enemy}
in verifying \eqref{smallest condition}.
\end{remark}

\subsection{The Decomposition of the Short-Range Term}

To conclude this section, our estimate of $A_{2}^{3}$ defined in %
\eqref{e.a32} combines the splitting \eqref{e.neighbor}, \eqref{e.energy}, %
\eqref{e.stopping} and the Corona Decomposition. Namely, the Corona
Decomposition selects the intervals $\widehat{I}$ that appear in %
\eqref{e.neighbor}---\eqref{e.stopping} according to the following rule.
Recall that%
\begin{equation*}
\mathcal{A}_{2}^{3}\equiv \{(I,J)\in \mathcal{D}^{\sigma }\times \mathcal{D}%
^{\omega }\;:\;J\subset I\text{ and }\lvert J\rvert <2^{-r}\lvert I\rvert \}.
\end{equation*}%
Here, we have used the fact that $J$ is good to make the condition defining $%
\mathcal{A}_{2}^{3}$ more explicit, i.\thinspace e.\thinspace $J\subset I$
and $2^{r}\lvert J\rvert <\lvert I\rvert $.

\begin{definition}
Given a pair $\left( I,J\right) \in \mathcal{A}_{2}^{3}$, choose $\widehat{I}%
\in \mathcal{S}$ to be the unique stopping interval such that $I_{J}\in 
\mathcal{C}^{\sigma }( \widehat{I}) , $ where $I_{J}$ is the child of $I$
containing $J$. Equivalently, $\widehat{I}\in \mathcal{S}$ is determined by
the requirement that $\left( I,J\right) \in \mathcal{P}( \widehat{I}) $.
\end{definition}

Note that if $I_{J}\notin \mathcal{S}$, then $\widehat{I}\supset I$, while
if $I_{J}\in \mathcal{S}$, then $\widehat{I}$ is the child of $I$ containing 
$J$. Thus $\widehat{I}$ is a function of the \emph{pair} $\left( I,J\right) $%
. With this choice of $\widehat{I}$ in the splitting \eqref{e.neighbor}, %
\eqref{e.energy}, \eqref{e.stopping} we obtain $\lvert A_{2}^{3}\rvert \leq
\sum_{j=1}^{3}\lvert A_{j}^{4}\rvert $ where%
\begin{align}
A_{1}^{4}& \equiv \sum_{(I,J)\in \mathcal{A}_{2}^{3}}\left\langle H(\mathbf{1%
}_{I\backslash I_{J}}\sigma \Delta _{I}^{\sigma }f),\Delta _{J}^{\omega
}\phi \right\rangle _{\omega }  \label{e.Neighbor} \\
A_{2}^{4}& \equiv \sum_{S\in \mathcal{S}}\sum_{(I,J)\in \mathcal{P}(S)}%
\mathbb{E}_{I_{J}}^{\sigma }\Delta _{I}^{\sigma }f\cdot \left\langle H(%
\mathbf{1}_{S}\sigma ),\Delta _{J}^{\omega }\phi \right\rangle _{\omega }
\label{e.Energy} \\
A_{3}^{4}& \equiv \sum_{S\in \mathcal{S}}\sum_{(I,J)\in \mathcal{P}(S)}%
\mathbb{E}_{I_{J}}^{\sigma }\Delta _{I}^{\sigma }f\cdot \left\langle H(%
\mathbf{1}_{S\backslash I_{J}}\sigma ),\Delta _{J}^{\omega }\phi
\right\rangle _{\omega }  \label{e.Stopping}
\end{align}%
Recall that the three terms above are the neighbor, paraproduct, and
stopping terms respectively.

The paraproduct term $A_{2}^{4}$ is further decomposed, while we will prove
in \S \ref{s.41} and \S \ref{s.43} respectively,%
\begin{align}
\lvert A_{1}^{4}\rvert & \lesssim  \mathcal{A}_2 \left\Vert
f\right\Vert _{L^{2}(\sigma )}\left\Vert \phi \right\Vert _{L^{2}(\omega
)}\,,  \label{e.41<} \\
\lvert A_{3}^{4}\rvert & \lesssim \mathcal{F}_{\gamma,\varepsilon}\left\Vert
f\right\Vert _{L^{2}(\sigma )}\left\Vert \phi \right\Vert _{L^{2}(\omega
)}\,.  \label{e.43<}
\end{align}


\section{The Carleson Measure Estimates}

\label{s.carlesonMeasures}

This section is devoted to the statement and proof of several Carleson
measure estimates, designed with the considerations of the next section in
mind. We collect them here, due to the common sets of techniques used to
prove them.

The following technical Lemma encompasses many of the applications of the
Energy Hypothesis and the stopping time definition. For an interval $J\in 
\mathcal{D}^{\omega }$, let 
\begin{equation}
\widetilde{\mathsf{P}}_{J}^{\omega }\phi =\sum_{\substack{ J^{\prime }\in 
\mathcal{D}^{\omega }\;:\;J^{\prime }\subset J}}\langle \phi ,h_{J^{\prime
}}^{\omega}\rangle_{\sigma}h_{J^{\prime }} ^{\omega}.  \label{e.Ptilde}
\end{equation}%
Note that this projection $\widetilde{\mathsf{P}}_{J}^{\omega }$ is onto the
span of \emph{all} Haar functions $h_{J^{\prime }}$ supported in the $%
\mathcal{D}^{\omega }$-interval $J$. By contrast, $\mathsf{P}_{S}^{\omega }$
projects onto the span of those Haar functions $h_{J}$ with $J$ in the
corona $\mathcal{C}^{\omega }\left( S\right) $ where $S$ is a stopping
interval in the $\mathcal{D}^{\sigma }$ grid.

\begin{lemma}
\label{l.minimal} Fix an interval $I_{0}\in \mathcal{D}^{\sigma }$ and let $%
\widehat{I}_{0}\in \mathcal{S}$ be its $\mathcal{S}$-parent. Let $%
\{I_{r}\;:\;r\geq 1\}\subset \mathcal{D}^{\sigma }$ be a strict subpartition
of $I_{0}$. For $r\geq 1$, let $\{J_{r,s}\;:\;s\geq 1\}\subset \mathcal{D}%
^{\omega }$ be a subpartition of $I_{r}$ with $\lvert J_{r,s}\rvert
<2^{-t}\lvert I_{r}\rvert $ for all $r,s\geq 1$, where $t\geq r$ is the 
integer of Definition~\ref{d.good}. We then have 
\begin{equation}
\sum_{r,s\geq 1}\left\Vert \widetilde{\mathsf{P}}_{J_{r,s}}^{\omega }H(%
\mathbf{1}_{\widehat{I}_{0}\backslash I_{r}} \sigma )\right\Vert _{L^{2}(\omega
)}^{2}\lesssim 2^{-\gamma t}\mathcal{F}_{\gamma,\varepsilon}^{2}\sigma (I_{0})\,.
\label{e.minimal<}
\end{equation}
\end{lemma}



\begin{proof}
We apply \eqref{e.Edual}, and \eqref{smallest condition} to deduce the
Lemma. We begin with \eqref{e.Edual} to obtain%
\begin{eqnarray*}
\sum_{r,s\geq 1}\left\Vert \widetilde{\mathsf{P}}_{J_{r,s}}^{\omega }H(%
\mathbf{1}_{\widehat{I}_{0}\setminus I_{r}} \sigma )\right\Vert _{L^{2}(\omega
)}^{2} &\lesssim &\sum_{r,s\geq 1}\Phi \left( J_{r,s},\widehat{I}%
_{0}\setminus I_{r}\right)  \\
&\lesssim &\sum_{r,s\geq 1}\Phi \left( J_{r,s},\widehat{I}_{0}\setminus
I_{0}\right) +\sum_{r,s\geq 1}\Phi \left( J_{r,s},I_{0}\setminus
I_{r}\right) ,
\end{eqnarray*}%
where the last inequality follows from the definition of $\Phi $ and 
\begin{equation*}
\mathsf{P}\left( J,\mathbf{1}_{\widehat{I}_{0}\setminus I_{r}}\sigma \right)
= \mathsf{P}\left( J,\mathbf{1}_{\widehat{I}_{0}\setminus I_{0}}\sigma
\right) +\mathsf{P}\left( J,\mathbf{1}_{I_{0}\setminus I_{r}}\sigma \right) .
\end{equation*}%
If $I_{0}\neq \widehat{I}_{0}$, we estimate the sum involving $\widehat{I}%
_{0}\setminus I_{0}$ using the fact that $\left\{ J_{r,s}\right\} _{r,s\geq
1}$ is an $\varepsilon $-good subpartition of $I_{0}$ (because the intervals 
$J_{r,s}$ are good). We can thus use the third line in (\ref{Psi properties}%
), and then the fact that \eqref{e.stoppingDef} \emph{fails} when $I_{0}\neq 
\widehat{I}_{0}$, to obtain 
\begin{equation*}
\sum_{r,s\geq 1}\Phi \left( J_{r,s},\widehat{I}_{0}\setminus I_{0}\right)
\lesssim \left\{ \sup_{s\geq 1}\left( \frac{\left\vert J_{r,s}\right\vert }{%
\left\vert I_{0}\right\vert }\right) ^{\gamma }\right\} \Psi _{\gamma
,\varepsilon }\left( I_{0},\widehat{I}_{0}\right) \lesssim 2^{-\gamma t}%
\mathcal{F}_{\gamma ,\varepsilon }^{2}\sigma (I_{0}).
\end{equation*}%
Next, to estimate the sum involving $I_{0}\setminus I_{r}$, we use the fact
that $\left\{ J_{r,s}\right\} _{s\geq 1}$ is an $\varepsilon $-good
subpartition of $I_{r}$ for each $r$ (again since the intervals $J_{r,s}$
are good). We can thus use the third line in \eqref{Psi properties}, and
finally the Energy Hypothesis \eqref{smallest condition} to obtain 
\begin{equation*}
\sum_{r,s\geq 1}\Phi \left( J_{r,s},I_{0}\setminus I_{r}\right) \lesssim
\sum_{r\geq 1}\left\{ \sup_{s\geq 1}\left( \frac{\left\vert
J_{r,s}\right\vert }{\left\vert I_{r}\right\vert }\right) ^{\gamma }\right\}
\Psi _{\gamma ,\varepsilon }\left( I_{r},I_{0}\right) \lesssim 2^{-\gamma t}%
\mathcal{F}_{\gamma ,\varepsilon }^{2}\sigma (I_{0}).
\end{equation*}%
This last estimate also proves the case $I_{0}=\widehat{I}_{0}\in \mathcal{S}
$.
\end{proof}



\begin{theorem}
We have the following Carleson measure estimates for $S\in \mathcal{S}\;$and$%
\;K\in \mathcal{D}^{\sigma }$:%
\begin{align}
& \sum_{S^{\prime }\in \mathcal{S}\left( S\right) }\sigma \left( S^{\prime
}\right) \leq \frac{1}{4}\sigma \left( S\right) \text{ and }\sum_{S\in 
\mathcal{S}:S\subsetneqq K}\sigma \left( S\right) \leq \sigma \left(
K\right) ,  \label{e.fStopping} \\
& \sum_{J\in \mathcal{C}^{\omega }(S):J\subset K,2^{r}\left\vert
J\right\vert <\left\vert K\right\vert }\left\vert \left\langle H(\mathbf{1}%
_{S}\sigma ),h_{J}^{\omega }\right\rangle _{\omega }\right\vert ^{2}\lesssim
\left( \mathcal{F}_{\gamma ,\varepsilon }^{2}+\mathcal{H}^{2}\right) \sigma
\left( K\right) .  \label{e.measure1}
\end{align}
\end{theorem}

\begin{remark}
\label{r.boxes} The Corona Decomposition and this last estimate can be
compared to a general strategy for proving Carleson measure estimates. The
Corona Decomposition is reminiscent of the sets in \eqref{e.wS1}; the
condition \eqref{e.fStopping} can be compared to \eqref{e.wS}; and the
condition \eqref{e.measure1} can be compared to \eqref{e.wS1}.
\end{remark}

\begin{proof}[Proof of \eqref{e.fStopping}.]
Concerning the second inequality in \eqref{e.fStopping}, 
as is well known, it suffices to verify it for $%
K=S_{0}\in \mathcal{S}$. And this case follows from the recursive
application of the estimate first half of \eqref{e.fStopping} to 
the interval $S_{0}$ and all of its
children in $\mathcal{S}$. 

So we turn to the first half of \eqref{e.fStopping}. 
The intervals in the collection $\mathcal{S}%
(S_{0})=\{S_{r}\;:\;r\geq 1\}$ given in Definition \ref{d.stopping} are
pairwise disjoint and strictly contained in $I_{0}$. Each of them 
satisfies 
\eqref{e.stoppingDef}, so  we can apply \eqref{smallest condition} to see that 
\begin{align}
\sum_{S\in \mathcal{S}(S_{0})}\sigma (S)& =\sum_{r\geq 1}\sigma (S_{r})
\label{e.14} \\
& \leq \frac{1}{4\mathcal{F}_{\gamma,\varepsilon}^{2}}\sum_{r\geq 1}\Psi _{\gamma
,\varepsilon }\left( S_{r},S_{0}\right) \leq \tfrac{1}{4}\sigma (S_{0})\,.
\end{align}
\end{proof}

\begin{proof}[Proof of \eqref{e.measure1}.]
Fix $S\in \mathcal{S}$ and $K$, which we can assume is a subset of $S$. If
we apply the Hilbert transform to $\sigma \mathbf{1}_{K}$, as opposed to $%
\sigma \mathbf{1}_{S}$, we have by \eqref{Plancherel} for $\omega $ and %
\eqref{e.H1}, 
\begin{equation*}
\sum_{\substack{ {J\in \mathcal{C}^{\omega }(S)}  \\ J\subset K,\,\lvert
J\rvert <2^{-r}\lvert K\rvert }}\left\vert \left\langle H(\mathbf{1}%
_{K}\sigma ),h_{J}^{\omega }\right\rangle _{\omega }\right\vert ^{2}\leq
\int_{K}\left\vert H(\mathbf{1}_{K}\sigma )\right\vert ^{2}\;\omega (dx)\leq 
\mathcal{H}^{2}\sigma (K)\,.
\end{equation*}%
And so we consider the Hilbert transform applied to $\sigma \mathbf{1}%
_{S\backslash K}$, and show 
\begin{equation}
\sum_{\substack{ {J\in \mathcal{C}^{\omega }(S)}  \\ J\subset K,\,\lvert
J\rvert <2^{-r}\lvert K\rvert }}\left\vert \left\langle H(\mathbf{1}%
_{S\backslash K}\sigma ),h_{J}^{\omega }\right\rangle _{\omega }\right\vert
^{2}\lesssim \mathcal{F}_{\gamma ,\varepsilon }^{2}\sigma (K)\,.
\label{e.m1}
\end{equation}

We can assume that $K\subsetneq S$, and that there is some $J\in \mathcal{C}%
^{\omega }(S) $ with $J \subset K $. From this we see that $K$ was \emph{not}
a stopping interval. That is, the interval $K$ \emph{must} fail %
\eqref{e.stoppingDef}.

Let $\mathcal{J}$ denote the maximal intervals $J\in \mathcal{C}^{\omega }(S)
$ with $J\subset K$ and $\lvert J\rvert <2^{-r}\lvert K\rvert $. Using the
notation of \eqref{e.Ptilde}, we can use \eqref{e.minimal<}, with $I^{\prime }=K
$, $\widehat{I}=S$, and $J\in \mathcal{J}$. It gives us 
\begin{equation*}
\sum_{J\in \mathcal{J}}\lVert \widetilde{\mathsf{P}}_{J}^{\omega }H(\mathbf{1%
}_{S\backslash K}\sigma )\rVert _{L^{2}(\omega )}^{2}\lesssim \sum_{J\in 
\mathcal{J}}\Phi \left( J,S\setminus K\right) \lesssim \mathcal{F}_{\gamma
,\varepsilon }^{2}\sigma (K)\,.
\end{equation*}%
The second inequality uses the fact that $K$ fails \eqref{e.stoppingDef}.
This proves \eqref{e.m1}.
\end{proof}

The following Carleson measure estimate,  along with 
\S \ref{s.43} and \S \ref{s.61},
are the three places where the Energy Hypothesis is used in this proof: It
will provide the decay in the the parameter $t$ in \eqref{e.miraculous}. For
all integers $t\geq 0$, we define for $S\in \mathcal{S}$, which are not
maximal, 
\begin{equation}
\alpha _{t}(S)\equiv \sum_{S^{\prime }:\pi _{\mathcal{S}}^{t}(S^{\prime
})=S}\left\Vert \mathsf{P}_{S^{\prime }}^{\omega }H(\sigma \mathbf{1}_{\pi _{%
\mathcal{S}}^{1}\left( S\right) \setminus S})\right\Vert _{L^{2}(\omega
)}^{2}  \label{e.alphat}
\end{equation}%
Here, we are taking the projection $H(\sigma \mathbf{1}_{\pi _{\mathcal{S}%
}^{1}\left( S\right) \backslash S})$ associated to parts of the Corona
decomposition which are `far below' $S$. We have this off-diagonal estimate.


\begin{theorem}
\label{t.miraculous} The following Carleson measure estimate holds: 
\begin{equation}
\sum_{S\in \mathcal{S}\;:\;\pi _{\mathcal{D}^{\sigma }}^{1}\left( S\right)
\subset K}\alpha _{t}(S)\lesssim 2^{-\gamma t}\mathcal{F}_{\gamma
,\varepsilon }^{2}\sigma (K)\,,\qquad K\in \mathcal{D}^{\sigma }.
\label{e.miraculous}
\end{equation}%
The implied constant is independent of the choice of interval $K$ and $t\geq
1$.
\end{theorem}



\begin{remark}
\label{r.payAttention} In the estimate \eqref{e.miraculous}, we draw
attention to the fact that the dyadic parent $\pi _{\mathcal{D}^{\sigma
}}^{1}\left( S\right) $ of $S$ appears. Similar conditions will arise below,
and it is essential to track them as the measures we are dealing with are 
\emph{not} doubling. In fact, the role of the dyadic parents is revealed in
the next proof: Use the negation of \eqref{e.stoppingDef} when $\pi _{%
\mathcal{D}^{\sigma }}^{1}\left( S\right) \notin \mathcal{S}$, and otherwise
use the Energy Condition.
\end{remark}



\begin{proof}
Our first task is to show that%
\begin{equation*}
\sum_{S\in \mathcal{S}\left( \widehat{S}\right) }\alpha _{t}(S)\leq
2^{-\gamma t}\mathcal{F}_{\gamma ,\varepsilon }^{2}\sigma (\widehat{S}),\ \
\ \ \ \widehat{S}\in \mathcal{S}.
\end{equation*}%
For the purposes of this proof, we will set $\mathcal{S}_{t}(S)=\{S^{\prime
}\in \mathcal{S}\;:\;\pi _{\mathcal{S}}^{t}(S^{\prime })=S\}$, using this
notation for $S\in \mathcal{S}(\widehat{S})$. We want to apply %
\eqref{e.Edual} to the expressions $\alpha _{t}$. To this end define%
\begin{equation}
\mathcal{J}(S^{\prime })\equiv \left\{ J\in \mathcal{C}^{\omega }(S)\;:\;J%
\text{ is maximal w.r.t. }J\subset S^{\prime },\left\vert J\right\vert
<2^{-r}\left\vert S^{\prime }\right\vert \right\} .  \label{e.maxJ}
\end{equation}%
It follows by definition that we have $\lvert J\rvert <2^{-r}\lvert
S^{\prime }\rvert $ for all $J\in \mathcal{J}(S^{\prime })$. And, as all
Haar functions have mean zero, we can apply \eqref{e.Edual}. From this, we
see that 
\begin{equation*}
\alpha _{t}(S)\lesssim \sum_{S^{\prime }\in \mathcal{S}_{t}(S)}\sum_{J\in 
\mathcal{J}(S^{\prime })}\Phi \left( J,\widehat{S}\setminus S\right) ,
\end{equation*}%
and so by the third line in \eqref{Psi properties},%
\begin{eqnarray}
\sum_{S\in \mathcal{S}(\widehat{S})}\alpha _{t}(S) &\lesssim &\sum_{S\in 
\mathcal{S}(\widehat{S})}\sum_{S^{\prime }\in \mathcal{S}_{t}(S)}\sum_{J\in 
\mathcal{J}(S^{\prime })}\Phi (J,\widehat{S}\setminus S)
\label{alpha control} \\
&\lesssim &2^{-t\gamma }\sum_{S\in \mathcal{S}(\widehat{S})}\Psi (S,\widehat{%
S})\lesssim 2^{-t\gamma }\mathcal{F}_{\gamma ,\varepsilon }^{2}\sigma (%
\widehat{S}),  \notag
\end{eqnarray}%
where the final inequality follows from the assumed Energy Hypothesis (\ref%
{smallest condition}).

Now fix $K$ as in \eqref{e.miraculous} and let $\widehat{S}\in \mathcal{S}$
be the stopping interval such that $K\in \mathcal{C}^{\sigma }(\widehat{S})$%
. Let $\mathcal{G}_{1}\equiv \left\{ S_{i}\right\} _{i}$ be the maximal
intervals from $\mathcal{S}$ that are strictly contained in $K$. Inductively
define the $\left( k+1\right) ^{st}$ generation $\mathcal{G}_{k+1}$ to
consist of the maximal intervals from $\mathcal{S}$ that are strictly
contained in some $k^{th}$ generation interval $S\in \mathcal{G}_{k}$.
Inequality \eqref{alpha control} shows that%
\begin{equation*}
\sum_{S\in \mathcal{G}_{k+1}}\alpha _{t}(S)\lesssim 2^{-t\gamma }\mathcal{F}%
_{\gamma ,\varepsilon }^{2}\sum_{S\in \mathcal{G}_{k}}\sigma \left( S\right)
.
\end{equation*}%
We also have from \eqref{e.fStopping} that%
\begin{equation*}
\sum_{k=1}^{\infty }\sum_{S\in \mathcal{G}_{k}}\sigma \left( S\right)
\lesssim \sum_{S\in \mathcal{G}_{1}}\sigma \left( S\right) \leq \sigma
\left( K\right) .
\end{equation*}

This will be all we need in the case $K=\widehat{S}$, but when $K\neq 
\widehat{S}$, we will use Lemma \ref{l.minimal} to control the first
generation intervals $S$ in $\mathcal{G}_{1}$:%
\begin{equation*}
\sum_{S\in \mathcal{G}_{1}}\alpha _{t}(S)\lesssim 2^{-t\eta }\sigma \left(
K\right) .
\end{equation*}%
Indeed, we simply apply Lemma \ref{l.minimal} with $\widehat{I}_{0}=\widehat{%
S}$, $I_{0}=K$, $\left\{ I_{r}\right\} _{r\geq 1}=\mathcal{G}_{1}$, and $%
\left\{ J_{r,s}\right\} _{s\geq 1}=\bigcup_{S^{\prime }\in \mathcal{S}%
_{t}(S)}\mathcal{J}(S^{\prime })$.

\smallskip

When $K\neq \widehat{S}$ we finish with%
\begin{eqnarray*}
\sum_{S\in \mathcal{S}\;:\;\pi _{\mathcal{D}^{\sigma }}^{1}\left( S\right)
\subset K}\alpha _{t}(S) &=&\sum_{S\in \mathcal{G}_{1}}\alpha
_{t}(S)+\sum_{k=1}^{\infty }\sum_{S\in \mathcal{G}_{k+1}}\alpha _{t}(S) \\
&\lesssim &2^{-t\eta }\sigma \left( K\right) +2^{-t\eta }\mathcal{F}_{\gamma
,\varepsilon }^{2}\sum_{k=1}^{\infty }\sum_{S\in \mathcal{G}_{k}}\sigma
\left( S\right)  \\
&\lesssim &2^{-t\gamma }\mathcal{F}_{\gamma ,\varepsilon }^{2}\sigma \left(
K\right) ,
\end{eqnarray*}%
and when $K=\widehat{S}$ we set $\mathcal{G}_{0}=\{\widehat{S}\}$ and
estimate 
\begin{align*}
\sum_{S\in \mathcal{S}\;:\;\pi _{\mathcal{D}^{\sigma }}^{1}\left( S\right)
\subset \widehat{S}}\alpha _{t}(S)=\sum_{k=0}^{\infty }\sum_{S\in \mathcal{G}%
_{k+1}}\alpha _{t}(S)& \lesssim 2^{-t\gamma }\mathcal{F}_{\gamma
,\varepsilon }^{2}\sum_{k=0}^{\infty }\sum_{S\in \mathcal{G}_{k}}\sigma
\left( S\right)  \\
& \lesssim 2^{-t\gamma }\mathcal{F}_{\gamma ,\varepsilon }^{2}\sigma (%
\widehat{S}).
\end{align*}
\end{proof}


We need a Carleson measure estimate that is a common variant of %
\eqref{e.fStopping} and \eqref{e.miraculous}. Define 
\begin{equation}
\beta (S)\equiv \left\Vert \mathsf{P}_{S}^{\omega }H(\sigma \mathbf{1}_{\pi
_{\mathcal{D}^{\sigma }}^{1}(S)})\right\Vert _{L^{2}(\omega )}^{2}\,.
\label{e.betat}
\end{equation}


\begin{theorem}
\label{t.betat} We have the Carleson measure estimate 
\begin{equation}
\sum_{S\in \mathcal{S}\;:\;\pi _{\mathcal{D}^{\sigma }}^{1}(S)\subset
K}\beta (S)\lesssim (\mathcal{H}^{2}+\mathcal{F}_{\gamma ,\varepsilon
}^{2})\sigma (K)  \label{e.betat<}
\end{equation}
\end{theorem}



\begin{proof}
Using the decomposition $\pi _{\mathcal{D^{\sigma }}}^{1}(S)=S\cup \{\pi _{%
\mathcal{D}^{\sigma }}^{1}(S)\backslash S\}$, we write $\beta (S)\leq
2(\beta _{1}(S)+\beta _{2}(S))$ where 
\begin{align*}
\beta _{1}(S)& \equiv \left\Vert \mathsf{P}_{S}^{\omega }H(\sigma \mathbf{1}%
_{S})\right\Vert _{L^{2}(\omega )}^{2}\,, \\
\beta _{2}(S)& \equiv \left\Vert \mathsf{P}_{S}^{\omega }H(\sigma \mathbf{1}%
_{\pi _{\mathcal{D}^{\sigma }}^{1}(S)\backslash S})\right\Vert
_{L^{2}(\omega )}^{2}\,.
\end{align*}%
We certainly have $\beta _{1}(S)\leq \mathcal{H}\sigma (S)$, so that by %
\eqref{e.fStopping}, we need only consider the Carleson measure norm of the
terms $\beta _{2}(S)$.

Fix an interval $K$ of the form $K=\pi _{\mathcal{D}}^{1}(S_{0})$ for some $%
S_{0}\in \mathcal{S}$. Let $\mathcal{T}$ be the maximal intervals of the
form $\pi _{\mathcal{D}}^{1}(S)\subsetneq K$, and for $T\in \mathcal{T}$,
let $\mathcal{S}(T)$ be all intervals $S\in \mathcal{S}$ with $S\subset T$
and $S$ is maximal. Using the notation of \eqref{e.maxJ} and \eqref{e.Ptilde}%
, we can estimate 
\begin{align*}
\sum_{T\in \mathcal{T}}\sum_{S\in \mathcal{S}(T)}\beta _{2}(S)& \lesssim
\sum_{T\in \mathcal{T}}\sum_{S\in \mathcal{S}(T)}\sum_{J\in \mathcal{J}%
(S)}\left\Vert \widetilde{\mathsf{P}}_{J}^{\omega }H(\sigma \mathbf{1}_{\pi
_{\mathcal{D}^{\sigma }}^{1}(S)\backslash S})\right\Vert _{L^{2}(\omega
)}^{2} \\
& \lesssim \mathcal{F}_{\gamma ,\varepsilon }^{2}\sigma (K)
\end{align*}%
Here, we have have been careful to arrange the collections $\mathcal{T}$, $%
\mathcal{S}(T)$ and $\mathcal{J}(S)$ so that \eqref{e.minimal<} applies.

We argue that this inequality is enough to conclude the Lemma. Suppose that $%
S^{\prime }\in \mathcal{S}$, with $S^{\prime }\subset K$, but $S^{\prime }$
is not in any collection $\mathcal{S}(T)$ for $T\in \mathcal{T}$. It follows
that $S^{\prime }\subsetneq S$ for some $S\in \mathcal{S}(T)$ and $T\in 
\mathcal{T}$. This implies that the Carleson measure estimate %
\eqref{e.fStopping} will conclude the proof.
\end{proof}


A last Carleson measure estimate needed arises from the quantities 
\begin{equation}
\gamma (S)\equiv \left\Vert \mathsf{P}_{S}^{\omega }H(\mathbf{1}_{\pi _{%
\mathcal{S}}^{1}(S)\backslash \pi _{\mathcal{D}^{\sigma }}^{1}(S)}\sigma
)\right\Vert _{L^{2}(\omega )}^{2}  \label{e.gammat}
\end{equation}


\begin{theorem}
\label{t.gammat} We have the estimate 
\begin{equation}
\sum_{S\in \mathcal{S}\;:\;\pi _{\mathcal{D}^{\sigma }}^{1}(S)\subset
K}\gamma (S)\lesssim \mathcal{F}_{\gamma ,\varepsilon }^{2}\sigma (K)\,.
\label{e.gammat<}
\end{equation}
\end{theorem}



\begin{proof}
We can take $K=\pi _{\mathcal{D}^{\sigma }}^{1}(S_{0})$ for some $S_{0}\in 
\mathcal{S}$, and in addition, we can assume that $K\not\in \mathcal{S}$,
because otherwise we are applying the Hilbert transform to the zero function.

We repeat an argument from the previous proof. Details are omitted.
\end{proof}

\section{The Paraproducts}

\label{s.paraproducts}

We continue to follow the line of argument in \cite{Vol} and \cite{NTV3}
using similar notation for the benefit of the reader. The paraproduct term $%
A_{2}^{4}$ is the central term in the proof. In this section, we reorganize
the sum in \eqref{e.Energy} according to the Corona Decomposition: The
essential point that must be accounted for is that for $J\in \mathcal{C}%
^{\omega }(S)$ and $J\subset I$, we \emph{need not have} $I\in \mathcal{C}%
^{\sigma }(S)$. On the other hand, it will be the case that $I\in \mathcal{C}%
^{\sigma }(\pi _{\mathcal{S}}^{t}(S))$ for some ancestor $\pi _{\mathcal{S}%
}^{t}(S)$ of $S$. The ancestor $\pi _{\mathcal{S}}^{t}(S)$ is only defined
for $1\leq t\leq \rho (S)$. (See Definition~\ref{d.stopping} for the
definition of $\rho (S)$.) In fact, the sum splits into $%
A_{2}^{4}=A_{1}^{5}+A_{2}^{5}$, where 
\begin{gather}
A_{1}^{5}\equiv \sum_{S\in \mathcal{S}}\sum_{\substack{ (I,J)\in \mathcal{P}%
(S)  \\ J\in \mathcal{C}^{\omega }(S)}}\mathbb{E}_{I_{J}}^{\sigma }\Delta
_{I}^{\sigma }f\cdot \left\langle H(\mathbf{1}_{S}\sigma ),\Delta
_{J}^{\omega }\phi \right\rangle _{\omega }\,,  \label{e.a51} \\
A_{2}^{5}\equiv \sum_{S\in \mathcal{S}\setminus \{I^{0}\}}\sum_{t=1}^{\rho
(S)}\sum_{\substack{ (I,J)\in \mathcal{P}(\pi _{\mathcal{S}}^{t}(S))  \\ %
J\in \mathcal{C}^{\omega }(S)}}\mathbb{E}_{I_{J}}^{\sigma }\Delta
_{I}^{\sigma }f\cdot \left\langle H(\mathbf{1}_{\pi _{\mathcal{S}%
}^{t}(S)}\sigma ),\Delta _{J}^{\omega }\phi \right\rangle _{\omega }\,.
\label{e.a52'}
\end{gather}

In $A_{1}^{5}$, we are treating the case where both $I \in \mathcal{D }%
^{\sigma }$ and $J $ are `controlled' by the same stopping interval. ($J$ is
not `very far' below $I$, as measured by the stopping intervals $\mathcal{S}$%
.) And, the point in the last line is that we are summing over $J\in 
\mathcal{C}^{\omega }(S)$, while the pair $(I,J) \in P (\pi _{\mathcal{S}%
}^{t}(S)) $, where $\pi _{ \mathcal{S}}^{t}(S)$ denotes the $t$-fold parent
of $S$ in the grid $\mathcal{S}$, see \eqref{e.parent}. This ancestor
appears in two places, controlling the sum over $I$, and in the argument of
the Hilbert transform.

We will prove 
\begin{equation}
\lvert A_{1}^{5}\rvert \lesssim (\mathcal{H}+\mathcal{F}_{\gamma
,\varepsilon })\left\Vert f\right\Vert _{L^{2}(\sigma )}\left\Vert \phi
\right\Vert _{L^{2}(\omega )}\,,  \label{e.a51<}
\end{equation}%
while $A_{2}^{5}$ will require further decomposition.

\subsection{$A ^{5}_1$: The First Paraproduct}

We use the telescoping sum identities \eqref{e.mart1} and \eqref{e.mart2} to
reorganize the sum in \eqref{e.a51}. Fix $S\in \mathcal{S}$ and $J\in 
\mathcal{C}^{\omega }(S)$. The sum over $I$ in \eqref{e.a51} is \eqref{e.PS}. 
\begin{equation}
\sum_{I\;:\;(I,J)\in \mathcal{P}(S)}\mathbb{E}_{I_{J}}^{\sigma }\Delta
_{I}^{\sigma }f=\mathbb{E}_{I_{J,\ast }}^{\sigma }f-\mathbb{E}_{\pi _{%
\mathcal{D}}\left( S\right) }^{\sigma }f\,.  \label{e.martApplied}
\end{equation}%
Here, we set $I_{J,\ast }$ to be the minimal member of $\mathcal{C}^{\sigma
}(S)$ that contains $J$, and satisfies $2^{r}\lvert J\rvert <\lvert I\rvert $%
. Such an interval must exist as $J$ is good. Thus, we can write 
\begin{equation}
A_{1}^{5}=\sum_{S\in \mathcal{S}}A_{1}^{5}(S)\,,  \label{e.a51=}
\end{equation}%
\begin{equation}
A_{1}^{5}(S)\equiv \sum_{J\in \mathcal{C}^{\omega }(S)}\left( \mathbb{E}%
_{I_{J,\ast }}^{\sigma }f-\mathbb{E}_{S}^{\sigma }f\right) \cdot
\left\langle H(\mathbf{1}_{S}\sigma ),\Delta _{J}^{\omega }\phi
\right\rangle _{\omega }\,.  \label{e.a51s}
\end{equation}%
The basic estimate here, and our first paraproduct style estimate is

\begin{proposition}
\label{p.a51S} We have the estimates 
\begin{equation}
\lvert A_{1}^{5}(S)\rvert \lesssim (\mathcal{H}+\mathcal{F}_{\gamma
,\varepsilon })\left\Vert \mathsf{P}_{S}^{\sigma }f-\mathbf{1}_{S}\mathbb{E}%
_{\pi _{\mathcal{D}}\left( S\right) }^{\sigma }f\right\Vert _{L^{2}(\sigma
)}\left\Vert \mathsf{P}_{S}^{\omega }\phi \right\Vert _{L^{2}(\omega
)}\,,\qquad S\in \mathcal{S}\,.  \label{e.a51S<}
\end{equation}%
Here, the projections on the right are defined in \eqref{e.PS}.
\end{proposition}

\begin{proof}
We should reorganize the sum in a fashion consistent with paraproduct-type
estimates. For $I\in \mathcal{C}^{\sigma }(S)$, let 
\begin{equation*}
\mathsf{Q}_{I}^{\omega }\phi \equiv \sum_{J\in \mathcal{C}^{\omega
}(S):I_{J,\ast }=I}\Delta _{J}^{\omega }\phi \,.
\end{equation*}%
Using the Cauchy-Schwartz inequality, and the fact that $\mathbb{E}%
_{I}^{\sigma }f=\mathbb{E}_{I}^{\sigma }\mathsf{P}_{S}^{\sigma }f$ and $%
\mathsf{Q}_{I}^{\omega }\phi =\mathsf{Q}_{I}^{\omega }\mathsf{P}_{S}^{\omega
}\phi $ we see that 
\begin{align*}
\lvert A_{1}^{5}(S)\rvert & =\left\vert \sum_{I\in \mathcal{C}^{\sigma
}(S)}\left( \mathbb{E}_{I}^{\sigma }\mathsf{P}_{S}^{\sigma }f-\mathbb{E}%
_{\pi _{\mathcal{D}}\left( S\right) }^{\sigma }f\right) \cdot \left\langle H(%
\mathbf{1}_{S}\sigma ),\mathsf{Q}_{I}^{\omega }\mathsf{P}_{S}^{\omega }\phi
\right\rangle _{\omega }\right\vert  \\
& \leq \left[ \sum_{I\in \mathcal{C}^{\sigma }(S)}\lvert \mathbb{E}%
_{I}^{\sigma }\mathsf{P}_{S}^{\sigma }f-\mathbb{E}_{\pi _{\mathcal{D}}\left(
S\right) }^{\sigma }f\rvert ^{2}\cdot \left\Vert \mathsf{Q}_{I}^{\omega }H(%
\mathbf{1}_{S}\sigma )\right\Vert _{L^{2}(\omega )}^{2}\sum_{I\in \mathcal{C}%
^{\sigma }(S)}\left\Vert \mathsf{Q}_{I}^{\omega }\mathsf{P}_{S}^{\omega
}\phi \right\Vert _{L^{2}(\omega )}^{2}\right] ^{\frac{1}{2}} \\
& \leq \left\Vert \mathsf{P}_{S}^{\omega }\phi \right\Vert _{L^{2}(\omega )}%
\left[ \sum_{I\in \mathcal{C}^{\sigma }(S)}\lvert \mathbb{E}_{I}^{\sigma
}\left( \mathsf{P}_{S}^{\sigma }f-\mathbf{1}_{S}\mathbb{E}_{\pi _{\mathcal{D}%
}\left( S\right) }^{\sigma }f\right) \rvert ^{2}\cdot \left\Vert \mathsf{Q}%
_{I}^{\omega }H(\mathbf{1}_{S}\sigma )\right\Vert _{L^{2}(\omega )}^{2}%
\right] ^{1/2}\,.
\end{align*}%
In view of the Carleson Embedding inequality, namely \eqref{e.embed1} and %
\eqref{e.embed2}, this last factor is at most $\left\Vert \mathsf{P}%
_{S}^{\sigma }f-\mathbf{1}_{S}\mathbb{E}_{\pi _{\mathcal{D}}\left( S\right)
}^{\sigma }f\right\Vert _{L^{2}(\sigma )}$ times the Carleson measure norm
of the coefficients 
\begin{equation*}
\left\{ \left\Vert \mathsf{Q}_{I}^{\omega }H(\mathbf{1}_{S}\sigma
)\right\Vert _{L^{2}(\omega )}^{2}\ :\;I\in \mathcal{C}^{\omega }(S)\right\}
\,.
\end{equation*}%
But by the Plancherel formula \eqref{Plancherel} this is what is shown in %
\eqref{e.measure1} to be at most a constant multiple of $\mathcal{H}+%
\mathcal{F}_{\gamma ,\varepsilon }$, so the proof is complete.
\end{proof}

To complete the estimate for $A_{1}^{5}$, from \eqref{e.a51=} and the
observation that the projections on the right in \eqref{e.a51S<} are
essentially orthogonal, see \eqref{e.2}, we can estimate 
\begin{align*}
\lvert A_{1}^{5}\rvert & \lesssim (\mathcal{H}+\mathcal{F}_{\gamma
,\varepsilon })\sum_{S\in \mathcal{S}}\left\Vert \mathsf{P}_{S}^{\sigma }f-%
\mathbf{1}_{S}\mathbb{E}_{\pi _{\mathcal{D}}\left( S\right) }^{\sigma
}f\right\Vert _{L^{2}(\sigma )}\left\Vert \mathsf{P}_{S}^{\omega }\phi
\right\Vert _{L^{2}(\omega )}. \\
& \lesssim (\mathcal{H}+\mathcal{F}_{\gamma ,\varepsilon })\left( \sum_{S\in 
\mathcal{S}}\left\Vert \mathsf{P}_{S}^{\sigma }f-\mathbf{1}_{S}\mathbb{E}%
_{\pi _{\mathcal{D}}\left( S\right) }^{\sigma }f\right\Vert _{L^{2}(\sigma
)}^{2}\sum_{S\in \mathcal{S}}\left\Vert \mathsf{P}_{S}^{\omega }\phi
\right\Vert _{L^{2}(\omega )}^{2}\right) ^{1/2} \\
& \lesssim (\mathcal{H}+\mathcal{F}_{\gamma ,\varepsilon })\left\Vert
f\right\Vert _{L^{2}(\sigma )}\left\Vert \phi \right\Vert _{L^{2}(\omega )},
\end{align*}%
since we have 
\begin{eqnarray}
\sum_{S\in \mathcal{S}}\left\Vert \mathbf{1}_{S}\mathbb{E}_{\pi _{\mathcal{D}%
}\left( S\right) }^{\sigma }f\right\Vert _{L^{2}(\sigma )}^{2} &=&\sum_{S\in 
\mathcal{S}}\sigma \left( S\right) \left\vert \mathbb{E}_{\pi _{\mathcal{D}%
}\left( S\right) }^{\sigma }f\right\vert ^{2}
\\&\leq& \sum_{S\in \mathcal{S}%
}\sigma \left( S\right) \left( \mathbb{E}_{S}^{\sigma }\left\vert
f\right\vert \right) ^{2} \label{e.dominate}\\
&\lesssim &\left\Vert \mathcal{M}_{\sigma }f\right\Vert _{L^{2}(\sigma
)}^{2}\lesssim \left\Vert f\right\Vert _{L^{2}(\sigma )}^{2}.
\end{eqnarray}%
Here, we should make an appeal to \eqref{e.fStopping} in order to conclude that the maximal function  $ \mathcal M _{\sigma }$ 
dominates the sum in \eqref{e.dominate}.  This is \eqref{e.a51<}.

\subsection{The Remaining Paraproducts}

We repeat the analysis of \eqref{e.martApplied}, but for the term $A_{2}^{5}$
defined in \eqref{e.a52'}. Fix $J$, which must be a member of $\mathcal{C}%
^{\omega }(S)$ for some $S\in \mathcal{S}\setminus \{I^{0}\}$. The sum over $%
I$ in \eqref{e.a52'}, as it turns out, is only a function of this $S$, and
equals%
\begin{align}
A_{2}^{5}(S)& \equiv \sum_{t=1}^{\rho (S)}\sum_{\substack{ (I,J)\in \mathcal{%
P}(\pi _{\mathcal{S}}^{t}(S))  \\ J\in \mathcal{C}^{\omega }(S)}}\mathbb{E}%
_{I_{J}}^{\sigma }\Delta _{I}^{\sigma }f\cdot \left\langle H(\mathbf{1}_{\pi
_{\mathcal{S}}^{t}(S)}\sigma ),\Delta _{J}^{\omega }\phi \right\rangle
_{\omega }  \label{e.a521} \\
& =\sum_{t=1}^{\rho (S)}\sum_{J\in \mathcal{C}^{\omega }(S)}(\mathbb{E}_{\pi
_{\mathcal{D}^{\sigma }}^{2}(\pi _{\mathcal{S}}^{t-1}(S))}^{\sigma }f-%
\mathbb{E}_{\pi _{\mathcal{D}^{\sigma }}^{1}(\pi _{\mathcal{S}%
}^{t}(S))}^{\sigma }f)\cdot \left\langle H(\mathbf{1}_{\pi _{\mathcal{S}%
}^{t}(S)}\sigma ),\Delta _{J}^{\omega }\phi \right\rangle _{\omega }\,.
\end{align}%
We argue as follows. With $J\in \mathcal{C}^{\omega }(S)$ fixed, the sum
over $I$ such that $(I,J)\in \mathcal{P}(\pi _{\mathcal{S}}^{t}(S))$ is only
a function of $S$ and $t$, and is a sum over consecutive intervals in the
grid $\mathcal{D}^{\sigma }$. The smallest interval that contributes to the
sum is $\pi _{\mathcal{D}^{\sigma }}^{2}(\pi _{\mathcal{S}}^{t-1}(S))$, the 
\emph{second} dyadic parent of $\pi _{\mathcal{S}}^{t-1}(S)$, and the
largest is $\pi _{\mathcal{D}^{\sigma }}^{1}(\pi _{\mathcal{S}}^{t}(S))$.
(Recall Definition~\ref{d.corona}. Also, these two intervals might be one
and the same.)

In \eqref{e.a521}, the sum over $J$ is independent of the sum over $t$. In
the next steps, we concentrate on the sum over $t$. Below we add and
subtract a cancellative term, to adjust for the second parent in %
\eqref{e.a521}. 
\begin{align}\label{e.addSub2}
\widetilde{A}_{2}^{5}(S)& =\sum_{t=1}^{\rho (S)}(\mathbb{E}_{\pi _{\mathcal{D%
}^{\sigma }}^{2}(\pi _{\mathcal{S}}^{t-1}(S))}^{\sigma }f-\mathbb{E}_{\pi _{%
\mathcal{D}^{\sigma }}^{1}\left( \pi _{\mathcal{S}}^{t-1}(S)\right)
}^{\sigma }f+\mathbb{E}_{\pi _{\mathcal{D}^{\sigma }}^{1}\left( \pi _{%
\mathcal{S}}^{t-1}(S)\right) }^{\sigma }f-\mathbb{E}_{\pi _{\mathcal{D}%
^{\sigma }}^{1}(\pi _{\mathcal{S}}^{t}(S))}^{\sigma }f)
\\ & \qquad \times H(\mathbf{1} _{\pi _{\mathcal{S}}^{t}(S)}\sigma )   \\
& =\widetilde{A}_{21}^{5}(S)+\widetilde{A}_{22}^{5}(S)\,, \\
\widetilde{A}_{21}^{5}(S)& \equiv \sum_{t=1}^{\rho (S)}\left( \mathbb{E}%
_{\pi _{\mathcal{D}^{\sigma }}^{2}(\pi _{\mathcal{S}}^{t-1}(S))}^{\sigma }f-%
\mathbb{E}_{\pi _{\mathcal{D}^{\sigma }}^{1}\left( \pi _{\mathcal{S}%
}^{t-1}(S)\right) }^{\sigma }f\right) \cdot H(\mathbf{1}_{\pi _{\mathcal{S}%
}^{t}(S)}\sigma ),  \label{e.a61} \\
\widetilde{A}_{22}^{5}(S)& \equiv \sum_{t=1}^{\rho (S)}\left( \mathbb{E}%
_{\pi _{\mathcal{D}^{\sigma }}^{1}(\pi _{\mathcal{S}}^{t-1}(S))}^{\sigma }f-%
\mathbb{E}_{\pi _{\mathcal{D}^{\sigma }}^{1}\left( \pi _{\mathcal{S}%
}^{t}(S)\right) }^{\sigma }f\right) \cdot H(\mathbf{1}_{\pi _{\mathcal{S}%
}^{t}(S)}\sigma )\,.  \label{e.a62}
\end{align}%
The term $\widetilde{A}_{22}^{5}(S)$ in \eqref{e.a62} is itself a
telescoping sum, and so we can sum by parts to write 
\begin{equation}
\widetilde{A}_{22}^{5}(S)=\mathbb{E}_{\pi _{\mathcal{D}^{\sigma }}^{1}\left(
S\right) }^{\sigma }f\cdot H(\mathbf{1}_{\pi _{\mathcal{S}}^{1}(S)}\sigma
)+\sum_{t=1}^{\rho (S)}\mathbb{E}_{\pi _{\mathcal{D}^{\sigma }}^{1}\left(
\pi _{\mathcal{S}}^{t}(S)\right) }^{\sigma }f\cdot H(\mathbf{1}_{\pi _{%
\mathcal{S}}^{t+1}(S)\backslash \pi _{\mathcal{S}}^{t}(S)}\sigma )\,.
\label{e.a522=}
\end{equation}%
Note that there is one term missing, but it has the expectation $\mathbb{E}%
_{\pi _{\mathcal{S}}^{\rho (S)}(S)}^{\sigma }f=\mathbb{E}_{I^{0}}^{\sigma }f$%
, where $I^{0}$ is the largest interval that we fixed at the beginning of
the proof. In particular we have assumed that this expectation is zero.

We combine these steps, specifically the definition of $A_{2}^{5}$ in %
\eqref{e.a52'} and the identities \eqref{e.addSub2}, \eqref{e.a61}, (\ref%
{e.a62}), and \eqref{e.a522=} to write $%
A_{2}^{5}=A_{1}^{6}+A_{2}^{6}+A_{3}^{6}$, where 
\begin{gather}
A_{i}^{6}\equiv \sum_{S\in \mathcal{S}\setminus
\{I^{0}\}}A_{i}^{6}(S)\,,\qquad i=1,2,3\,, \\
A_{1}^{6}(S)\equiv \sum_{t=1}^{\rho (S)}(\mathbb{E}_{\pi _{\mathcal{D}%
^{\sigma }}^{2}(\pi _{\mathcal{S}}^{t-1}(S))}^{\sigma }f-\mathbb{E}_{\pi
_{\pi _{\mathcal{D}^{\sigma }}^{1}(\pi _{\mathcal{S}}^{t-1}(S))}}^{\sigma
}f)\cdot \left\langle H(\mathbf{1}_{\pi _{\mathcal{S}}^{t}(S)}\sigma ),%
\mathsf{P}_{S}^{\omega }\phi \right\rangle _{\omega }\,,  \label{e.a61s} \\
A_{2}^{6}(S)\equiv \mathbb{E}_{\pi _{\mathcal{D}^{\sigma }}^{1}\left(
S\right) }^{\sigma }f\cdot \left\langle H(\mathbf{1}_{\pi _{\mathcal{S}%
}^{1}(S)}\sigma ),\mathsf{P}_{S}^{\omega }\phi \right\rangle _{\omega }\,,
\label{e.a62s} \\
A_{3}^{6}(S)\equiv \sum_{t=1}^{\rho (S)}\mathbb{E}_{\pi _{\mathcal{D}%
^{\sigma }}^{1}\left( \pi _{\mathcal{S}}^{t}(S)\right) }^{\sigma }f\cdot
\left\langle H(\mathbf{1}_{\pi _{\mathcal{S}}^{t+1}(S)\backslash \pi _{%
\mathcal{S}}^{t}(S)}\sigma ),\mathsf{P}_{S}^{\omega }\phi \right\rangle
_{\omega }\,.  \label{e.a63s}
\end{gather}

Of these three expressions, the first $A_{1}^{6}$ has cancellative terms on
both $f$ and $\phi $, hence it is not (yet) a paraproduct as such. The
second $A_{2}^{6}$ is a paraproduct, one that is very close in form to that
of $A_{1}^{5}$, compare \eqref{e.a51s} and \eqref{e.a62s}. The third term is
a paraproduct, but looking at the support of the argument of the Hilbert
transform, one sees that it is also degenerate, and we should obtain some
additional decay in the parameter $t$, the `miraculous improvement of the
Carleson property' in Chapter 21.3 of \cite{Vol} - see \eqref{e.a63t<}
below. We take up these estimates in the next subsections, passing from more
intricate to less intricate.

In fact we will prove in \S \ref{s.63}, \S \ref{s.62} and \S \ref{s.61}
respectively,%
\begin{gather}
\lvert A_{3}^{6}\rvert \lesssim \mathcal{F}_{\gamma ,\varepsilon }\left\Vert
f\right\Vert _{L^{2}(\sigma )}\left\Vert \phi \right\Vert _{L^{2}(\omega
)}\,,  \label{e.63<} \\
\lvert A_{2}^{6}\rvert \lesssim (\mathcal{H}+\mathcal{F}_{\gamma
,\varepsilon })\left\Vert f\right\Vert _{L^{2}(\sigma )}\left\Vert \phi
\right\Vert _{L^{2}(\omega )}\,,  \label{e.62<} \\
\lvert A_{1}^{6}\rvert \lesssim (\mathcal{H}+\mathcal{F}_{\gamma
,\varepsilon })\left\Vert f\right\Vert _{L^{2}(\sigma )}\left\Vert \phi
\right\Vert _{L^{2}(\omega )}\,.  \label{e.61<}
\end{gather}%
In particular, the Energy Hypothesis enters into \eqref{e.63<}.

\subsection{The Term $A ^{6} _{3}$}

\label{s.63}

Let us fix $t$, and define 
\begin{align}
A_{3}^{6}(S,t)& \equiv \mathbb{E}_{\pi _{\mathcal{D}^{\sigma }}^{1}\left(
\pi _{\mathcal{S}}^{t}(S)\right) }^{\sigma }f\cdot \left\langle H(\mathbf{1}%
_{\pi _{\mathcal{S}}^{t+1}(S)\backslash \pi _{\mathcal{S}}^{t}(S)}\sigma ),%
\mathsf{P}_{S}^{\omega }\phi \right\rangle _{\omega }\,,\qquad S\in \mathcal{%
S},\rho (S)\geq t\,.  \label{e.a63ts} \\
A_{3}^{6}(t)& \equiv \sum_{S\in \mathcal{S}\;:\;\rho (S)\geq
t}A_{3}^{6}(S,t)\,.
\end{align}%
Here, we impose the restriction $\rho (S)\geq t$ so that the $t$-fold parent
of $S$ is defined.

The estimate we prove is 
\begin{equation}
\lvert A_{3}^{6}(t)\rvert \lesssim 2^{-\gamma t}\mathcal{F}_{\gamma
,\varepsilon }\left\Vert f\right\Vert _{L^{2}(\sigma )}\left\Vert \phi
\right\Vert _{L^{2}(\omega )}\,,\qquad t\geq 1\,.  \label{e.a63t<}
\end{equation}%
The constant $\epsilon =\gamma /2>0$. Clearly this proves \eqref{e.63<}
after summation on $t\geq 1$.

The projections $\mathsf{P}_{S}^{\omega }$ are orthogonal, so we have 
\begin{equation}
\lvert A_{3}^{6}(t)\rvert \leq \left\Vert \phi \right\Vert _{L^{2}(\omega )} 
\left[ \sum_{S\in \mathcal{S}\;:\;\rho (S)\geq t}\left\vert \mathbb{E}_{\pi
_{\mathcal{D}^{\sigma }}^{1}\left( \pi _{\mathcal{S}}^{t}(S)\right)
}^{\sigma }f\right\vert ^{2}\left\Vert \mathsf{P}_{S}^{\omega }H(\mathbf{1}%
_{\pi _{\mathcal{S}}^{t+1}(S)\backslash \pi _{\mathcal{S}}^{t}(S)}\sigma
)\right\Vert _{L^{2}(\omega )}^{2}\right] ^{1/2}\,.  \label{e.63tsum}
\end{equation}%
Recalling the notation \eqref{e.alphat}, the sum on the right in %
\eqref{e.63tsum} is 
\begin{equation*}
\left[ \sum_{S\in {\mathcal{S}}}\alpha _{t}({S})\left\vert \mathbb{E}_{\pi _{%
\mathcal{D}^{\sigma }}^{1}\left( S\right) }^{\sigma }f\right\vert ^{2}\right]
^{1/2}\,.
\end{equation*}%
Therefore, to prove \eqref{e.a63t<}, we should verify that the Carleson
measure norm of the coefficients $\{\alpha _{t}({S})\;:\;{S}\in {\mathcal{S}}%
\}$ is at most $C2^{-\gamma t}\mathcal{F}_{\gamma ,\varepsilon }$. But this
is the content of Theorem \ref{t.miraculous}, and so our proof is complete.

\subsection{The term $A ^{6} _2$}

\label{s.62} We certainly have $\pi _{\mathcal{D}^{\sigma }}^{1}(S)\subset
\pi _{\mathcal{S}}^{1}(S)$, so that it is natural to split term in %
\eqref{e.a62s} into two, namely writing $\pi _{\mathcal{S}}^{1}(S)=\pi _{%
\mathcal{D}^{\sigma }}^{1}(S)\cup \{\pi _{\mathcal{S}}^{1}(S)\backslash \pi
_{\mathcal{D}^{\sigma }}^{1}(S)\}$, to give us 
\begin{align}
\lvert A_{1}^{7}\rvert & \equiv \left\vert \sum_{S\in \mathcal{S}}\mathbb{E}%
_{\pi _{\mathcal{D}^{\sigma }}^{1}\left( S\right) }^{\sigma }f\cdot
\left\langle H(\mathbf{1}_{\pi _{\mathcal{D}^{\sigma }}^{1}(S)}\sigma ),%
\mathsf{P}_{S}^{\omega }\phi \right\rangle _{\omega }\right\vert \lesssim (%
\mathcal{H}+\mathcal{F}_{\gamma ,\varepsilon })\left\Vert f\right\Vert
_{L^{2}(\sigma )}\left\Vert \phi \right\Vert _{L^{2}(\omega )}\,,
\label{e.6s1<} \\
\lvert A_{2}^{7}\rvert & \equiv \left\vert \sum_{S\in \mathcal{S}}\mathbb{E}%
_{\pi _{\mathcal{D}^{\sigma }}^{1}\left( S\right) }^{\sigma }f\cdot
\left\langle H(\mathbf{1}_{\pi _{\mathcal{S}}^{1}(S)\backslash \pi _{%
\mathcal{D}^{\sigma }}^{1}(S)}\sigma ),\mathsf{P}_{S}^{\omega }\phi
\right\rangle _{\omega }\right\vert \lesssim \mathcal{F}_{\gamma
,\varepsilon }\left\Vert f\right\Vert _{L^{2}(\sigma )}\left\Vert \phi
\right\Vert _{L^{2}(\omega )}\,.  \label{e.622<}
\end{align}%
Together these prove \eqref{e.62<}. We treat them in turn.

Recalling the notation \eqref{e.betat}, we estimate 
\begin{align*}
\left\vert \left\langle H(\mathbf{1}_{\pi _{\mathcal{S}}^{1}(S)}\sigma ),%
\mathsf{P}_{S}^{\omega }\phi \right\rangle _{\omega }\right\vert &
=\left\vert \left\langle \mathsf{P}_{S}^{\omega }H(\mathbf{1}_{\pi _{%
\mathcal{S}}^{1}(S)}\sigma ),\mathsf{P}_{S}^{\omega }\phi \right\rangle
_{\omega }\right\vert \\
& \leq \beta (S)^{1/2}\left\Vert \mathsf{P}_{S}^{\omega }\phi \right\Vert
_{L^{2}(\omega )}\,.
\end{align*}%
The latter projections are mutually orthogonal so we can estimate 
\begin{align*}
\lvert A_{1}^{7}\rvert & \leq \left[ \sum_{S\in \mathcal{S}}\beta
(S)\left\vert \mathbb{E}_{\pi _{\mathcal{D}^{\sigma }}^{1}\left( S\right)
}^{\sigma }f\right\vert ^{2}\right] ^{1/2}\left\Vert \phi \right\Vert
_{L^{2}(\omega )} \\
& \lesssim (\mathcal{H}+\mathcal{F}_{\gamma ,\varepsilon })\left\Vert
f\right\Vert _{L^{2}(\sigma )}\left\Vert \phi \right\Vert _{L^{2}(\omega
)}\,.
\end{align*}%
We have appealed to the Carleson measure estimate \eqref{e.betat<} to get
the $\left\Vert f\right\Vert _{L^{2}(\sigma )}$ term. This proves %
\eqref{e.6s1<}.

\medskip

The argument for \eqref{e.622<} is similar. Recalling the notation %
\eqref{e.gammat}, we have 
\begin{align*}
\left\vert \left\langle H(\mathbf{1}_{\pi _{\mathcal{S}}^{1}(S)\backslash
\pi _{\mathcal{D}^{\sigma }}^{1}(S)}\sigma ),\mathsf{P}_{S}^{\omega }\phi
\right\rangle _{\omega }\right\vert & =\left\vert \left\langle \mathsf{P}%
_{S}^{\omega }H(\mathbf{1}_{\pi _{\mathcal{S}}^{1}(S)\backslash \pi _{%
\mathcal{D}^{\sigma }}^{1}(S)}\sigma ),\mathsf{P}_{S}^{\omega }\phi
\right\rangle _{\omega }\right\vert \\
& \leq \gamma (S)^{1/2}\left\Vert \mathsf{P}_{S}^{\omega }\phi \right\Vert
_{L^{2}(\omega )}\,.
\end{align*}%
We estimate 
\begin{align*}
\lvert A_{2}^{7}\rvert & \leq \left[ \sum_{S\in \mathcal{S}}\gamma
(S)\left\vert \mathbb{E}_{\pi _{\mathcal{D}^{\sigma }}^{1}\left( S\right)
}^{\sigma }f\right\vert ^{2}\right] ^{1/2}\left\Vert \phi \right\Vert
_{L^{2}(\omega )} \\
& \lesssim \mathcal{F}_{\gamma ,\varepsilon }\left\Vert f\right\Vert
_{L^{2}(\sigma )}\left\Vert \phi \right\Vert _{L^{2}(\omega )}\,.
\end{align*}%
We have appealed to the Carleson measure estimate \eqref{e.gammat<} to get
the $\left\Vert f\right\Vert _{L^{2}(\sigma )}$ term.

\subsection{The Term $A ^{6} _{1}$}

\label{s.61}

In the definition of $A_{1}^{6}(S)$, see \eqref{e.a61s}, note that the
difference of expectations depends upon a single Haar coefficient, the one
for the dyadic interval $\pi _{\mathcal{D}^{\sigma }}^{2}(\pi _{\mathcal{S}%
}^{t}(S))$. To be explicit, we will have the following equality 
\begin{align*}
\mathbb{E}_{\pi _{\mathcal{D}^{\sigma }}^{2}(S)}^{\sigma }f- \mathbb{E}_{\pi
_{\mathcal{D}^{\sigma }}^{1}(S)}^{\sigma }f &= - \mathbb{E}_{\pi _{\mathcal{D%
}^{\sigma }}^{1}(S)}^{\sigma } \Delta ^{\sigma } _{\pi _{\mathcal{D}^{\sigma
}}^{2}(S)} , \qquad S\in \mathcal{S }\,.
\end{align*}

We reindex the sum defining $A_{1}^{6}$ as follows. From \eqref{e.a61s}, we
write%
\begin{equation*}
-A_{1}^{6}(S)=\sum_{t=1}^{\rho (S)}\mathbb{E}_{\pi _{\mathcal{D}^{\sigma
}}^{1}(\pi _{\mathcal{S}}^{t-1}(S))}^{\sigma }\Delta _{\pi _{\mathcal{D}%
^{\sigma }}^{2}(\pi _{\mathcal{S}}^{t-1}(S))}^{\sigma }f\cdot \left\langle H(%
\mathbf{1}_{\pi _{\mathcal{S}}^{t}(S)}\sigma ),\mathsf{P}_{S}^{\omega }\phi
\right\rangle _{\omega }=A_{3}^{7}(S)+A_{4}^{7}(S)\,,
\end{equation*}%
\begin{equation}
A_{3}^{7}(S)\equiv \sum_{t=1}^{\rho (S)}\mathbb{E}_{\pi _{\mathcal{D}%
^{\sigma }}^{1}(\pi _{\mathcal{S}}^{t-1}(S))}^{\sigma }\Delta _{\pi _{%
\mathcal{D}^{\sigma }}^{2}(\pi _{\mathcal{S}}^{t-1}(S))}^{\sigma }f\cdot
\left\langle H(\mathbf{1}_{\pi _{\mathcal{S}}^{t}(S)\backslash \pi _{%
\mathcal{S}}^{t-1}(S)}\sigma ),\mathsf{P}_{S}^{\omega }\phi \right\rangle
_{\omega }  \label{e.73}
\end{equation}%
\begin{equation}
A_{4}^{7}(S)\equiv \sum_{t=1}^{\rho (S)}\mathbb{E}_{\pi _{\mathcal{D}%
^{\sigma }}^{1}(\pi _{\mathcal{S}}^{t-1}(S))}^{\sigma }\Delta _{\pi _{%
\mathcal{D}^{\sigma }}^{2}(\pi _{\mathcal{S}}^{t-1}(S))}^{\sigma }f\cdot
\left\langle H(\mathbf{1}_{\pi _{\mathcal{S}}^{t-1}(S)}\sigma ),\mathsf{P}%
_{S}^{\omega }\phi \right\rangle _{\omega }  \label{e.74}
\end{equation}

We argue that 
\begin{equation}
\left\vert \sum_{S\in \mathcal{S}-\{S_{0}\}}A_{3}^{7}(S)\right\vert \lesssim 
\mathcal{F}_{\gamma ,\varepsilon }\left\Vert f\right\Vert _{L^{2}(\sigma
)}\left\Vert \phi \right\Vert _{L^{2}(\omega )}\ ,  \label{e.73<}
\end{equation}%
\begin{equation}
\left\vert \sum_{S\in \mathcal{S}-\{S_{0}\}}A_{4}^{7}(S)\right\vert \lesssim 
\mathcal{F}_{\gamma ,\varepsilon }\left\Vert f\right\Vert _{L^{2}(\sigma
)}\left\Vert \phi \right\Vert _{L^{2}(\omega )}\,.  \label{e.74<}
\end{equation}%
Indeed, the first inequality \eqref{e.73<} is easier than the argument for %
\eqref{e.63<}, due to the extra orthogonality present with the Haar
difference applied to $f$ in \eqref{e.73}. We omit the proof.

\medskip

We turn to the proof of \eqref{e.74<}, and will need to appeal to our Energy
Hypothesis again. Begin by reindexing the sum. We define 
\begin{gather}
A_{4}^{7}(S,t)\equiv \mathbb{E}_{\pi _{\mathcal{D}^{\sigma
}}^{1}(S)}^{\sigma }\Delta _{\pi _{\mathcal{D}^{\sigma }}^{2}(S)}^{\sigma
}f\cdot \sum_{\substack{ S^{\prime }\in \mathcal{S}  \\ \pi ^{t-1}(S^{\prime
})=S}}\left\langle H(\mathbf{1}_{S}\sigma ),\mathsf{P}_{S^{\prime }}^{\omega
}\phi \right\rangle _{\omega } \\
\left\vert \sum_{S\in \mathcal{S}}A_{4}^{7}(S,t)\right\vert \lesssim
2^{-\gamma t/2}\mathcal{F}_{\gamma ,\varepsilon }\left\Vert f\right\Vert
_{L^{2}(\sigma )}\left\Vert \phi \right\Vert _{L^{2}(\omega )}\,,\qquad
t\geq 1\,.  \label{e.74t}
\end{gather}%
(The decay in $t$ is slightly worse in this case than in others.) Indeed, we
first exploit the implicitly orthogonality in the sum. Note that we will have 
\begin{gather*}
\sum_{S\in \mathcal{S}}\lvert \langle f,h_{\pi _{\mathcal{D}^{\sigma
}}^{2}(S)}^{\sigma }\rangle _{\omega }\rvert ^{2}\leq \left\Vert
f\right\Vert _{L^{2}(\sigma )}^{2}\,, \\
\sum_{S\in \mathcal{S}}\left\Vert \sum_{S^{\prime }\in \mathcal{S}:\pi
^{t-1}(S^{\prime })=S}\mathsf{P}_{S^{\prime }}^{\omega }\phi \right\Vert
_{L^{2}(\omega )}^{2}\leq \left\Vert \phi \right\Vert _{L^{2}(\omega )}\,.
\end{gather*}%
We also have from \eqref{e.Eh}, that 
\begin{equation*}
\left\vert \mathbb{E}_{\pi _{\mathcal{D}^{\sigma }}^{1}(S)}^{\sigma }h_{\pi
_{\mathcal{D}^{\sigma }}^{2}(S)}^{\sigma }\right\vert \leq \lvert \pi _{%
\mathcal{D}^{\sigma }}^{1}(S)\rvert ^{-1/2}
\end{equation*}%
Combining these facts, we see that \eqref{e.74t} follows from the estimate 
\begin{equation}
\left\Vert \sum_{S^{\prime }\in \mathcal{S}:\pi ^{t-1}(S^{\prime })=S}%
\mathsf{P}_{S^{\prime }}^{\omega }H(\mathbf{1}_{\pi _{\mathcal{S}%
}^{t-1}(S)}\sigma )\right\Vert _{L^{2}(\omega )}^{2}\lesssim 2^{-\gamma t}(%
\mathcal{H}^{2}+\mathcal{F}_{\gamma ,\varepsilon }^{2})\lvert \pi _{\mathcal{%
D}^{\sigma }}^{1}(S)\rvert \,,\qquad S\in \mathcal{S},t\geq 1\,.
\label{e.74<<}
\end{equation}

We turn to the proof of this last estimate. We will need geometric decay
from two different sources. One is the geometric decay in \eqref{e.fStopping}%
, and the second is the application of the Energy Hypothesis, as in the
proof of Theorem~\ref{t.miraculous}. Fix $S\in \mathcal{S}$, and integer $%
u\simeq \frac{t-1}{2}$, and let $\mathcal{S}_{u}$ be those $S^{\prime }\in 
\mathcal{S}$ with $\pi _{\mathcal{S}}^{u}(S^{\prime })=S$. We have 
\begin{gather}
\left\Vert \sum_{S^{\prime }\in \mathcal{S}:\pi ^{t-1}(S^{\prime })=S}%
\mathsf{P}_{S^{\prime }}^{\omega }H(\mathbf{1}_{\pi _{\mathcal{S}%
}^{t-1}(S)}\sigma )\right\Vert _{L^{2}(\omega )}^{2}=\sum_{S^{\prime }\in 
\mathcal{S}_{u}}B(S^{\prime }) \\
B(S^{\prime })\equiv \left\Vert \sum_{S^{\prime \prime }\in \mathcal{S}:\pi
^{t-1-u}(S^{\prime \prime })=S^{\prime }}\mathsf{P}_{S^{\prime \prime
}}^{\omega }H(\mathbf{1}_{\pi _{\mathcal{S}}^{t-1}(S)}\sigma )\right\Vert
_{L^{2}(\omega )}^{2}
\end{gather}%
Now, in the definition of $B(S^{\prime })$, we adjust the argument of the
Hilbert transform, writing $B(S^{\prime })=B_{1}(S^{\prime
})+B_{2}(S^{\prime \prime })$, where 
\begin{align}
B_{1}(S^{\prime })& \equiv \left\Vert \sum_{S^{\prime \prime }\in \mathcal{S}%
:\pi ^{t-1-u}(S^{\prime \prime })=S^{\prime }}\mathsf{P}_{S^{\prime \prime
}}^{\omega }H(\mathbf{1}_{\pi _{\mathcal{S}}^{t-1}(S)\backslash S^{\prime
}}\sigma )\right\Vert _{L^{2}(\omega )}^{2}  \label{e.B2d} \\
B_{2}(S^{\prime })& \equiv \left\Vert \sum_{S^{\prime \prime }\in \mathcal{S}%
:\pi ^{t-1-u}(S^{\prime \prime })=S^{\prime }}\mathsf{P}_{S^{\prime \prime
}}^{\omega }H(\mathbf{1}_{S^{\prime }}\sigma )\right\Vert _{L^{2}(\omega
)}^{2}
\end{align}

Now, by the testing condition \eqref{e.H1}, we have 
\begin{align}  \label{e.B2<}
\sum _{S^{\prime }\in \mathcal{S }_{u}} B _2(S^{\prime }) & \leq \mathcal{H }%
^2 \sum _{S^{\prime }\in \mathcal{S }_{u}} \sigma (S^{\prime }) \\
& \leq 2 ^{- u/2} \mathcal{H }^2 \sigma (S) \leq 2 ^{- u/2} \mathcal{H }^2
\sigma ( \pi ^{1} _{\mathcal{D }^ \sigma }(S))
\end{align}
where we have appealed to the Carleson measure property of the measure $%
\sigma $ on the stopping cubes, more precisely \eqref{e.14}, to deduce the
last line. This proves half of \eqref{e.74<<}.

\smallskip We use the notation \eqref{e.maxJ}, and apply \eqref{e.minimal<}
to see that 
\begin{align}
\sum_{S^{\prime }\in \mathcal{S}_{u}}B_{1}(S^{\prime })& =\sum_{S^{\prime
}\in \mathcal{S}_{u}}\sum_{S^{\prime \prime }\in \mathcal{S}:\pi
^{t-1-u}(S^{\prime \prime })=S^{\prime }}\sum_{J\in \mathcal{J}(S^{\prime
\prime })}\lVert \mathsf{P}_{J}^{\omega }H(\mathbf{1}_{\pi _{\mathcal{S}%
}^{t-1}(S)\backslash S^{\prime }}\sigma )\rVert _{L^{2}(\omega )}^{2} \\
& \lesssim \mathcal{F}_{\gamma ,\varepsilon }^{2}2^{-\gamma t/2}\sigma (\pi
_{\mathcal{D}^{\sigma }}^{1}(S))\,.
\end{align}%
This completes the proof of \eqref{e.74<<}.

\section{The Remaining Estimates}

\label{s.remain}

We collect together the estimates claimed in earlier sections. The estimates
in the first two subsections below are in \cite{Vol}, and the remaining
three subsections essentially follow the arguments in \cite{Vol} but using
the Energy Hypothesis in \S \ref{s.43}.

\subsection{$A ^{4} _{3}$: The Stopping Terms}

\label{s.43}

To control \eqref{e.Stopping}, and prove \eqref{e.43<}, it is important that
we are dealing with the Energy Hypothesis \eqref{smallest condition}.

We first claim that for $S\in \mathcal{S}$ and $s\geq 0$ an integer%
\begin{align}
A_{3}^{4}(S,s)& \equiv \sum_{(I,J)\in \mathcal{P}(S):\lvert J\rvert
=2^{-s}\lvert I\rvert }\left\vert \mathbb{E}_{I_{J}}^{\omega }\Delta
_{I}^{\sigma }f\cdot \left\langle H(\mathbf{1}_{S\backslash I_{J}}\sigma
),\Delta _{J}^{\omega }\phi \right\rangle _{\omega }\right\vert 
\label{e.stop<} \\
& \lesssim 2^{-\gamma s}\mathcal{F}_{\gamma ,\varepsilon }F(S)\Lambda
(S,s)\,, \\
F(S)^{2}& \equiv \sum_{I\in \mathcal{C}^{\sigma }(S)}\left\vert \left\langle
f,h_{I}^{\sigma }\right\rangle _{\sigma }\right\vert ^{2} \\
\Lambda (S,s)^{2}& \equiv \sum_{I\in \mathcal{C}^{\sigma
}(S)}\sum_{J\;:\;(I,J)\in \mathcal{P}(S):\lvert J\rvert =2^{-s}\lvert
I\rvert }\left\vert \left\langle \phi ,h_{J}^{\omega }\right\rangle _{\omega
}\right\vert ^{2}\,.
\end{align}%
Indeed, apply Cauchy-Schwarz in the $I$ variable above to obtain, and appeal
to \eqref{e.Eh} to see that 
\begin{equation*}
A_{3}^{4}(S,s)\leq F(S)\left[ \sum_{I\in \mathcal{C}^{\sigma }(S)}\left(
\sum_{J\;:\;(I,J)\in \mathcal{P}(S):\lvert J\rvert =2^{-s}\lvert I\rvert }%
\frac{1}{\sigma (I_{J})^{1/2}}\left\vert \left\langle H(\mathbf{1}%
_{S\backslash I_{J}}\sigma ),\Delta _{J}^{\omega }\phi \right\rangle
_{\omega }\right\vert \right) ^{2}\right] ^{\frac{1}{2}},
\end{equation*}%
We can then estimate the sum inside the braces by%
\begin{align*}
& \sum_{I\in \mathcal{C}^{\sigma }(S)}\sum_{J\;:\;(I,J)\in \mathcal{P}%
(S):\lvert J\rvert =2^{-s}\lvert I\rvert }\lvert \left\langle \phi
,h_{J}^{\omega }\right\rangle _{\omega }\rvert ^{2} \\
& \times \sum_{J\;:\;(I,J)\in \mathcal{P}(S):\lvert J\rvert =2^{-s}\lvert
I\rvert }\frac{1}{\sigma (I_{J})}\cdot \left\vert \left\langle H(\mathbf{1}%
_{S\backslash I_{J}}\sigma ),h_{J}^{\omega }\phi \right\rangle _{\omega
}\right\vert ^{2} \\
& \lesssim \Lambda (S,s)^{2}\cdot A(S,s) \\
A(S,s)& \equiv \sup_{I\in \mathcal{C}^{\sigma }(S)}\sum_{J\;:\;(I,J)\in 
\mathcal{P}(S):\lvert J\rvert =2^{-s}\lvert I\rvert }\sigma
(I_{J})^{-1}\cdot \left\vert \left\langle H(\mathbf{1}_{S\backslash
I_{J}}\sigma ),h_{J}^{\omega }\phi \right\rangle _{\omega }\right\vert
^{2}\,.
\end{align*}

We turn to the analysis of the supremum in last display. We denote the two
chilren of $I$ by $I_{\theta }$ for $\theta \in \{-,+\}$. Using %
\eqref{e.Eimplies} and then the third inequality in \eqref{Psi properties},
we have 
\begin{eqnarray*}
A(S,s) &\lesssim &\sup_{I\in \mathcal{C}^{\sigma }(S)}\sup_{\theta \in
\{-,+\}}\sigma (I_{\theta })^{-1}\sum_{J\;:\;(I,J)\in \mathcal{P}%
(S):I_{J}=I_{\theta }\lvert J\rvert =2^{-s}\lvert I\rvert }\Phi \left(
J,S\setminus I_{\theta }\right) \\
&\lesssim &\sup_{I\in \mathcal{C}^{\sigma }(S)}\sup_{\theta \in
\{-,+\}}\sigma (I_{\theta })^{-1}2^{-\gamma s}\Psi _{\gamma ,\varepsilon
}\left( I_{\theta },S\right) \\
&\lesssim &2^{-\gamma s}\mathcal{F}_{\gamma ,\varepsilon }^{2}\,.
\end{eqnarray*}%
The third inequality is the one for which the definition of stopping
intervals was designed to deliver: From Definition~\ref{d.corona}, as $%
(I,J)\in \mathcal{P}(S)$, we have that $S$ is the $\mathcal{S}$-parent of $%
I_{J}$, hence $I_{J}$ was \emph{not} a stopping interval, that is %
\eqref{e.stoppingDef} does not hold, delivering the estimate above.

\medskip

We clearly have from \eqref{Plancherel} that 
\begin{equation*}
\sum_{S\in \mathcal{S}}F(S)^{2}\leq \sum_{I}\left\vert \left\langle
f,h_{I}^{\sigma }\right\rangle _{\sigma }\right\vert ^{2}=\left\Vert
f\right\Vert _{L^{2}(\sigma )}^{2}\,.
\end{equation*}%
And so we have from \eqref{e.stop<}, 
\begin{align*}
\sum_{S\in \mathcal{S}}\sum_{s=0}^{\infty }A_{3}^{4}(S,s)& \lesssim \mathcal{%
F}_{\gamma ,\varepsilon }\left\Vert f\right\Vert _{L^{2}(\sigma )}\left(
\sum_{S\in \mathcal{S}}\sum_{s=0}^{\infty }2^{-\gamma s}\Lambda
(S,s)^{2}\right) ^{1/2} \\
& \lesssim \mathcal{F}_{\gamma ,\varepsilon }\left\Vert f\right\Vert
_{L^{2}(\sigma )}\left\Vert \phi \right\Vert _{L^{2}(\omega )}\,.
\end{align*}

\subsection{$A ^{2} _{1}$: Diagonal Short Range Terms}

\label{s.21}

To prove \eqref{A21<}, let us recall the definition \eqref{A21}. The pairs
of intervals $I,J$ arise from the dyadic grids $\mathcal{D}^{\sigma }$ and $%
\mathcal{D}^{\omega }$ respectively. But these grids share a common set of
endpoints of the intervals. And the intervals $I,J$ have comparable lengths, 
$2^{-r}\lvert I\rvert \leq \lvert J\rvert \leq \lvert I\rvert $. Accordingly
, these pairs of intervals satisfy the conditions of the weak boundedness
condition \eqref{e.W}. A Haar function $h_{I}^{\sigma }$ is a linear
combination of its children, and the children of $I$ and $J$ also satisfy
the weak boundedness condition \eqref{e.W}. From this, we see that 
\begin{equation*}
\lvert \left\langle H(\sigma h_{I}^{\sigma }),h_{J}^{\omega }\right\rangle
_{\omega }\rvert \leq 4\mathcal{W}\,,\qquad (I,J)\in \mathcal{A}_{1}^{2}\,.
\end{equation*}%
The Schur test easily implies that 
\begin{equation*}
A_{1}^{2}\leq 4\mathcal{W}\sum_{I\in \mathcal{D}^{\sigma }}\lvert
\left\langle f,h_{I}^{\sigma }\right\rangle _{\sigma }\rvert
\sum_{J\;:\;(I,J)\in \mathcal{A}_{1}^{2}}\lvert \left\langle \phi
,h_{J}^{\omega }\right\rangle _{\omega }\rvert \lesssim \mathcal{W}%
\left\Vert f\right\Vert _{L^{2}(\sigma )}\left\Vert \phi \right\Vert
_{L^{2}(\omega )}\,.
\end{equation*}

\subsection{ $A_{2}^{2}$: The Long Range Term }
\label{s.22}

We prove the estimate \eqref{A22<}. Recall that the pairs of intervals $I,J$
in question satisfy $\lvert J\rvert \leq \lvert I\rvert $ and $\textup{dist}%
(I,J)\geq \lvert I\rvert $. The hypotheses of \eqref{e.Eimplies} are in
force, in particular \eqref{e.Jgood} holds.

We observe that the Energy Lemma can be applied to estimate the inner
product $\left\langle H(h_{I}^{\sigma }\sigma ),h_{J}^{\omega }\right\rangle
_{\omega }$. To see this, note that $h_{I}^{\sigma }$ is constant on each
child $I_{\pm }$. So, take a child $I_{\theta }$, and apply the Energy Lemma
with the largest interval $\widehat{I}$ taken to be 
\begin{equation*}
\widehat{I}=\textup{hull}\left[ I_{\theta },\left( \frac{\lvert I\rvert }{%
\lvert J\rvert }\right) ^{1-\varepsilon }J\right] \,.
\end{equation*}%
Here $\lambda J$ means the interval with the same center as $J$ and length
equal to $\lambda \lvert J\rvert $. The two intervals $I_{\theta }$ and $%
\left( \frac{\lvert I\rvert }{\lvert J\rvert }\right) ^{1-\varepsilon }J$
are disjoint. We take $I^{\prime }\subset \widehat{I}$ so that $\widehat{I}%
\backslash I^{\prime }=I_{\theta }$. Then, the Energy Lemma (\ref{e.Eimplies}%
) and \eqref{e.Eh} apply to give us the estimate below. 
\begin{align}
\beta (I,J)& \equiv \left\vert \sum_{\theta }\left\langle H(\mathbf{1}%
_{I_{\theta }}h_{I}^{\sigma }\sigma ),h_{J}^{\omega }\right\rangle _{\omega
}\right\vert \leq \left\vert \mathbb{E}_{I_{\theta }}^{\sigma }h_{I}^{\sigma
}\right\vert \sum_{\theta }\left\vert \left\langle H(\mathbf{1}_{I_{\theta
}}\sigma ),h_{J}^{\omega }\right\rangle _{\omega }\right\vert  \label{e.za<}
\\
& \lesssim \sum_{\theta }\left[ \frac{\omega (J)}{\sigma (I_{\theta })}%
\right] ^{1/2}\mathsf{E}(J,\omega )\mathsf{P}(J,\mathbf{1}_{\widehat{I}%
\backslash I^{\prime }}\sigma ) \\
& \lesssim \sum_{\theta }\omega (J)^{1/2}\sigma (I_{\theta })^{1/2}\cdot 
\frac{\lvert J\rvert }{\textup{dist}(I,J)^{2}}\,.
\end{align}%
We have used the trivial inequalities $\mathsf{E}(\omega ,J)\leq 1$ and $%
\mathsf{P}(J,\mathbf{1}_{I_{\theta }}\sigma )\leq \frac{\lvert J\rvert }{%
\textup{dist}(I,J)^{2}}\sigma (I_{\theta })$.

We may assume that $\left\Vert f\right\Vert _{L^{2}\left( \sigma \right)
}^{2}=\left\Vert \phi \right\Vert _{L^{2}\left( \omega \right) }^{2}=1$. We
then estimate%
\begin{eqnarray*}
\left\vert A_{2}^{2}\right\vert &\leq &\sum_{I}\sum_{J\;:\;\lvert J\rvert
\leq \lvert I\rvert :\textup{dist}(I,J)\geq \lvert I\rvert }\lvert
\left\langle f,h_{I}^{\sigma }\right\rangle _{\sigma }\rvert \beta
(I,J)\lvert \left\langle \phi ,h_{J}^{\omega }\right\rangle _{\omega }\rvert
\\
&\lesssim &\sum_{I}\sum_{J\;:\;\lvert J\rvert \leq \lvert I\rvert :\textup{%
dist}(I,J)\geq \lvert I\rvert }\lvert \left\langle f,h_{I}^{\sigma
}\right\rangle _{\sigma }\rvert \sigma (I)^{\frac{1}{2}}\frac{\lvert J\rvert 
}{\textup{dist}(I,J)^{2}}\omega (J)^{\frac{1}{2}}\lvert \left\langle \phi
,h_{J}^{\omega }\right\rangle _{\omega }\rvert \\
&\lesssim &\sum_{I}\lvert \left\langle f,h_{I}^{\sigma }\right\rangle
_{\sigma }\rvert ^{2}\sum_{J\;:\;\lvert J\rvert \leq \lvert I\rvert :%
\textup{dist}(I,J)\geq \lvert I\rvert }\left( \frac{\lvert J\rvert }{\lvert
I\rvert }\right) ^{-\delta }\sigma (I)^{\frac{1}{2}}\frac{\lvert J\rvert }{%
\textup{dist}(I,J)^{2}}\omega (J)^{\frac{1}{2}} \\
&&+\sum_{J}\lvert \left\langle \phi ,h_{J}^{\omega }\right\rangle _{\omega
}\rvert ^{2}\sum_{I\;:\;\lvert J\rvert \leq \lvert I\rvert :\textup{dist}%
(I,J)\geq \lvert I\rvert }\left( \frac{\lvert J\rvert }{\lvert I\rvert }%
\right) ^{\delta }\sigma (I)^{\frac{1}{2}}\frac{\lvert J\rvert }{\textup{%
dist}(I,J)^{2}}\omega (J)^{\frac{1}{2}},
\end{eqnarray*}%
where we have inserted the gain and loss factors $\left( \frac{\lvert
J\rvert }{\lvert I\rvert }\right) ^{\pm \delta }$ with $0<\delta <1$ to
facilitate application of Schur's test. For each fixed $I$ we have%
\begin{eqnarray*}
&&\sum_{J\;:\;\lvert J\rvert \leq \lvert I\rvert :\textup{dist}(I,J)\geq
\lvert I\rvert }\left( \frac{\lvert J\rvert }{\lvert I\rvert }\right)
^{\delta }\sigma (I)^{\frac{1}{2}}\frac{\lvert J\rvert }{\textup{dist}%
(I,J)^{2}}\omega (J)^{\frac{1}{2}} \\
&\lesssim &\sigma (I)^{\frac{1}{2}}\sum_{k=0}^{\infty }2^{-k\delta }\left(
\sum_{J\;:\;2^{k}\lvert J\rvert =\lvert I\rvert :\textup{dist}(I,J)\geq
\lvert I\rvert }\frac{\lvert J\rvert }{\textup{dist}(I,J)^{2}}\omega
(J)\right) ^{\frac{1}{2}} \\
&&\times \left( \sum_{J\;:\;2^{k}\lvert J\rvert =\lvert I\rvert :\textup{%
dist}(I,J)\geq \lvert I\rvert }\frac{\lvert J\rvert }{\textup{dist}(I,J)^{2}%
}\right) ^{\frac{1}{2}},
\end{eqnarray*}%
which is bounded by%
\begin{equation*}
\sum_{k=0}^{\infty }2^{-k\delta }\left( \frac{\sigma (I)}{\left\vert
I\right\vert }\mathsf{P}\left( I,\omega \right) \right) ^{\frac{1}{2}%
}\lesssim \mathcal{A}_{2},
\end{equation*}%
if $\delta >0$. For each fixed $J$ we have%
\begin{eqnarray*}
&&\sum_{I\;:\;\lvert J\rvert \leq \lvert I\rvert :\textup{dist}(I,J)\geq
\lvert I\rvert }\left( \frac{\lvert J\rvert }{\lvert I\rvert }\right)
^{-\delta }\sigma (I)^{\frac{1}{2}}\frac{\lvert J\rvert }{\textup{dist}%
(I,J)^{2}}\omega (J)^{\frac{1}{2}} \\
&\lesssim &\omega (J)^{\frac{1}{2}}\sum_{k=0}^{\infty }2^{-k\left( 1-\delta
\right) }\sum_{I\;:\;2^{k}\lvert J\rvert =\lvert I\rvert :\textup{dist}%
(I,J)\geq \lvert I\rvert }\frac{\lvert I\rvert }{\textup{dist}(I,J)^{2}}%
\sigma (I)^{\frac{1}{2}} \\
&\lesssim &\omega (J)^{\frac{1}{2}}\sum_{k=0}^{\infty }2^{-k\left( 1-\delta
\right) }\left( \sum_{I\;:\;2^{k}\lvert J\rvert =\lvert I\rvert :\textup{%
dist}(I,J)\geq \lvert I\rvert }\frac{\lvert I\rvert }{\textup{dist}(I,J)^{2}%
}\sigma (I)\right) ^{\frac{1}{2}} \\
&&\times \left( \sum_{I\;:\;2^{k}\lvert J\rvert =\lvert I\rvert :\textup{%
dist}(I,J)\geq \lvert I\rvert }\frac{\lvert I\rvert }{\textup{dist}(I,J)^{2}%
}\right) ^{\frac{1}{2}},
\end{eqnarray*}%
which is bounded by%
\begin{eqnarray*}
&&\omega (J)^{\frac{1}{2}}\sum_{k=0}^{\infty }2^{-k\left( 1-\delta \right) }%
\mathsf{P}\left( 2^{k}J,\sigma \right) ^{\frac{1}{2}}\left( \frac{1}{%
\left\vert 2^{k}J\right\vert }\right) ^{\frac{1}{2}} \\
&\lesssim &\sum_{k=0}^{\infty }2^{-k\left( 1-\delta \right) }\left( \frac{%
\omega (2^{k}J)}{\left\vert 2^{k}J\right\vert }\mathsf{P}\left(
2^{k}J,\sigma \right) \right) ^{\frac{1}{2}}\lesssim \mathcal{A}_{2},
\end{eqnarray*}%
if $\delta <1$. With any fixed $0<\delta <1$ we obtain from the inequalities
above that 
\begin{eqnarray*}
\left\vert A_{2}^{2}\right\vert &\lesssim &\sum_{I}\lvert \left\langle
f,h_{I}^{\sigma }\right\rangle _{\sigma }\rvert ^{2}\mathcal{A}%
_{2}+\sum_{J}\lvert \left\langle \phi ,h_{J}^{\omega }\right\rangle _{\omega
}\rvert ^{2}\mathcal{A}_{2} \\
&=&\left( \left\Vert f\right\Vert _{L^{2}\left( \sigma \right)
}^{2}+\left\Vert \phi \right\Vert _{L^{2}\left( \omega \right) }^{2}\right) 
\mathcal{A}_{2}=2\mathcal{A}_{2}\left\Vert f\right\Vert _{L^{2}\left( \sigma
\right) }^{2}\left\Vert \phi \right\Vert _{L^{2}\left( \omega \right) }^{2},
\end{eqnarray*}%
since we assumed $\left\Vert f\right\Vert _{L^{2}\left( \sigma \right)
}^{2}=\left\Vert \phi \right\Vert _{L^{2}\left( \omega \right) }^{2}=1$.

\subsection{$A ^{3}_{1}$\color{blue} The Mid-Range Term}

\label{s.31}

We control the term associated with \eqref{e.a31}, namely we prove (\ref%
{e.a31<}). For integers $s\geq r$, set 
\begin{align}
A_{1}^{3}(s)& \equiv \sum_{I}\sum_{J\;:\;2^{s}\lvert J\rvert =\lvert I\rvert
:\textup{dist}(I,J)\leq \lvert I\rvert \,,\ I\cap J=\emptyset }\lvert
\left\langle f,h_{I}^{\sigma }\right\rangle _{\sigma }\left\langle
H(h_{I}^{\sigma }\sigma ),h_{J}^{\omega }\right\rangle _{\omega
}\left\langle \phi ,h_{J}^{\omega }\right\rangle _{\omega }\rvert \\
& \lesssim \left\Vert f\right\Vert _{L^{2}(\sigma )}\left[ \sum_{I}\left(
\sum_{J\;:\;2^{s}\lvert J\rvert =\lvert I\rvert :\textup{dist}(I,J)\leq
\lvert I\rvert \,,\ I\cap J=\emptyset }\lvert \left\langle H(h_{I}^{\sigma
}\sigma ),h_{J}^{\omega }\right\rangle _{\omega }\left\langle \phi
,h_{J}^{\omega }\right\rangle _{\omega }\rvert \right) ^{2}\right] ^{1/2} \\
& \lesssim \Lambda (s)\left\Vert f\right\Vert _{L^{2}(\sigma )}\left\Vert
\phi \right\Vert _{L^{2}(\omega )}\,, \\
\Lambda (s)^{2}& \equiv 2^{s}\sup_{I}\sum_{J\;:\;2^{s}\lvert J\rvert =\lvert
I\rvert :\textup{dist}(I,J)\leq \lvert I\rvert \,,\ I\cap J=\emptyset
}\lvert \left\langle H(h_{I}^{\sigma }\sigma ),h_{J}^{\omega }\right\rangle
_{\omega }\rvert ^{2}
\end{align}%
since, by \eqref{Plancherel}, 
\begin{equation*}
\sum_{I}\sum_{J\;:\;2^{s}\lvert J\rvert =\lvert I\rvert :\textup{dist}%
(I,J)\leq \lvert I\rvert \,,\ I\cap J=\emptyset }\left\vert \left\langle
\phi ,h_{J}^{\omega }\right\rangle _{\omega }\right\vert
^{2}=2^{s}\sum_{I}\sum_{J}\left\vert \left\langle \phi ,h_{J}^{\omega
}\right\rangle _{\omega }\right\vert ^{2}=2^{s}\left\Vert \phi \right\Vert
_{L^{2}(\omega )}^{2}\,.
\end{equation*}%
Due to the `local' nature of the sum in $J$, we have thus gained a small
improvement in the Schur test to derive the last line.

But $J$ is good, so that \eqref{e.Eimplies} applies to each child $I_{\pm }$
of $I$ as in \eqref{e.za<} above. Hence, we have using \eqref{e.Jsimeq} that 
\begin{align*}
\Lambda (s)^{2}& \lesssim \sup_{I}2^{s}\sum_{\theta }\sum_{J\;:\;2^{s}\lvert
J\rvert =\lvert I\rvert :\textup{dist}(I,J)\leq \lvert I\rvert \,,\ I\cap
J=\emptyset }\frac{\omega (J)}{\sigma (I_{\theta })}\cdot \mathsf{E}%
(J,\omega )^{2}\cdot \mathsf{P}(J,\mathbf{1}_{I_{\theta }}\sigma )^{2} \\
& \lesssim \sup_{I}2^{s}\sum_{\theta }\sum_{J\;:\;2^{s}\lvert J\rvert
=\lvert I\rvert :\textup{dist}(I,J)\leq \lvert I\rvert \,,\ I\cap
J=\emptyset }\frac{\omega (J)}{\sigma (I_{\theta })}\left( \frac{\left\vert
J\right\vert }{\left\vert I\right\vert }\right) ^{2-2\varepsilon }\cdot 
\mathsf{P}(I_{\theta },\mathbf{1}_{I_{\theta }}\sigma )^{2} \\
& \lesssim \sup_{I}2^{s}2^{-s\left( 2-2\varepsilon \right) }\sum_{\theta }%
\frac{\sigma (I_{\theta })}{\lvert I\rvert ^{2}}\sum_{J\;:\;2^{s}\lvert
J\rvert =\lvert I\rvert :\textup{dist}(I,J)\leq \lvert I\rvert \,,\ I\cap
J=\emptyset }\omega (J) \\
& \lesssim 2^{-\left( 1-2\varepsilon \right) s}\mathcal{A}_{2}\,.
\end{align*}%
This is clearly a summable estimate in $s\geq r$, so the proof of %
\eqref{e.a31<} is complete.

\subsection{$A ^{4}_1$: The Neighbor Terms}

\label{s.41}

The neighbor terms are defined in \eqref{e.a32}, \eqref{e.neighbor}, %
\eqref{e.Neighbor}, and we are to prove \eqref{e.41<}. To recall, $I\in 
\mathcal{D}^{\sigma }$, $J\in \mathcal{D}^{\omega }$ is contained in $I$,
with $\lvert J\rvert <2^{-r}\lvert I\rvert $, and $I_{J}$ is the child of $I$
that contains $J$.

Fix $\theta\in \{-,+\}$, and an integer $s\geq r$. Below we will use the
convention that $I\backslash I_{\theta}=I_{-\theta }$. The inner product to
be estimated is that in \eqref{e.Neighbor}: 
\begin{eqnarray*}
\langle H\left( \mathbf{1}_{I_{-\theta }}\sigma \Delta _{I}^{\sigma
}f\right) ,\Delta _{J}^{\omega }\phi \rangle _{\omega } &=&\langle (\mathbf{1%
}_{I_{-\theta }}\sigma \Delta _{I}^{\sigma }f,H\left( \omega \Delta
_{J}^{\omega }\phi \right) \rangle _{\sigma } \\
&=&\mathbb{E}_{I_{-\theta }}^{\sigma }\Delta _{I}^{\sigma }f\cdot \langle (%
\mathbf{1}_{I_{-\theta }}\sigma ,H\left( \omega \Delta _{J}^{\omega }\phi
\right) \rangle _{\sigma } \\
&=&\mathbb{E}_{I_{-\theta }}^{\sigma }\Delta _{I}^{\sigma }f\cdot \langle
H\left( \mathbf{1}_{I_{-\theta }}\sigma \right) ,\Delta _{J}^{\omega }\phi
\rangle _{\sigma } .
\end{eqnarray*}

Use $\left\Vert \Delta _{J}^{\omega }\phi \right\Vert _{L^{2}\left( \omega
\right) }=\left\vert \left\langle \phi ,h_{J}^{\omega }\right\rangle _{\omega }
\right\vert $ and $\frac{\left\vert J\right\vert }{\left\vert I_{\theta
}\right\vert }=2^{-s}$ in the Energy Lemma with $J\subset I_{\theta }\subset
I$ to obtain 
\begin{align*}
\left\vert \langle H\left( \mathbf{1}_{I_{-\theta }}\sigma \right) ,\Delta
_{J}^{\omega }\phi \rangle _{\omega } \right\vert & \lesssim \left\vert \left\langle
\phi ,h_{J}^{\omega }\right\rangle _{\omega } \right\vert \omega \left( J\right) ^{%
\frac{1}{2}}\cdot E\left( J,\omega \right) \cdot \mathsf{P}\left( J,\mathbf{1%
}_{I_{-\theta }}\sigma \right)  \\
& \lesssim \left\vert \left\langle \phi ,h_{J}^{\omega }\right\rangle _{\omega }
\right\vert \omega \left( J\right) ^{\frac{1}{2}}\cdot 2^{-\left(
1-\varepsilon \right) s}\mathsf{P}\left( I_{-\theta },\mathbf{1}_{I_{-\theta
}}\sigma \right) 
\end{align*}%
Here, we are using $E\left( J,\omega \right) \leq 1$ and \eqref{e.Jsimeq},
which inequality applies since $J\subset I\setminus I_{-\theta }$.

In the sum below, we keep the length of the intervals $J$ fixed, and assume
that $J\subset I_{\theta }$. We estimate 
\begin{align*}
A_{1}^{4}(I,\theta ,s)& \equiv \sum_{J\;:\;2^{s}\lvert J\rvert =\lvert
I\rvert :J\subset I_{\theta }}\left\vert \langle H\left( \mathbf{1}%
_{I_{-\theta }}\sigma \Delta _{I}^{\sigma }f\right) ,\Delta _{J}^{\omega
}\phi \rangle _{\omega }\right\vert  \\
& \leq 2^{-\left( 1-\varepsilon \right) s}\vert \mathbb{E}_{I_{-\theta
}}^{\sigma }\Delta _{I}^{\sigma }f \vert \mathsf{P}(I_{\theta },\mathbf{%
1}_{I_{-\theta }}\sigma )\sum_{J\;:\;2^{s}\lvert J\rvert =\lvert I\rvert
:J\subset I_{\theta }}\lvert \langle \phi ,h_{J}^{\omega }\rangle _{\omega} \rvert
\omega (J)^{1/2} \\
& \leq 2^{-\left( 1-\varepsilon \right) s}\vert \mathbb{E}_{I_{-\theta
}}^{\sigma }\Delta _{I}^{\sigma }f \vert \mathsf{P}(I_{\theta },\mathbf{%
1}_{I_{-\theta }}\sigma )\omega (I_{\theta })^{1/2}\Lambda (I,\theta ,s), \\
\Lambda (I,\theta ,s) ^2 & \equiv \sum_{J\;:\;2^{s}\lvert J\rvert =\lvert
I\rvert :J\subset I_{\theta }}\left\vert \langle \phi ,h_{J}^{\omega
}\rangle  _{\omega }\right\vert ^{2}\,.
\end{align*}%
The last line follows upon using the Cauchy-Schwartz inequality.

Using \eqref{e.Eh}, we have 
\begin{equation}
\vert \mathbb{E}_{I_{-\theta }}^{\sigma }\Delta _{I}^{\sigma
}f\vert
\le \lvert \langle f,h_{I}^{\sigma }\rangle _{\sigma } \rvert \cdot
 \sigma (I_{-\theta })^{-1/2}
\label{e.haarAvg}
\end{equation}%
And so, we can estimate $A_{1}^{4}(I,\theta ,s)$ as follows, in which we use
the $A_{2}$ hypothesis \eqref{A2}: 
\begin{eqnarray*}
A_{1}^{4}(I,\theta ,s) &\lesssim &2^{-\left( 1-\varepsilon \right) s}\lvert
\langle f,h_{I}^{\sigma }\rangle _{\sigma} \rvert \Lambda (I,\theta ,s)\cdot 
 \sigma (I_{-\theta })^{-1/2}
\mathsf{P}(I_{\theta
},\mathbf{1}_{I_{-\theta }}\sigma )\omega (I_{\theta })^{1/2} \\
&\lesssim &\mathcal{A}_{2}2^{-\left( 1-\varepsilon \right) s}\lvert \langle
f,h_{I}^{\sigma}\rangle_{\sigma}\rvert \Lambda (I,\theta ,s)\,,
\end{eqnarray*}%
since $\mathsf{P}(I_{\theta },\mathbf{1}_{I_{-\theta }}\sigma )\lesssim 
\frac{\sigma (I_{-\theta })}{\left\vert I_{\theta }\right\vert }$ shows that 
\begin{equation*}
 \sigma (I_{-\theta })^{-1/2}\mathsf{P}%
(I_{\theta },\mathbf{1}_{I_{-\theta }}\sigma )\omega (I_{\theta
})^{1/2}\lesssim \frac{\sigma (I_{-\theta })^{1/2}\omega (I_{\theta })^{1/2}%
}{\left\vert I_{\theta }\right\vert }\lesssim \mathcal{A}_{2},
\end{equation*}%
A straight forward application of Cauchy-Schwartz then shows that 
\begin{equation*}
\sum_{I}A_{1}^{4}(I,\theta ,s)\lesssim \mathcal{A}_{2}2^{-\left(
1-\varepsilon \right) s}\lVert f\rVert _{L^{2}(\sigma )}\lVert \Lambda (I,\theta ,s) \rVert
_{L^{2}(\omega )}\,.
\end{equation*}%
This estimate is summable in $\theta \in \left\{ {-,+}\right\} $ and $s\geq r
$, so the proof of \eqref{e.41<} is complete.

\section{A Counterexample to the Pivotal Conditions}

\label{s.example}

We exhibit a weight pair $\left( \omega ,\sigma \right) $ that illustrates
the nature of the Energy Condition, and the subtlety of the two weight
problem in general. In particular it shows that the Pivotal Conditions are
not necessary for the two weight inequality \eqref{e.H<}.


\begin{theorem}
\label{t.example} There is a weight pair $(\omega ,\sigma )$ which satisfies
the two weight inequality \eqref{e.H<}, and fails the dual Pivotal Condition, 
namely \eqref{pivotalcondition} with the roles of $ \omega $ and $ \sigma $ reversed. 
\end{theorem}


Thus, this pair of weights satisfy the two weight inequality, but would not
be included in the analysis of \cite{NTV3}. We prove this result by
appealing to our Theorem~\ref{main}. In the course of the construction, we
will see that one can make seemingly small modifications of the example
measure $\sigma $, and in so doing violate the $L ^2 $ inequality.

The plan of the proof of the Theorem is to (1) construct the pair of
weights, and then to verify (2) the assertions on the Hybrid Conditions,
(3) the $A_{2}$ condition and (4) the two testing conditions \eqref{e.H1}
and \eqref{e.H2}. We take up these steps in the subsections below.


\subsection{Construction of the Pair of Weights}

Recall the middle-third Cantor set $\mathsf{E}$ and Cantor measure $\omega $
on the closed unit interval $I_{1}^{0}=\left[ 0,1\right] $. At the $k^{th}$
generation in the construction, there is a collection $\left\{
I_{j}^{k}\right\} _{j=1}^{2^{k}}$ of $2^{k}$ pairwise disjoint closed
intervals of length $\left\vert I_{j}^{k}\right\vert =\frac{1}{3^{k}}$. With 
$K_{k}=\bigcup_{j=1}^{2^{k}}I_{j}^{k}$, the Cantor set is defined by $%
\mathsf{E}=\bigcap_{k=1}^{\infty }K_{k}=\bigcap_{k=1}^{\infty }\left(
\bigcup_{j=1}^{2^{k}}I_{j}^{k}\right) $. The Cantor measure $\omega $ is the
unique probability measure supported in $\mathsf{E}$ with the property that
it is equidistributed among the intervals $\left\{ I_{j}^{k}\right\}
_{j=1}^{2^{k}}$ at each scale $k$, i.e.%
\begin{equation}
\omega (I_{j}^{k})=2^{-k},\ \ \ \ \ k\geq 0,1\leq j\leq 2^{k}.
\label{omega measure}
\end{equation}

We will define three measures $\sigma ,\dot{\sigma},\ddot{\sigma}$. We
denote the removed open middle third of $I_{j}^{k}$ by $G_{j}^{k}$. The
three measures, restricted to an interval $G_{j}^{k}$ will be a point mass
with weight that is only a function of $k$. The only distinction will be the
location of the point mass.

Let $\dot z_{j}^{k}\in G_{j}^{k}$ be the center of the interval $G_{j}^{k}$,
which is also the center of the interval $I_{j}^{k}$. Now we define 
\begin{equation}  \label{e.sigmaD}
\dot\sigma =\sum_{k,j}s_{j}^{k}\delta _{\dot z_{j}^{k}},
\end{equation}%
where the sequence of positive numbers $s_{j}^{k}$ is chosen to satisfy the
following precursor of the $A_{2}$ condition:%
\begin{equation*}
\frac{s_{j}^{k}\omega(I_{j}^{k})}{\vert I_{j}^{k}\vert ^{2}}=1,\ \ \
s_{j}^{k}=\left( \frac{1}{3}\right) ^{k}\left( \frac{2}{3}\right) ^{k}
\qquad k\geq 0,1\leq j\leq 2^{k}.
\end{equation*}
The self-similarity of this measure makes it useful in verifying the
counterexample. But, it appears that the pair of weights $(\omega , \sigma )$
\emph{do not} satisfy the two weight inequality \eqref{e.H}.

The construction of the other two example measures is closely related to the
structure of the function $H \omega $. On each interval $G_{j}^{k}$, $H
\omega $ is monotonically decreasing, from $\infty $ at the left hand
endpoint of $G_{j}^{k}$, to $-\infty $ at the right hand endpoint. In
particular, $H \omega $ has a unique zero $z ^{j} _{k}$. And this selection
of points define $\sigma $ as in \eqref{e.sigmaD}, namely 
\begin{equation*}
{\sigma }=\sum_{k,j}s_{j}^{k}\delta _{{z}^j_k}.
\end{equation*}%
Of course, we gain a substantial cancellation in the testing condition %
\eqref{e.H2} by locating the point mass at the zero of $H \omega $.

We then define the third measure $\ddot{\sigma}$ by taking $\ddot{z}%
_{j}^{k}\in G_{k}^{j}$ to the unique point so that $H\omega (\ddot{z}%
_{j}^{k})=(3/2)^{k}$. We then can easily check that the $L^{2}$ inequality
for $(\omega ,\sigma )$ does not hold: 
\begin{equation*}
\int \lvert H\omega \rvert ^{2}\;d\ddot{\sigma}(x)=\sum_{k=1}^{\infty
}\sum_{j=1}^{2^{k}}\left( \frac{9}{4}\cdot \frac{2}{9}\right) ^{k}=\infty \,.
\end{equation*}%
The weight pair $(\omega ,\ddot{\sigma})$ can be seen to satisfy the $A_{2}$
condition, the forward testing condition \eqref{e.H1}, but fail the
backwards testing condition. Thus, this pair of weights provides an
alternate example to those provided in \cite{NaVo} and \cite{NiTr}. We will
not further discuss the measure $\ddot{\sigma}$.

We can calculate the rate at which $H \omega $ blows up at the endpoints of
the complementary intervals. The rate is a reflection of the fractal
dimension of the Cantor set.


\begin{lemma}
\label{l.gkj} Write $G_{j}^{k}=(a_{j}^{k},b_{j}^{k})$. We have 
\begin{equation}
H\omega (a_{j}^{k}-c3^{-k})\simeq (3/2)^{k}\,,\ \ \ \ \ k\geq 1\,,\ 1\leq
j\leq 2^{k}\,,  \label{e.rapid}
\end{equation}%
and a similar equality holds for $b_{j}^{k}$. (Implied constants can be
taken absolute; signs will be reversed for $b_{j}^{k}$.)
\end{lemma}


This in particular shows that the zeros $z ^{k} _{j}$ cannot move too far
from the middle: 
\begin{equation}  \label{e.closeToMiddle}
\sup _{j,k} \frac {\lvert z ^{k} _{j} - \dot z ^{k} _{j}\rvert } {\lvert G
^{k} _{j}\rvert } < \zeta < 1 \,.
\end{equation}
The points $\ddot z ^{k} _{j}$ satisfy a similar inequality. This indicates
the sensitivity of the two weight inequality to the precise definition of
the measures involved.


\begin{proof}
Fix $k$, and consider the numbers $H \omega ( a ^{k} _{j} - c3 ^{-k})$ for $%
1\le j \le 2 ^{k}$. These are monotonically increasing as the point of
evaluation moves from left to right across the interval $[0,1]$. So we
should verify that 
\begin{equation}  \label{e..H<}
C_1 (3/2) ^{k} \le H \omega ( a ^{k} _{1} + c3 ^{-k}) \le H \omega ( a ^{k}
_{2 ^{k}} + c 3 ^{-k}) \le C_2 (3/2) ^{k}
\end{equation}

We consider the right hand inequality. Writing 
\begin{align*}
H\omega (a_{2^{k}}^{k}+c3^{-k})& =\int_{(G_{2^{k}}^{k})^{c}}\frac{\omega (dy)%
}{a_{2^{k}}^{k}+c3^{-k}-y} \\
& \leq \int_{0}^{a_{2^{k}}^{k}}\frac{\omega (dy)}{a_{2^{k}}^{k}+c3^{-k}-y}
\end{align*}%
Here, we have discarded that part of the domain of the integral where the
integrand would be negative. Now, on the interval $[0,a_{2^{k}}^{k}]$, the
support of $\omega $ is contained in the set $\bigcup_{\ell
=1}^{k}I_{2^{\ell }-1}^{\ell }$. Using this, we continue the estimate above
as 
\begin{align*}
H\omega (a_{2^{k}}^{k}+c3^{-k})& \leq \sum_{\ell =1}^{k}\omega (I_{2^{\ell
}-1}^{\ell })\sup_{y\in I_{2^{\ell }-1}^{\ell }}\frac{1}{a_{1}^{k}+c3^{-k}-y}
\\
& \lesssim c^{-1}\frac{2^{-k}}{3^{-k}}\sum_{\ell =1}^{k-1}\frac{2^{-\ell }}{%
3^{-\ell }}\lesssim c^{-1}(3/2)^{k}\,.
\end{align*}%
It is useful to note for use below, that in this sum, the summand associated
with $k=\ell $ is the dominant one.

\smallskip We consider the left hand inequality in \eqref{e..H<}. We split
the support of $\omega $ into the sets $I_{1}^{k}\,,\
I_{2}^{k}\,I_{2}^{k-1},\dotsc ,I_{2}^{1}$. By the argument above, we have 
\begin{equation*}
\left\vert \sum_{\ell =1}^{k-1}H(\omega \mathbf{1}_{I_{2}^{\ell
}})(a_{1}^{k}+c3^{-k})\right\vert \leq A(3/2)^{k}\,,
\end{equation*}%
where $A$ is absolute, and we have yet to select $c$. Then, we have 
\begin{align*}
H(\omega \mathbf{1}_{I_{1}^{k}\cup I_{2}^{k}})& =\int_{I_{1}^{k}}\frac{1}{%
a_{j}^{k}-c3^{-k}-y}-\frac{1}{a_{j}^{k}-(1+c)3^{-k}-y}\omega (dy) \\
& \gtrsim c^{-1}3^{k}{\omega (I_{1}^{k})}
\end{align*}%
For $0<c<(2A)^{-1}$, we conclude our Lemma.
\end{proof}



\subsection{The $A_2$ Condition}

We verify that the usual $A_{2}$ condition holds for the pair $\left( \omega
,\sigma \right) $. Due to the property \eqref{e.closeToMiddle}, this same
argument will apply to the measures $\dot{\sigma}$ and $\ddot{\sigma}$. The
starting point is the estimate 
\begin{equation}
\sigma (I_{r}^{\ell })=\sum_{\left( k,j\right) :z_{j}^{k}\in I_{r}^{\ell
}}s_{j}^{k}=\sum_{k=\ell }^{\infty }2^{k-\ell }\left( \frac{1}{3}\right)
^{k}\left( \frac{2}{3}\right) ^{k}=2^{-\ell }\sum_{k=\ell }^{\infty }\left( 
\frac{2}{3}\right) ^{2k}\approx 2^{-\ell }\left( \frac{2}{3}\right) ^{2\ell
}=s_{r}^{\ell }\,.  \label{sigma measure}
\end{equation}%
From this, it follows that we have 
\begin{equation}
\frac{\sigma (I_{j}^{k})\omega (I_{j}^{k})}{|I_{j}^{k}|^{2}}\approx \frac{%
s_{j}^{k}\omega (I_{j}^{k})}{|I_{j}^{k}|^{2}}=1  \label{e.sA2}
\end{equation}%
The analogous condition for the Poisson or strengthened $A_{2}$ condition
also holds. Indeed, using the uniformity of $\omega $, one can verify 
\begin{eqnarray*}
\mathsf{P}\left( I_{r}^{\ell },\omega \right) &\lesssim &\frac{\omega
(I_{r}^{\ell })}{\left\vert I_{r}^{\ell }\right\vert }, \\
\mathsf{P}\left( I_{r}^{\ell },\sigma \right) &\lesssim &\sum_{m=0}^{\infty }%
\frac{1}{2^{m}}\frac{\omega (I^{\ell +m})_{\sigma }}{\omega (I^{\ell +m})}%
\leq \sum_{m=0}^{\infty }\frac{1}{2^{m}}\frac{2^{-\left( \ell +m\right)
}\left( \frac{2}{3}\right) ^{2\left( \ell +m\right) }}{\left( \frac{1}{3}%
\right) ^{\ell +m}}\lesssim \left( \frac{2}{3}\right) ^{2\ell }\lesssim 
\frac{\sigma (I_{r}^{\ell })} {\left\vert I_{r}^{\ell }\right\vert }.
\end{eqnarray*}%
From this and \eqref{e.sA2}, we see that 
\begin{equation*}
\mathsf{P}\left( I_{r}^{\ell },\omega \right) \mathsf{P}\left( I_{r}^{\ell
},\sigma \right) \lesssim 1\,.
\end{equation*}

Let us consider an interval $I\subset \lbrack 0,1]$, and let $k$ be the
smallest integer such that $z_{j}^{k}\in AI$. Here $A>1$ is a large
constant, dependent upon the constant in \eqref{e.closeToMiddle}. We note
that $j$ is unique. For $j<j^{\prime }$, it follows that for some $j^{\prime
\prime }$ we have $z_{j^{\prime }}^{k}<z_{j^{\prime \prime
}}^{k-1}<z_{j^{\prime }}^{k}$. In particular, we will have $\sigma(
AI)\simeq \sigma( G_{j}^{k})$. Let us also assume that $G_{j}^{k}\subset AI$%
. Let $I_{r}^{k-1}\supset G_{j}^{k}$. It follows that we have%
\begin{equation}
\mathsf{P}\left( I,\omega \right) \mathsf{P}\left( I,\sigma \right) \lesssim 
\mathsf{P}\left( I_{r}^{k-1},\omega \right) \mathsf{P}\left(
I_{r}^{k-1},\sigma \right) \lesssim \frac{\sigma( I)}{\lvert I\rvert }\frac{%
\omega( I)}{\lvert I\rvert }\lesssim 1\,.  \label{e.toUse}
\end{equation}

The last case is $G_{j}^{k}\supsetneqq AI$. We then have 
\begin{equation*}
\mathsf{P}\left( I,\sigma \right) \simeq \frac{s_{j}^{k}\lvert I\rvert }{%
(\lvert I\rvert +\textup{dist}(z_{j}^{k},I))^{2}}\simeq \frac{s_{j}^{k}}{%
\lvert I\rvert }\,.
\end{equation*}%
The last inequality follows from the definition of $z_{j}^{k}$, and fact
that we must have $\textup{dist}(I,\partial G_{j}^{k})>\lvert I\rvert $,
provided $A$ is sufficiently large. We then have $\mathsf{P}\left( I,\omega
\right) \lesssim 2^{-k}\frac{\lvert I\rvert }{\lvert G_{j}^{k}\rvert }$. And
so we can estimate 
\begin{align*}
\mathsf{P}\left( I,\sigma \right) \mathsf{P}\left( I,\omega \right) &
\lesssim \frac{s_{j}^{k}}{\lvert I\rvert }2^{-k}\frac{\lvert I\rvert }{%
\lvert G_{j}^{k}\rvert } \\
& \lesssim \frac{2^{-k}s_{j}^{k}}{\lvert G_{j}^{k}\rvert }\lesssim 1\,.
\end{align*}


\subsection{The Pivotal and  Hybrid  Conditions}  

In this section, we show that the weight pair $(\omega ,\sigma )$ fails the
dual Pivotal Condition, namely the Hybrid Condition with $\epsilon =0$ and
the roles of $\omega $ and $\sigma $ reversed. But, they satisfy the Hybrid
Condition for all $0<\epsilon \leq 2$, and the dual Hybrid Condition for $%
\epsilon _{0}\leq \varepsilon \leq 2$ for some $\epsilon _{0}<2$.


\subsubsection{Failure of the Pivotal Condition for $\protect\epsilon =0$}

Failure of the Pivotal Condition is straight forward. Indeed, $I_{1}^{\ell
}\subset I_{1}^{\ell -1}\subset ...\subset I_{1}^{0}$ and so 
\begin{equation*}
\mathsf{P}\left( G_{1}^{\ell },\omega \right) \approx \mathsf{P}\left(
I_{1}^{\ell },\omega \right) \approx \sum_{k=0}^{\ell }\frac{\left\vert
I_{1}^{\ell }\right\vert }{\left\vert I_{1}^{k}\right\vert ^{2}}\omega(
I_{r}^{k})\approx \sum_{k=0}^{\ell }\frac{3^{-\ell }}{3^{-2k}}2^{-k}\approx
\left( \frac{3}{2}\right) ^{\ell },
\end{equation*}%
and similarly%
\begin{equation*}
\mathsf{P}\left( G_{r}^{\ell },\omega \right) \approx \mathsf{P}\left(
I_{r}^{\ell },\omega \right) \approx \left( \frac{3}{2}\right) ^{\ell },\ \
\ \ \ \text{all }r.
\end{equation*}%
Considering the decomposition $\overset{\cdot }{\dot\bigcup_{\ell ,r}}%
G_{r}^{\ell }\subset \left[ 0,1\right] $ we thus have%
\begin{equation*}
\sum_{\ell ,r}\left\vert G_{r}^{\ell }\right\vert _{\sigma }\mathsf{P}\left(
G_{r}^{\ell },\omega \right) ^{2}\approx \sum_{\ell =0}^{\infty }2^{\ell }
\left( \frac{1}{3}\right) ^{\ell }\left( \frac{2}{3}\right) ^{\ell } \left( 
\frac{3}{2}\right) ^{2\ell }\mathbb{\approx }\sum_{\ell =0}^{\infty
}1=\infty ,
\end{equation*}%
which shows that the dual Pivotal Condition, the one dual to  \eqref{pivotalcondition}, fails.


\subsubsection{The dual Hybrid Condition for large $\protect\epsilon $}

Next we show that the dual Hybrid Condition%
\begin{equation}
\sum_{r=1}^{\infty }\sigma (I_{r})\mathsf{E}(I_{r},\sigma )^{\epsilon }%
\mathsf{P}(I_{r},\mathbf{1} _{I_{0}}\omega ) ^2 \leq \left( \mathcal{E}%
_{\epsilon }^{\ast }\right) ^{2}\omega (I_{0}),  \label{dual energy}
\end{equation}%
holds for all $\epsilon _{0}\leq \epsilon \leq 2$ where

\begin{equation*}
\epsilon _{0}=\frac{1}{\frac{\ln 3}{\ln 2}-\frac{1}{2}}\approx 0.92<2.
\end{equation*}%
We need this estimate, which shows that with Energy, we can get an essential strengthening of the 
$ A_2$ condition. 


\begin{proposition}
\label{p.gettingE}For $\epsilon \geq \epsilon _{0}$ and any interval $%
I\subset \lbrack 0,1]$, we have the inequality 
\begin{equation}
\sigma( I)\mathsf{E}(I;\sigma )^{\epsilon }\mathsf{P}(I;\omega )^{2}\lesssim
\omega( I).  \label{e.gettingE}
\end{equation}
\end{proposition}



\begin{proof}
We can assume that $\mathsf{E}(I;\sigma )\neq 0$. Let $k$ be the smallest
integer for which there is a $r$ with $z_{r}^{k}\in I$. And let $n$ be the
smallest integer so that for some $s$ we have $z_{s}^{k+n}\in I$ and $%
z_{s}^{k+n}\neq z_{r}^{k}$. We can estimate $\mathsf{E}(I;\sigma )$ in terms
of $n$. Namely, we have 
\begin{equation}
\mathsf{E}(I;\sigma )^{2}\lesssim \left( \frac{2}{9}\right) ^{n}\,.
\label{e.En}
\end{equation}%
Indeed, the worst case is when $s$ is not unique. Then, there are two
choices of $s$ -- but not more. Let $z_{s^{\prime }}^{k+n}\in I$, where $%
s^{\prime }\neq s$. Then, note that we have 
\begin{equation*}
\frac{\lvert I-\{z_{r}^{k}\}\rvert _{\sigma }}{\sigma( I)}\simeq \left( 
\frac{2}{9}\right) ^{n}.
\end{equation*}%
and this leads to the estimate above, remembering the characterization of
Energy as a variance term.

Next we note that $\sigma( I)\approx \left( \frac{2}{9}\right) ^{k}$, $%
\omega( I)\geq 2^{-k-n}$, and $\mathsf{P}(I;\omega )\simeq \left( \frac{3}{2}%
\right) ^{k}$. This specifies everything in \eqref{e.gettingE}, so we choose 
$\epsilon $ so that 
\begin{equation*}
\left( \frac{2}{9}\right) ^{k}\left( \frac{2}{9}\right) ^{\frac{\epsilon }{2}%
n}\left( \frac{3}{2}\right) ^{2k}\lesssim 2^{-k-n}\,.
\end{equation*}%
This inequality will be true for all pairs of $n,k$ if $\epsilon _{0}\leq
\epsilon <2$ where%
\begin{equation*}
\left( \frac{2}{9}\right) ^{\frac{\epsilon _{0}}{2}}=\frac{1}{2}.
\end{equation*}
\end{proof}


It is now clear that the pair of weights $(\omega ,\sigma )$ satisfy the
dual Energy conditions $\mathcal{E}_{\epsilon }^{\ast }$ for $\epsilon
_{0}\leq \epsilon \leq 2$. Let $I_{0}\subset \lbrack 0,1]$ and let $%
\{I_{r}\;:\;r\geq 1\}$ be any partition of $I_{0}$. We appeal to %
\eqref{e.gettingE} to see that 
\begin{equation*}
\sum_{r\geq 1}\sigma( I_{r})\mathsf{E}(I;\sigma )^{\epsilon }\mathsf{P}%
(I;\omega )^{2}\lesssim \sum_{r\geq 1} \omega (I_{r}) =\omega( I_{0})\,.
\end{equation*}

\subsubsection{The Hybrid Condition for positive $\protect\epsilon $}

It remains to verify that the pair of measures $(\omega ,\sigma )$ satisfy
the Hybrid Conditions for all $0\leq \epsilon \leq 2$. We will establish the
pivotal condition \eqref{pivotalcondition}, i.e. $\mathcal{E}_{0}<\infty $,
which then implies that $\mathcal{E}_{\epsilon }<\infty $ for all $0\leq
\epsilon \leq 2$. For this it suffices to show that the forward maximal
inequality%
\begin{equation}
\int M\left( f\sigma \right) ^{2}d\omega \leq C\int \left\vert f\right\vert
^{2}d\sigma  \label{M2weight}
\end{equation}%
holds for the pair $\left( \omega ,\sigma \right) $, and \eqref{M2weight} in
turn follows from the testing condition%
\begin{equation}
\int M\left( \mathbf{1} _{Q}\sigma \right) ^{2}d\omega \leq C\int_{Q}d\sigma
,  \label{testmax}
\end{equation}%
for all intervals $Q$ (see \cite{MR676801}). We will show \eqref{testmax}
when $Q=I_{r}^{\ell }$, the remaining cases being an easy consequence of
this one. For this we use the fact that%
\begin{equation}
\mathcal{M}\left( \mathbf{1} _{I_{r}^{\ell }}\sigma \right) \left( x\right)
\leq C\left( \frac{2}{3}\right) ^{\ell },\ \ \ \ \ x\in E\cap I_{r}^{\ell }.
\label{Msigma bounded}
\end{equation}%
To see \eqref{Msigma bounded}, note that for each $x\in I_{r}^{\ell }$ that
also lies in the Cantor set $E$, we have%
\begin{equation*}
\mathcal{M}\left( \mathbf{1} _{I_{r}^{\ell }}\sigma \right) \left( x\right)
\leq \sup_{\left( k,j\right) :x\in I_{j}^{k}}\frac{1}{\left\vert
I_{j}^{k}\right\vert }\int_{I_{j}^{k}\cap I_{r}^{\ell }}d\sigma \approx
\sup_{\left( k,j\right) :x\in I_{j}^{k}}\frac{\left( \frac{1}{3}\right)
^{k\vee \ell }\left( \frac{2}{3}\right) ^{k\vee \ell }}{\left( \frac{1}{3}%
\right) ^{k}}\approx \left( \frac{2}{3}\right) ^{\ell }.
\end{equation*}%
Now we consider for each fixed $m$, the approximations $\omega ^{\left(
m\right) }$ and $\sigma ^{\left( m\right) }$ to the measures $\omega $ and $%
\sigma $ given by 
\begin{eqnarray}
d\omega ^{\left( m\right) }\left( x\right) &=&\sum_{i=1}^{2^{m}}2^{-m}\frac{1%
}{\left\vert I_{i}^{m}\right\vert }\mathbf{1} _{I_{i}^{m}}\left( x\right) dx,
\label{approximations} \\
\sigma ^{\left( m\right) } &=&\sum_{k<m}\sum_{j=1}^{2^{k}}s_{j}^{k}\delta
_{z_{j}^{k}}.  \notag
\end{eqnarray}%
For these approximations we have in the same way the estimate%
\begin{equation*}
\mathcal{M}\left( \mathbf{1} _{I_{r}^{\ell }}\sigma ^{\left( m\right)
}\right) \left( x\right) \leq C\left( \frac{2}{3}\right) ^{\ell },\ \ \ \ \
x\in \bigcup_{i=1}^{2^{m}}I_{i}^{m}.
\end{equation*}%
Thus for each $m\geq 1$ we have%
\begin{eqnarray*}
\int_{I_{r}^{\ell }}\mathcal{M}\left( \mathbf{1} _{I_{r}^{\ell }}\sigma
^{\left( m\right) }\right) ^{2}d\omega ^{\left( m\right) } &\leq
&C\sum_{i:I_{i}^{m}\subset I_{r}^{\ell }}\left( \frac{2}{3}\right) ^{2\ell
}2^{-m} \\
&=&C2^{m-\ell }\left( \frac{2}{3}\right) ^{2\ell }2^{-m}=Cs_{r}^{\ell
}\approx C\int_{I_{r}^{\ell }}d\sigma .
\end{eqnarray*}%
Taking the limit as $m\rightarrow \infty $ yields the case $Q=I_{r}^{\ell }$
of \eqref{testmax}. This completes our proof of the pivotal condition, and
hence also the Hybrid Conditions \eqref{Hy Con} for all $0\leq \epsilon \leq
2$.


\subsection{The Testing Conditions}

As an initial step in verifying the forward testing condition \eqref{e.H1}
for the pair $(\omega ,\sigma )$, we replace $\sigma $ by the self-similar
measure $\dot \sigma $, and exploit the self-similarity of both measures $%
\omega $ and $\dot \sigma $: 
\begin{eqnarray}
\omega &=&\frac{1}{2}\operatorname{Dil}_{\frac{1}{3}}\omega +\frac{1}{2}\operatorname{Trans}_{\frac{2}{3}%
}\operatorname{Dil}_{\frac{1}{3}}\omega \equiv \omega _{1}+\omega _{2},
\label{self-similar} \\
\dot\sigma &=&\frac{2}{9}\operatorname{Dil}_{\frac{1}{3}}\dot\sigma +\delta _{\frac{1}{2}}+%
\frac{2}{9}\operatorname{Trans}_{\frac{2}{3}}\operatorname{Dil}_{\frac{1}{3}}\dot\sigma \equiv \dot\sigma
_{1}+\delta _{\frac{1}{2}}+\dot\sigma _{2}.
\end{eqnarray}%
We compute%
\begin{eqnarray*}
\int \left\vert H\dot\sigma \right\vert ^{2}\omega &=&\int \left\vert
H\left( \dot\sigma _{1}+\delta _{\frac{1}{2}}+\dot\sigma _{2}\right)
\right\vert ^{2}\omega _{1}+\int \left\vert H\left( \dot\sigma _{1}+\delta _{%
\frac{1}{2}}+\dot\sigma _{2}\right) \right\vert ^{2}\omega _{2} \\
&=&\left( 1+\varepsilon \right) \left\{ \int \left\vert H\dot\sigma
_{1}\right\vert ^{2}\omega _{1}+\int \left\vert H\dot\sigma _{2}\right\vert
^{2}\omega _{2}\right\} +\mathcal{R}_{\varepsilon },
\end{eqnarray*}%
where the remainder term $\mathcal{R}_{\varepsilon }$ is easily seen to
satisfy 
\begin{equation*}
\mathcal{R}_{\varepsilon }\lesssim _{\varepsilon }\mathcal{A}_{2}^{2}\left(
\int \dot\sigma \right) ,
\end{equation*}%
since the supports of $\delta _{\frac{1}{2}}+\dot\sigma _{2}$ and $\omega
_{1}$ are well separated, as are those of $\delta _{\frac{1}{2}}+\dot\sigma
_{1}$ and $\omega _{2}$. For this we first use $\left( a+b\right) ^{2}\leq
\left( 1+\varepsilon \right) a^{2}+\left( 1+\frac{1}{\varepsilon }\right)
b^{2}$ to obtain%
\begin{eqnarray*}
&&\int \left\vert H\left( \dot\sigma _{1}+\delta _{\frac{1}{2}}+\dot\sigma
_{2}\right) \right\vert ^{2}\omega _{1} \\
&\lesssim &\int \left( \left\vert H\left( \dot\sigma _{1}\right) \right\vert
+\left\vert H\left( \delta _{\frac{1}{2}}+\dot\sigma _{2}\right) \right\vert
\right) ^{2}\omega _{1} \\
&\lesssim &\int \left\{ \left( 1+\varepsilon \right) \left\vert H\left(
\dot\sigma _{1}\right) \right\vert ^{2}+\left( 1+\frac{1}{\varepsilon }%
\right) \left\vert H\left( \delta _{\frac{1}{2}}+\dot\sigma _{2}\right)
\right\vert ^{2}\right\} \omega _{1},
\end{eqnarray*}%
and then for example,%
\begin{eqnarray*}
\int \left\vert H\left( \dot\sigma _{2}\right) \right\vert ^{2}\omega _{1}
&=&\int_{\left[ 0,\frac{1}{3}\right] }\left\vert \int_{\left[ \frac{2}{3},1%
\right] }\frac{1}{y-x}\dot\sigma \left( y\right) \right\vert ^{2}\omega
\left( x\right) \\
&\lesssim &\left[ \frac{1}{\frac{1}{3}\left\vert \left[ 0,1\right]
\right\vert }\right] ^{2} \dot \sigma ([ \tfrac{2}{3},1])^{2} \omega ([ 0,%
\tfrac{1}{3}]) \\
&\lesssim &9\frac{\dot \sigma ([ 0,1]) \omega ([0,1] }{\vert [ 0,1 ] \vert
^{2}} \int \dot\sigma _{2} \lesssim \mathcal{A}_{2}^{2} \int \dot\sigma .
\end{eqnarray*}

But now we note that%
\begin{eqnarray*}
\int \left\vert H\dot\sigma _{1}\right\vert ^{2}\omega _{1} &=&\frac{1}{2}%
\int \left\vert H\dot\sigma _{1}\left( x\right) \right\vert ^{2}\operatorname{Dil}_{\frac{1%
}{3}}\omega \left( x\right) =\frac{1}{2}\int \left\vert H\dot\sigma
_{1}\left( \frac{x}{3}\right) \right\vert ^{2}\omega \left( x\right) \\
&=&\frac{1}{2}\int \left\vert \int \frac{1}{z-\frac{x}{3}}\frac{2}{9}\operatorname{Dil}_{%
\frac{1}{3}}\dot\sigma \left( z\right) \right\vert ^{2}\omega \left(
x\right) =\frac{1}{2}\left( \frac{2}{9}\right) ^{2}\int \left\vert \int 
\frac{1}{\frac{z}{3}-\frac{x}{3}}\dot\sigma \left( z\right) \right\vert
^{2}\omega \left( x\right) \\
&=&\frac{1}{2}\left( \frac{2}{9}\right) ^{2}9\int \left\vert H\dot\sigma
\left( x\right) \right\vert ^{2}\omega \left( x\right) =\frac{2}{9}\int
\left\vert H\dot\sigma \right\vert ^{2}\omega ,
\end{eqnarray*}%
and similarly $\int \left\vert H\dot\sigma _{2}\right\vert ^{2}\omega _{2}=%
\frac{2}{9}\int \left\vert H\dot\sigma \right\vert ^{2}\omega $. Thus we have

\begin{equation}
\int \left\vert H\dot\sigma \right\vert ^{2}\omega =\frac{2}{9}\left(
1+\varepsilon \right) \int \left\vert H\dot\sigma \right\vert ^{2}\omega +%
\frac{2}{9}\left( 1+\varepsilon \right) \int \left\vert H\dot\sigma
\right\vert ^{2}\omega +\mathcal{R}_{\varepsilon },  \label{reproduce}
\end{equation}%
and provided $\int \left\vert H\dot\sigma \right\vert ^{2}\omega $ is finite
we conclude that%
\begin{equation*}
\int \left\vert H\dot\sigma \right\vert ^{2}\omega =\frac{1}{1-\frac{4}{9}%
\left( 1+\varepsilon \right) }\mathcal{R}_{\varepsilon }\lesssim
_{\varepsilon }\mathcal{A}_{2}^{2}\left( \int \dot\sigma \right) ,
\end{equation*}%
for $\varepsilon >0$ so small that $1-\frac{4}{9}\left( 1+\varepsilon
\right) >0$.

To avoid making the assumption that $\int \left\vert H\dot\sigma \right\vert
^{2}\omega $ is finite, we use approximations as follows. For each fixed $%
m\geq 1$, consider the approximations $\omega ^{\left( m\right) }$ and $%
\dot\sigma ^{\left( m\right) }$ to the measures $\omega $ and $\dot\sigma $
as in \eqref{approximations}. We have the following self-similarity
equations involving $\omega ^{\left( m\right) }$ and $\dot\sigma ^{\left(
m\right) }$ that substitute for \eqref{self-similar}: for $m\geq 2$,%
\begin{eqnarray*}
\omega ^{\left( m\right) } &=&\frac{1}{2}\operatorname{Dil}_{\frac{1}{3}}\omega ^{\left(
m-1\right) }+\frac{1}{2}\operatorname{Trans}_{\frac{2}{3}}\operatorname{Dil}_{\frac{1}{3}}\omega ^{\left(
m-1\right) }\equiv \omega _{1}^{\left( m\right) }+\omega _{2}^{\left(
m\right) }, \\
\dot\sigma ^{\left( m\right) } &=&\frac{2}{9}\operatorname{Dil}_{\frac{1}{3}}\dot\sigma
^{\left( m-1\right) }+\delta _{\frac{1}{2}}+\frac{2}{9}\operatorname{Trans}_{\frac{2}{3}%
}\operatorname{Dil}_{\frac{1}{3}}\dot\sigma ^{\left( m-1\right) }\equiv \dot\sigma
_{1}^{\left( m\right) }+\delta _{\frac{1}{2}}+\dot\sigma _{2}^{\left(
m\right) }.
\end{eqnarray*}%
As above we compute that%
\begin{eqnarray*}
\int \left\vert H\dot\sigma ^{\left( m\right) }\right\vert ^{2}\omega
^{\left( m\right) } &=&\int \left\vert H\left( \dot\sigma _{1}^{\left(
m\right) }+\delta _{\frac{1}{2}}+\dot\sigma _{2}^{\left( m\right) }\right)
\right\vert ^{2}\omega _{1}^{\left( m\right) }+\int \left\vert H\left(
\dot\sigma _{1}^{\left( m\right) }+\delta _{\frac{1}{2}}+\dot\sigma
_{2}^{\left( m\right) }\right) \right\vert ^{2}\omega _{2}^{\left( m\right) }
\\
&=&\left( 1+\varepsilon \right) \left\{ \int \left\vert H\dot\sigma
_{1}^{\left( m\right) }\right\vert ^{2}\omega _{1}^{\left( m\right) }+\int
\left\vert H\dot\sigma _{2}^{\left( m\right) }\right\vert ^{2}\omega
_{2}^{\left( m\right) }\right\} +\mathcal{R}_{\varepsilon }^{\left( m\right)
},
\end{eqnarray*}%
where the remainder term $\mathcal{R}_{\varepsilon }^{\left( m\right) }$
satisfies $\mathcal{R}_{\varepsilon }^{\left( m\right) }\lesssim
_{\varepsilon }\mathcal{A}_{2}^{2}\left( \int \dot\sigma \right) $, and also
that%
\begin{eqnarray*}
\int \left\vert H\dot\sigma _{1}^{\left( m\right) }\right\vert ^{2}\omega
_{1}^{\left( m\right) } &=&\frac{1}{2}\int \left\vert H\dot\sigma
_{1}^{\left( m\right) }\right\vert ^{2}\operatorname{Dil}_{\frac{1}{3}}\omega ^{\left(
m-1\right) }\left( x\right) =\frac{1}{2}\int \left\vert H\dot\sigma
_{1}^{\left( m\right) }\left( \frac{x}{3}\right) \right\vert ^{2}\omega
^{\left( m-1\right) }\left( x\right) \\
&=&\frac{1}{2}\int \left\vert \int \frac{1}{z-\frac{x}{3}}\frac{2}{9}\operatorname{Dil}_{%
\frac{1}{3}}\dot\sigma ^{\left( m-1\right) }\left( z\right) \right\vert
^{2}\omega ^{\left( m-1\right) }\left( x\right) \\
&=&\frac{1}{2}\left( \frac{2}{9}\right) ^{2}\int \left\vert \int \frac{1}{%
\frac{z}{3}-\frac{x}{3}}\dot\sigma ^{\left( m-1\right) }\left( z\right)
\right\vert ^{2}\omega ^{\left( m-1\right) }\left( x\right) \\
&=&\frac{1}{2}\left( \frac{2}{9}\right) ^{2}9\int \left\vert H\dot\sigma
^{\left( m-1\right) }\left( x\right) \right\vert ^{2}\omega ^{\left(
m-1\right) }\left( x\right) =\frac{2}{9}\int \left\vert H\dot\sigma ^{\left(
m-1\right) }\right\vert ^{2}\omega ^{\left( m-1\right) },
\end{eqnarray*}%
and $\int \left\vert H\dot\sigma _{2}^{\left( m\right) }\right\vert
^{2}\omega _{2}^{\left( m\right) }=\frac{2}{9}\int \left\vert H\dot\sigma
^{\left( m-1\right) }\right\vert ^{2}\omega ^{\left( m-1\right) }$. Thus we
have%
\begin{equation*}
\int \left\vert H\dot\sigma ^{\left( m\right) }\right\vert ^{2}\omega
^{\left( m\right) }=\frac{4}{9}\left( 1+\varepsilon \right) \int \left\vert
H\dot\sigma ^{\left( m-1\right) }\right\vert ^{2}\omega ^{\left( m-1\right)
}+\mathcal{R}_{\varepsilon }^{\left( m\right) },\qquad m\geq 2.
\end{equation*}%
Iterating this equality yields%
\begin{equation*}
\int \left\vert H\dot\sigma ^{\left( m\right) }\right\vert ^{2}\omega
^{\left( m\right) }=\left( \frac{4}{9}\left( 1+\varepsilon \right) \right)
^{m}\int \left\vert H\dot\sigma ^{\left( 0\right) }\right\vert ^{2}\omega
^{\left( 0\right) }+\sum_{j=0}^{m-1}\left( \frac{4}{9}\left( 1+\varepsilon
\right) \right) ^{j}\mathcal{R}_{\varepsilon }^{\left( m-j\right) },\qquad
m\geq 2,
\end{equation*}%
from which we obtain%
\begin{equation*}
\int \left\vert H\dot\sigma ^{\left( m\right) }\right\vert ^{2}\omega
^{\left( m\right) }\lesssim _{\varepsilon }\mathcal{A}_{2}^{2}\left( \int
\dot\sigma \right) ,\qquad m\geq 2,
\end{equation*}%
with a constant $C$ independent of $m$. Taking the limit as $m\rightarrow
\infty $ proves $\int \left\vert H\dot\sigma \right\vert ^{2}\omega \leq
C_{\varepsilon }\mathcal{A}_{2}^{2}\left( \int \dot\sigma \right) <\infty $.

This completes the proof of the forward testing condition \eqref{e.H1} for
the interval $I=\left[ 0,1\right] $. The proof for the case $I=I_{j}^{k}$ is
similar using $\mathcal{R}_{\varepsilon }\left( I_{j}^{k}\right) \leq
C_{\varepsilon }\mathcal{A}_{2}^{2}\left( \int_{I_{j}^{k}}\dot\sigma \right) 
$, and the general case now follows without much extra work.

\begin{remark}
The self-similarity argument above works on the forward testing condition
because the central point mass $\delta _{\frac{1}{2}}$ is a significant
fraction $\frac{5}{9}$ of the mass of $\dot\sigma $ and is well separated
from the measure $\omega $ at all scales. This accounts for the fact that a
mere fraction\ $\frac{4}{9}$ of the left side of \eqref{reproduce} appears
on the right side. This argument fails to apply to the two weight inequality
\eqref{e..H<} itself because the measure $f\dot\sigma $ need not have a
significant proportion of its mass concentrated at the point $\frac{1}{2}$.
\end{remark}

Having verified the forward testing condition for the weight pair $(\omega
,\dot\sigma )$, we show that the forward testing condition \eqref{e.H1}
holds for $\left( \omega ,{\sigma }\right) $. For this, we estimate the
difference%
\begin{eqnarray*}
\int_{I_{r}^{\ell }}\left\vert H\mathbf{1} _{I_{r}^{\ell }}\left( \sigma -%
\dot{\sigma }\right) \right\vert ^{2}\omega &=&\int_{I_{r}^{\ell
}}\left\vert \sum_{\left( k,j\right) :z_{j}^{k}\in I_{r}^{\ell
}}s_{j}^{k}\left( \frac{1}{x-z_{j}^{k}}-\frac{1}{x-\dot{z}_{k}^{j}}\right)
\right\vert ^{2}\omega \left( x\right) \\
&=&\int_{I_{r}^{\ell }}\left\vert \sum_{\left( k,j\right) :z_{j}^{k}\in
I_{r}^{\ell }}s_{j}^{k}\left( \frac{z_{j}^{k}-\dot{z}_{k}^{j}}{\left(
x-z_{j}^{k}\right) \left( x-\dot{z}_{k}^{j}\right) }\right) \right\vert
^{2}\omega \left( x\right) \\
&\lesssim &C\int_{I_{r}^{\ell }}\left\vert \sum_{\left( k,j\right)
:z_{j}^{k}\in I_{r}^{\ell }}s_{j}^{k}\left( \frac{\left\vert
I_{j}^{k}\right\vert }{\left\vert x-z_{j}^{k}\right\vert ^{2}}\right)
\right\vert ^{2}\omega \left( x\right) .
\end{eqnarray*}%
In the last line, we have used \eqref{e.closeToMiddle}. Now for any fixed $x$
in the support of $\omega $ inside $I_{r}^{\ell }$, we have%
\begin{eqnarray*}
\sum_{\left( k,j\right) :z_{j}^{k}\in I_{r}^{\ell }}s_{j}^{k}\left( \frac{%
\left\vert I_{j}^{k}\right\vert }{\left\vert x-z_{j}^{k}\right\vert ^{2}}%
\right) &=&\sum_{m=0}^{\infty }\sum_{\left( k,j\right) :z_{j}^{k}\in
I_{r}^{\ell }\text{ and }\left\vert x-z_{j}^{k}\right\vert \approx
3^{-m}\left\vert I_{r}^{\ell }\right\vert }s_{j}^{k}\left( \frac{\left\vert
I_{j}^{k}\right\vert }{\left\vert x-z_{j}^{k}\right\vert ^{2}}\right) \\
&\lesssim &\sum_{m=0}^{\infty }\sum_{\left( k,j\right) :z_{j}^{k}\in
I_{r}^{\ell }\text{ and }\left\vert x-z_{j}^{k}\right\vert \approx
3^{-m}\left\vert I_{r}^{\ell }\right\vert }\left( \frac{1}{3}\right)
^{k}\left( \frac{2}{3}\right) ^{k}\left( \frac{3^{-k}}{\left( 3^{-m-\ell
}\right) ^{2}}\right) \\
&\approx &\sum_{m=0}^{\infty }\sum_{k\geq \ell +m}2^{k-\ell -m}\left( \frac{1%
}{3}\right) ^{k}\left( \frac{2}{3}\right) ^{k}\left( \frac{3^{-k}}{\left(
3^{-m-\ell }\right) ^{2}}\right) \\
&=&\sum_{m=0}^{\infty }\sum_{k\geq \ell +m}\left( \frac{4}{27}\right)
^{k}\left( \frac{9}{2}\right) ^{\left( m+\ell \right) } \\
&=&\sum_{m=0}^{\infty }\left( \frac{2}{3}\right) ^{m+\ell }\approx \left( 
\frac{2}{3}\right) ^{\ell }.
\end{eqnarray*}%
Thus we get%
\begin{equation*}
\int_{I_{r}^{\ell }}\left\vert H\mathbf{1} _{I_{r}^{\ell }}\left( \sigma -%
\dot{\sigma }\right) \right\vert ^{2}\omega \lesssim \left( \frac{2}{3}%
\right) ^{2\ell } \omega ( I_{r}^{\ell }) =C^{2}\left( \frac{2}{3}\right)
^{2\ell }2^{-\ell }\approx \sigma ( I_{r}^{\ell }),
\end{equation*}%
which yields%
\begin{eqnarray*}
\left( \int_{I_{r}^{\ell }}\left\vert H\mathbf{1} _{I_{r}^{\ell }}{\sigma }%
\right\vert ^{2}\omega \right) ^{\frac{1}{2}} &\lesssim &\left(
\int_{I_{r}^{\ell }}\left\vert H\mathbf{1} _{I_{r}^{\ell }}\dot\sigma
\right\vert ^{2}\omega \right) ^{\frac{1}{2}}+\left( \int_{I_{r}^{\ell
}}\left\vert H\mathbf{1} _{I_{r}^{\ell }}\left( \sigma -\dot{\sigma }\right)
\right\vert ^{2}\omega \right) ^{\frac{1}{2}} \\
&\lesssim &C\sqrt{ \sigma ( I_{r}^{\ell })}
\end{eqnarray*}%
This is the case $I=I_{r}^{\ell }$ of the forward testing condition (\ref%
{e.H1}) for the weight pair $\left( \omega ,\dot{\sigma }\right) $, and the
general case follows easily from this.

\bigskip

Finally, we turn to the dual testing condition \eqref{e.H2} for the weight
pair $(\omega , \sigma )$. For interval $I ^{\ell } _{r}$ and $z ^{k} _{j}
\in I ^{\ell } _{r}$, we claim that 
\begin{equation}
\left\vert H\left( \mathbf{1} _{I_{r}^{\ell }}\omega \right) \left(
z_{j}^{k}\right) \right\vert \lesssim \mathsf{P}\left( I_{r}^{\ell },\omega
\right) .  \label{claim Poisson}
\end{equation}%
To see this let $I_{s}^{\ell -1}$ denote the parent of $I_{r}^{\ell }$ and $%
I_{r+1}^{\ell }$ denote the other child of $I_{s}^{\ell -1}$. Then we have
using $H\omega \left( {z}^j_k\right) =0$, 
\begin{eqnarray*}
H\left( \mathbf{1} _{I_{r}^{\ell }}\omega \right) \left( {z}^j_k\right)
&=&-H\left( \mathbf{1} _{\left( I_{r}^{\ell }\right) ^{c}}\omega \right)
\left( {z}^j_k\right) \\
&=&-H\left( \mathbf{1} _{\left( I_{s}^{\ell -1}\right) ^{c}}\omega \right)
\left( {z}^j_k\right) -H\left( \mathbf{1} _{I_{r+1}^{\ell }}\omega \right)
\left( {z}^j_k\right) .
\end{eqnarray*}%
Now we have using $H\left( \omega \right) \left( {z_{r}^{\ell }}\right) =0$
that%
\begin{eqnarray*}
H\left( \mathbf{1} _{\left( I_{s}^{\ell -1}\right) ^{c}}\omega \right)
\left( {z}^j_k\right) &=&H\left( \mathbf{1} _{\left( I_{s}^{\ell -1}\right)
^{c}}\omega \right) \left( {z_{r}^{\ell }}\right) -\left\{ H\left( \mathbf{1}
_{\left( I_{s}^{\ell -1}\right) ^{c}}\omega \right) \left( {z_{r}^{\ell }}%
\right) -H\left( \mathbf{1} _{\left( I_{s}^{\ell -1}\right) ^{c}}\omega
\right) \left( {z}^j_k\right) \right\} \\
&=&-H\left( \mathbf{1} _{I_{s}^{\ell -1}}\omega \right) \left( {z_{r}^{\ell }%
}\right) -A,
\end{eqnarray*}%
where 
\begin{equation*}
A\equiv H\left( \mathbf{1} _{\left( I_{s}^{\ell -1}\right) ^{c}}\omega
\right) \left( {z_{r}^{\ell }}\right) -H\left( \mathbf{1} _{\left(
I_{s}^{\ell -1}\right) ^{c}}\omega \right) \left( {z}^j_k\right) .
\end{equation*}%
Combining equalities yields 
\begin{equation*}
H\left( \mathbf{1} _{I_{r}^{\ell }}\omega \right) \left( {z}^j_k\right)
=H\left( \mathbf{1} _{I_{s}^{\ell -1}}\omega \right) \left( z_{r}^{\ell
}\right) +A-H\left( \mathbf{1} _{I_{r+1}^{\ell }}\omega \right) \left( {z}%
^j_k\right) .
\end{equation*}%
We then have for $\left( k,j\right) $ such that $z_{j}^{k}\in I_{r}^{\ell }$,%
\begin{eqnarray*}
&&\left\vert H\left( \mathbf{1} _{I_{s}^{\ell -1}}\omega \right) \left(
z_{r}^{\ell }\right) \right\vert \lesssim \frac{\omega ( I_{s}^{\ell -1} )}{%
\left\vert {I_{s}^{\ell -1}}\right\vert }, \\
&&\left\vert A\right\vert \lesssim \int_{\left( I_{s}^{\ell -1}\right)
^{c}}\left\vert \frac{1}{x- z_{r}^{\ell }}-\frac{1}{x- z_{j}^{k}}\right\vert
\omega \left( x\right) \lesssim \int_{\left( I_{s}^{\ell -1}\right) ^{c}}%
\frac{\left\vert I_{r}^{\ell }\right\vert }{\left\vert x- z_{s}^{\ell -1}
\right\vert ^{2}}\omega \left( x\right) , \\
&&\left\vert H\left( \mathbf{1} _{I_{r+1}^{\ell }}\omega \right) \left(
z_{j}^{k}\right) \right\vert \lesssim \frac{\omega ( I_{s}^{\ell -1} )} {%
\left\vert {I_{s}^{\ell -1}}\right\vert },
\end{eqnarray*}%
which proves \eqref{claim Poisson}.

Now we compute using \eqref{claim Poisson} and the estimate for $\mathsf{P}%
\left( I_{r}^{\ell },\omega \right) $ above that%
\begin{eqnarray*}
\int_{I_{r}^{\ell }}\left\vert H\left( \mathbf{1} _{I_{r}^{\ell }}\omega
\right) \right\vert ^{2}d{\sigma } &=&\sum_{\left( k,j\right) :z_{j}^{k}\in
I_{r}^{\ell }}\left\vert H\left( \mathbf{1} _{I_{r}^{\ell }}\omega \right)
\left( {z}_{k}^{j}\right) \right\vert ^{2}s_{j}^{k}\leq C\sum_{\left(
k,j\right) :z_{j}^{k}\in I_{r}^{\ell }}\left\vert \mathsf{P}\left(
I_{r}^{\ell },\omega \right) \right\vert ^{2}s_{j}^{k} \\
&\lesssim &\sigma ( I_{r}^{\ell})\left( \frac{\omega ( I_{r}^{\ell})} {%
\left\vert I_{r}^{\ell }\right\vert }\right) ^{2}\lesssim \mathcal{A}%
_{2}\omega ( I_{r}^{\ell}).
\end{eqnarray*}%
This is the case $I=I_{r}^{\ell }$ of the dual testing condition (\ref{e.H2}%
) for the weight pair $\left( \omega ,{\sigma }\right) $, and the general
case follows easily from this.

\def\baselinestretch{1}

\end{document}